\documentclass[twoside,reqno]{amsart}

\usepackage{bbm}
\usepackage{mathrsfs}
\newtheorem{Lemma}{Lemma}[section]
\newtheorem{Theorem}{Theorem}[section]
\newtheorem{Proposition}{Proposition}[section]

\usepackage{amsmath,amssymb,amsthm,amsfonts}

\usepackage{color}

\numberwithin{equation}{section}

\newcommand{\dis}{\displaystyle}

\newcommand{\R}{\mathbb{R}}

\newcommand{\FP}{\mathbf{P}}

\newcommand{\FI}{\mathbf{I}}

\newcommand{\CD}{\mathcal{D}}
\newcommand{\CE}{\mathcal{E}}
\newcommand{\CF}{\mathcal{F}}

\newcommand{\na}{\nabla}

\newcommand{\al}{\alpha}
\newcommand{\be}{\beta}
\newcommand{\ga}{\gamma}

\newcommand{\la}{\lambda}
\newcommand{\de}{\delta}
\newcommand{\si}{\sigma}
\newcommand{\pa}{\partial}

\newcommand{\eps}{\epsilon}

\newcommand{\De}{\Delta}
\newcommand{\Ga}{\Gamma}

\newcommand{\lag}{\langle}
\newcommand{\rag}{\rangle}

\begin{document}

\title[The Vlasov-Poisson-Boltzmann system without angular cutoff]{The Vlasov-Poisson-Boltzmann system without angular cutoff}

\author[R.-J. Duan]{Renjun Duan}
\address[RJD]{Department of Mathematics, The Chinese University of Hong Kong,
Shatin, Hong Kong}
\email{rjduan@math.cuhk.edu.hk}
%\thanks{Corresponding author: rjduan@math.cuhk.edu.hk}

\author[S.-Q. Liu]{Shuangqian Liu}
\address[SQL]{Department of Mathematics, Jinan Unviersity, Guangdong, P.R.~China}
\email{shqliusx@163.com}

\date{\today}

%\thanks{}

\begin{abstract}
This paper is concerned with the Vlasov-Poisson-Boltzmann system for plasma particles of two species in three space dimensions. The Boltzmann collision kernel is assumed to be angular non-cutoff with $-3<\ga<-2s$ and $1/2\leq s<1$, where $\ga$, $s$ are two parameters describing the kinetic and angular singularities, respectively. We establish the global existence and convergence rates of classical solutions to the Cauchy problem when initial data is near Maxwellians. This extends the results in \cite{DYZ-h, DYZ-s} for the cutoff kernel with $-2\leq \ga\leq 1$ to the case $-3<\ga<-2$ as long as the angular singularity exists instead and is strong enough, i.e., $s$ is close to $1$. The proof is based on the time-weighted energy method building also upon the recent studies of the non cutoff Boltzmann equation in \cite{GR} and the Vlasov-Poisson-Landau system in \cite{Guo5}.\\
\begin{center}
\it Dedicated to Professor Seiji Ukai (1940-2012)
\end{center}
\end{abstract}

\maketitle

\thispagestyle{empty}
\setcounter{tocdepth}{1}

\tableofcontents

\section{Introduction}

\subsection{Problem}
We consider the following
Vlasov-Poisson-Boltzmann system describing the motion of plasma
particles of two species (e.g.~ions and electrons) in the whole space ${\bf R}^3$, cf.~\cite{TrKr}:
\begin{equation}\label{VPB eqn}
\begin{split}
\partial_{t}F_++v\cdot\nabla_{x}F_++E\cdot\nabla_vF_+&=Q(F_+,F_+)+Q(F_-,F_+),\\
\partial_{t}F_-+v\cdot\nabla_{x}F_--E\cdot\nabla_vF_-&=Q(F_-,F_-)+Q(F_+,F_-).
\end{split}
\end{equation}
The self-consistent electrostatic field takes the form of $E(t,x)=-\nabla_x\phi$, with the electric potential $\phi$ satisfying
\begin{equation}\label{poisson eqn}
-\Delta\phi=\int_{{\bf R}^3}(F_+-F_-)\,dv, \ \ \phi\rightarrow 0\ \ \text{as}\ \ |x| \rightarrow\infty.
\end{equation}
The initial data of the  system is given as
\begin{equation}\label{initial data}
F_\pm(0,x,v)=F_{\pm,0}(x,v).
\end{equation}
Here, the unknown $F_{\pm}(t,x,v)\geq0$ stand for the velocity distribution functions for the
particles with position
$x=(x_1,x_2,x_3)\in{\bf R}^3$ and velocity $v=(v_1,v_2,v_3)\in{\bf
R}^{3}$  at time $t\geq0$. The bilinear  collision operator $Q(F,G)$  {on the right-hand side of \eqref{VPB eqn}} is defined by
\begin{equation*}
\begin{split}
Q(F,G)(v)=&\int_{{\bf R}^{3}}\int_{{\bf S}^{2}}B(v-u,\si) {\left[
F(u')G(v')-F(u)G(v)\right]\,dud\si},
\end{split}
\end{equation*}
where in terms of velocities $v$ and $u$ before the collision, velocities $v'$ and $u'$ after the collision are defined by
\begin{equation*}
v'=\frac{v+u}{2}+\frac{|v-u|}{2}\si,\quad
u'=\frac{v+u}{2}-\frac{|v-u|}{2}\si.
\end{equation*}
The Boltzmann collision kernel $B(v-u,\si)\geq 0$ depends only on the relative velocity
$|v-u|$ and on the deviation angle $\theta$ given by $\cos \theta=\langle\si,(v-u)/|v-u|\rangle$, where $\langle\,,\,\rangle$ is the usual dot product in ${\bf R}^3$. As in \cite{GR}, without loss of generality, we suppose that
$B(v-u,\si)$ is supported on $\cos \theta\geq 0$. Notice also that all the physical parameters, such as the particle masses and the light speed, and all other involving constants, have been chosen to be unit  for simplicity of presentation.
Throughout the paper, the collision kernel is further supposed to satisfy the following assumptions:

\begin{itemize}
  \item[$\bullet$] $B(v-u,\si)$ takes the product form in its argument as
$$
B(v-u,\si)=\Phi(|v-u|)\mathbbm{b}(\cos \theta)
$$
with $\Phi$ and  $\mathbbm{b}$ being nonnegative functions.

  \item[$\bullet$]  The angular function $\si \rightarrow \mathbbm{b}(\lag \si, (v-u)/|v-u|\rag )$ is not integrable on ${\bf S}^2$, i.e.
$$
\int_{{\bf S}^{2}}\mathbbm{b}(\cos \theta)\,d\si=2\pi\int_{0}^{\frac{\pi}{2}}\sin \theta \mathbbm{b}(\cos\theta)\,d\theta=\infty.
$$
Moreover, there are $c_\mathbbm{b}>0$, $0<s<1$ such that
\begin{equation*}
\frac{c_\mathbbm{b}}{\theta^{1+2s}}\leq \sin\theta\mathbbm{b}(\cos\theta)\leq \frac{1}{c_\mathbbm{b}\theta^{1
+2s}},\ \forall\,0<\theta\leq \frac{\pi}{2}.
\end{equation*}

  \item[$\bullet$] The kinetic function $z\rightarrow\Phi(|z|)$ satisfies
$$
\Phi(|z|)=C_{\Phi}|z|^\gamma
$$
for a contant  $C_\Phi>0$,
where the exponent $\gamma>-3$ is determined by the intermolecular interactive mechanism.
\end{itemize}

It is convenient to call soft potentials
when $-3<\gamma<-2s$, and hard potentials when $\gamma+2s\geq 0$. The current work will be restricted to the case of $-3<\gamma<-2s$ and $1/2\leq s<1$. Recall that when the  intermolecular interactive potential takes the inverse power law in the form of $U(|x|)=|x|^{-(\ell-1)}$ with $2<\ell<\infty$, the collision kernel $B(v-u,\si)$ in three space dimensions satisfies the above assumptions with $\gamma=\frac{\ell-5}{\ell-1}$ and $s=\frac{1}{\ell-1}$, and our restriction corresponds to the condition $2<\ell<3$ in terms of $\ell$. Note  $\gamma\rightarrow-3$ and $s\rightarrow1$ as $\ell\rightarrow2$ in the limiting case, for which the grazing collisions between particles are dominated and the Boltzmann collision term has to be replaced by the classical Landau collision term for the Coulomb potential, cf.~{\cite{Villani1}}.  {As far as the global classical solutions near Maxwellians to the pure Boltzmann equation with angular cutoff in the absence of any force are concerned, we only mention Ukai \cite{U74}, Ukai-Asano \cite{UA}, Caflisch \cite{Ca}, and Guo \cite{Guo-s}. }  

In the paper our goal is to establish the global existence of solutions to the
Cauchy problem (\ref{VPB eqn})-(\ref{initial data}) of the  Vlasov-Poisson-Boltzmann system near the global Maxwellian equilibrium states. This issue was firstly investigated by Guo \cite{Guo2} for the hard-sphere model of the Vlasov-Poisson-Boltzmann system in a periodic box. Since then, the robust energy method was also developed in \cite{Guo3} to deal with the hard-sphere Boltzmann equation even with the self-consistent electric and magnetic fields; see also  {\cite{Guo4,Guo-IUMJ,HY,LZ,YYZ}.}  However, the non hard-sphere case has remained open for general collision potentials either with the Grad's angular cutoff assumption or not. Until recently, Guo \cite{Guo5} made  further progress in proving
the global existence of classical solutions
to  the Vlasov-Poisson-Landau system in a periodic box for the most important Coulomb potential. One of the key points in the proof  there is to design a new velocity weight depending on the order of space and velocity derivatives so as to capture the anisotropic dissipation property of the linearized Landau operator. Due to the recent study of the non cutoff Boltzmann equation independently by Gressman-Strain \cite{GR, GR-am} and AMUXY \cite{AMUXY-1, AMUXY-2, AMUXY-3}, it is now well known that the linearized Boltzmann operator without angular cutoff has the similar anisotropic dissipation phenomenon with the Landau, cf.~\cite{DL,Guo-L}. Therefore, as mentioned in \cite{Guo5}, it is also interesting to see whether or not the approach in \cite{Guo5} can be applied to the non cutoff Vlasov-Poisson-Boltzmann system for the non hard-sphere model; see also \cite{Strain-Zhu} and \cite{LY} for two recent applications.

On the other hand, basing on the time weighted energy method, \cite{DYZ-h,DYZ-s, DYZ-L} recently developed another approach for the study of the Boltzmann or Landau equation with external forces for general collision potentials. The main difference with \cite{Guo5} is to introduce another kind of time-velocity dependent weight function which can induce the extra dissipation mechanism to compensate the weaker dissipation of the linearized collision operator in the  case of non hard-sphere models, particularly physically interesting soft potentials. Unfortunately, for the Vlasov-Poisson-Boltzmann system with angular cutoff, the problem was solved only in the case of $-2\leq \ga\leq 1$ and is still left open for the very soft potential case $-3<\ga<-2$.
In this paper, building on \cite{GR} and \cite{Strain11}, we will extend the results in \cite{DYZ-h, DYZ-s} for the cutoff kernel with $-2\leq \ga\leq 1$ to the non cutoff case $-3<\ga<-2$ as long as the angular singularity exists instead and it is strong enough, i.e., $s$ is close to $1$.

\subsection{Reformulation}
In what follows we will reformulate the problem as in {\cite{Guo2,Guo3}}. Denote a normalized global Maxwellian $\mu$ by
$$\mu(v)=\frac{1}{(2\pi)^{\frac{3}{2}}}\exp\left(-|v|^{2}/2\right).$$
Set $F_{\pm}(t,x,v)=\mu(v)+\sqrt{\mu(v)}f_{\pm}(t,x,v)$. Denote by $[\cdot,\cdot]$ the column vectors $F=[F_+,F_-]$, $f=[f_+,f_-]$ and $f_0=[f_{0,+},f_{0,-}]$. Then the
Cauchy problem (\ref{VPB eqn})-(\ref{initial data}) can be
reformulated as
\begin{eqnarray}
&\dis \partial_{t}f+v\cdot\nabla_{x}f-q\nabla_x\phi\cdot\nabla_v f+\nabla_x\phi\cdot v\sqrt{\mu }q_1+Lf=\Gamma(f,f)-\frac{q}{2}v\cdot \nabla_x\phi f,\label{perturbed eqn}\\
&\dis
-\Delta\phi=\int_{{\bf
R}^3}(f_+-f_-)\sqrt{\mu(v)}\,dv, \ \ \phi\rightarrow0\ \ \text{as}\ \ |x| \rightarrow\infty,\label{perturbed E}
\end{eqnarray}
with given initial data
\begin{equation}\label{perturbed data}
f(0,x,v)=f_{0}(x,v).
\end{equation}
Here, $q=diag(1,-1)$, $q_1=[1,-1]$. $L$ are $\Ga$ are the linearized and nonlinear collision operators, respectively.
For $f=[f_+,f_-]$ and  $g=[g_+,g_-]$,
\begin{eqnarray*}
&\dis Lf=\left[L_+f,L_-f\right],\\
&\dis L_{\pm}f=-\mu^{-1/2}\left\{2Q\left(\mu,\mu^{1/2}f_{\pm}\right)
+Q\left(\mu^{1/2}\{f_{\pm}+f_{\mp}\},\mu\right)\right\},
\end{eqnarray*}
and
\begin{eqnarray*}
&\dis \Gamma(f,g)=\left[\Gamma_+(f,g),\Gamma_-(f,g)\right],\\
&\dis \Gamma_{\pm}(f,g)=\mu^{-1/2}\left\{Q\left(\mu^{1/2}f_{\pm},\mu^{1/2}g_{\pm}\right)
+Q\left(\mu^{1/2}f_{\mp},\mu^{1/2}g_{\pm}\right)\right\}.
\end{eqnarray*}
For later use, it is convenient to introduce the bilinear operator $\mathscr{T}$ by
\begin{equation}\label{simple nonop def}
\begin{split}
\mathscr{T}(g_1,g_2)=&\mu^{-1/2}Q\left(\mu^{1/2}g_1,\mu^{1/2}g_2\right)\\
=&\int_{{\bf R}^3}du\int_{{\bf S}^2}d\si\, B(v-u,\si)\mu^{1/2}(u)\left[g_1(u')g_2(v')-g_1(u)g_2(v)\right]
\end{split}
\end{equation}
for two  scalar functions $g_1$, $g_2$, and thus $L=[L_+,L_-]$ and $\Ga=[\Ga_+,\Ga_-]$
can be rewritten as
\begin{equation}\label{Lop expr2}
L_{\pm}f=-\left\{2\mathscr{T}\left(\mu^{1/2},f_{\pm}\right)
+\mathscr{T}\left(f_{\pm}+f_{\mp},\mu^{1/2}\right)\right\},
\end{equation}
\begin{equation*}
\Gamma_{\pm}(f,g)=\mathscr{T}(f_{\pm},g_{\pm})+\mathscr{T}(f_{\mp},g_{\pm}).
\end{equation*}

\subsection{Basic properties of $L$}
For scalar functions $f_\pm$,
the first part of the linearized Boltzmann collision term $L_\pm f$ in (\ref{Lop expr2}) can be splitted as
\begin{equation}\label{main Lop}
-2\mathscr{T}\left(\mu^{1/2}, f_\pm\right)=-2\int_{{\bf R}^3}du\int_{{\bf S}^2}d\si\, B(v-u,\si)(f_\pm(v')-f_\pm(v))\mu^{1/2}(u)\mu^{1/2}(u')+2\tilde{\nu}(v)f_\pm(v),
\end{equation}
where
$$
\tilde{\nu}(v)=\int_{{\bf R}^3}du\int_{{\bf S}^2}d\si\, B(v-u,\si)\left(\mu^{1/2}(u)-\mu^{1/2}(u')\right)\mu^{1/2}(u).
$$
The first term on the right-hand side of \eqref{main Lop} contains a crucial Hilbert space structure, while for the second term, Pao's splitting
$$
\tilde{\nu}(v)=\nu_1(v)+\nu_{2}(v)
$$
holds true, cf.~\cite{Pao}, with the following known asymptotics
$$
\nu_1(v)\sim(1+|v|^2)^{\frac{\gamma+2s}{2}},\ \ \nu_2(v)\lesssim (1+|v|^2)^{\frac{\gamma}{2}}.
$$
We now collect some basic properties of the linearized collision operator $L$ as follows:
\begin{enumerate}
\item[(i)]
As in \cite{GR}, $L$ can be decomposed as $L=\mathcal {N}+\mathcal {K}$. Here for $f=[f_+,f_-]$,
$\mathcal {N}f=[\mathcal {N}_+f,\mathcal {N}_-f]$ is the
``norm part", given by
\begin{equation*}
\begin{split}
\mathcal {N}_\pm f=&-2\mathscr{T}\left(\mu^{1/2},f_{\pm}\right)-2\nu_{2}(v)f_{\pm}\\
=&-2\int_{{\bf R}^3}du\int_{{\bf S}^2}d\si\, B(v-u,\si)(f_{\pm}(v')-f_{\pm}(v))\mu^{1/2}(u)\mu^{1/2}(u')+2\nu_1(v)f_{\pm}(v).
\end{split}
\end{equation*}
Thus by using the pre-post collisional change of variables, $\mathcal {N}_\pm g$ satisfies the identity
\begin{equation*}
\begin{split}
\langle \mathcal {N}_\pm f,f_{\pm}\rangle=&\int_{{\bf R}^3}du\int_{{\bf R}^3}dv\int_{{\bf S}^2}d\si\, B(v-u,\si)(f_{\pm}(v')-f_{\pm}(v))^2\mu^{1/2}(u)\mu^{1/2}(u')\\ &+2\int_{{\bf R}^3}dv\,\nu_1(v)|f_{\pm}(v)|^2.
\end{split}
\end{equation*}
Moreover, $\mathcal {K}f=[\mathcal {K}_+f,\mathcal {K}_-f]$ is the ``compact part", given by
\begin{equation*}
\begin{split}
&\mathcal {K}_{\pm}f=-\mathscr{T}\left(f_{+}+f_{-},\mu^{1/2}\right)+2\nu_{2}(v)f_{\pm}\\
&=-\int_{{\bf R}^3}du\int_{{\bf S}^2}d\si\, B(v-u,\si)\mu^{1/2}(u)\left[(f_{+}+f_{-})(u')\mu^{1/2}(v')-(f_{+}+f_{-})(u)\mu^{1/2}(v)\right]
\\
&\quad +2\nu_2(v)f_{\pm}(v).
\end{split}
\end{equation*}

\item[(ii)] As in \cite{Guo3}, the null space of $L$ is given by
$$
\mathscr{N}=\ker L={\rm span}\left\{[1,0]\mu^{1/2}, [0,1]\mu^{1/2},
[v_i,v_i]\mu^{1/2} (1\leq i\leq3),
\left[|v|^2,|v|^2\right]\mu^{1/2}\right\}.
$$
For given $f(t,x,v)$, one can decompose $f(t,x,v)$ uniquely as
\begin{equation*}
f={\bf P}f+{\bf(I-P)}f.
\end{equation*}
Here, ${\bf P}$ denotes the orthogonal projection from
$L^{2}_{v}\times L^{2}_{v}$ to $\mathscr{N}$, defined by
\begin{equation}\label{macro def}
{\bf P}f=\left\{a_{+}(t,x)[1,0]+a_{-}(t,x)[0,1]+v\cdot
b(t,x)[1,1]+\left(|v|^{2}-3\right)c(t,x)[1,1]\right\}\sqrt{\mu},
\end{equation}
or equivalently ${\bf P}=[{\bf P}_+,{\bf P}_-]$ with
$$
{\bf P}_{\pm}f=\left\{a_{\pm}(t,x)+v\cdot
b(t,x)+\left(|v|^{2}-3\right)c(t,x)\right\}\sqrt{\mu}.
$$
Notice that
$$
\int_{{\bf R}^{3}}\psi(v)\cdot({\bf I-P})f\,dv=0, \quad \forall\,
\psi=[\psi_+,\psi_-]\in\mathscr{N}.
$$

\item[(iii)] For any fixed $(t,x)$, $L$ is nonnegative and further $L$ is known to be
locally coercive in the sense that there is a constant $\la>0$
such that, cf.~\cite{Mouhot, Mouhot-Strain}
\begin{equation*}
\langle f,Lf\rangle=\langle({\bf I-P})f,L({\bf
I-P})f\rangle\geq\la {\left\|(1+|v|^2)^{\frac{\gamma}{4}}({\bf
I-P})f\right\|_{L^2_v}^2}.
\end{equation*}
\end{enumerate}

\subsection{Notations}

Through the paper, $C$ denotes some positive constant (generally
large) and $\lambda$ denotes some positive constant (generally
small), where both $C$ and $\lambda$ may take different values in
different places. $A\lesssim B$ means that  there is a generic constant $C>0$ such that $A\leqslant CB$. $A\sim B$ means $A\lesssim B$ and $B\lesssim A$. For multi-indices
$\alpha=[\alpha_1, \alpha_2, \alpha_3]$ and $\beta=[\beta_{1},
\beta_{2}, \beta_{3}]$,
$
\partial^{\alpha}_{\beta}=\partial_{x_{1}}^{\alpha_{1}}
\partial_{x_{2}}^{\alpha_{2}}\partial_{x_{3}}^{\alpha_{3}}
\partial_{v_{1}}^{\beta_{1}}\partial_{v_{2}}^{\beta_{2}}\partial_{v_{3}}^{\beta_{3}},
$
and the length of $\al$ is denoted by $|\al|=\al_1+\al_2+\al_3$.
$\al'\leq\al$ means that no component of $\al'$
is greater than the component of $\al$, and $\al'<\al$ means that
$\al'\leq\al$ and $|\al'|<|\al|$. $\langle\cdot,\cdot\rangle$
denotes the $L^{2}$ inner product in ${\bf R}^{3}_{v}$, with the
$L^{2}$ norm $|\cdot|_2$ or $|\cdot|$. For notational simplicity,
$(\cdot,\cdot)$ denotes the $L^{2}$ inner product either in ${\bf
R}^{3}_{x}\times {\bf R}^{3}_{v}$ or in ${\bf R}^{3}_{x}$ with
the $L^{2}$ norm $||\cdot||$. For $\ell \in{\bf R}$,
$H^{\ell}$ denotes the usual Sobolev space. For $q\geq 1$, $Z_q$ denotes the space
$Z_{q}=L^{2}({\bf
R}_v^3; L^{q}({\bf R}_x^3))$ with the norm $\|f\|_{Z_q}=\left\|\|f\|_{L^q_x}\right\|_{L^2_v}$.

We now list series of notations introduced in \cite{GR}. Let $\mathcal {S}'({\bf
R}^3)$ be the space of the tempered distribution functions. $N^{s}_{\gamma}$ denotes the weighted geometric
fractional Sobolev space
$$
N^{s}_{\gamma}=\{f\in \mathcal {S}'({\bf R}^3):
|f|_{N^{s}_{\gamma}}<\infty\},
$$
with the anisotropic norm
\begin{equation*}
|f|^2_{N^{s}_{\gamma}}= {\left|f\right|_{L^2_{\gamma+2s}}^2}+\int_{{\bf
R}^3}dv\int_{{\bf R}^3}dv'\left(\langle v\rangle\langle
v'\rangle\right)^{\frac{\gamma+2s+1}{2}}
\frac{(f(v)-f(v'))^2}{d(v,v')^{3+2s}}{\chi}_{d(v,v')\leq1},
\end{equation*}
where $\langle v\rangle=\sqrt{1+|v|^2}$, $L^2_{\ell}$  for $\ell\in{\bf R}$ denotes
 the weighted  space with the norm
$$
|f|^2_{L^2_\ell}=\int_{{\bf R}^3}dv\,\langle v\rangle^{\ell}|f(v)|^2,
$$
the anisotropic metric $d(v,v')$ measuring the fractional differentiation effects is given by
$$
d(v,v')=\sqrt{|v-v'|^2+\frac{1}{4}(|v|^2-|v'|^2)^2},
$$
and $\chi_A$ is the indicator function of a set $A$.
In ${\bf
R}^3\times{\bf
R}^3$, we use $\|\cdot\|_{N^{s}_{\gamma}}=\left\||\cdot|_{N^{s}_{\gamma}}\right\|_{L^2_x}$.

To the end, the velocity weight function $w=w(v)$ always denotes
\begin{equation}
\label{def.w}
w(v)=\lag v\rag^{-\ga}.
\end{equation}
Let $\ell\in \R$. The weighted fractional Sobolev norm is given by
\begin{equation*}
\left|w^\ell f\right|^2_{H^{s}_{\gamma}}=\left|w^\ell
h\right|_{L^2_\gamma}^2+\int_{{\bf R}^3}dv\int_{{\bf R}^3}dv'\,
\frac{\left[\langle v\rangle^{\frac{\gamma}{2}}w^\ell(v)f(v)-\langle
v'\rangle^{\frac{\gamma}{2}}w^\ell(v')f(v')\right]^2}{|v-v'|^{3+2s}}{\chi}_{|v-v'|\leq1},
\end{equation*}
which turns out to be equivalent with
\begin{equation*}
\left|w^\ell f\right|^2_{H^{s}_{\gamma}}=\int_{{\bf R}^3}dv\, \langle
v\rangle^{\gamma}\left|\left(1-\Delta_v\right)^{\frac{s}{2}}\left(w^\ell(v)
f(v)\right)\right|^2.
\end{equation*}
The velocity-weighted $|\cdot|_{N^{s}_{\ga}}$-norm is given  by
\begin{equation*}
\left|w^\ell f\right|^2_{N^{s}_{\gamma}}=\left|w^{\ell}
f\right|_{L^2_{\gamma+2s}}^2+\int_{{\bf R}^3}dv\int_{{\bf
R}^3}dv'\left(\langle v\rangle\langle
v'\rangle\right)^{\frac{\gamma+2s+1}{2}}w^{2\ell}(v)
\frac{(f(v)-f(v'))^2}{d(v,v')^{3+2s}}{\chi}_{d(v,v')\leq1}.
\end{equation*}
Notice that, cf.~\cite{GR,GR-am},
\begin{equation*}
\left|w^\ell f\right|_{L^2_{\ga+2s}}^2+\left|w^\ell f\right|^2_{H^{s}_{\gamma}}\lesssim \left|w^\ell
f\right|^2_{N^{s}_{\gamma}}\lesssim \left|w^\ell
f\right|^2_{H^{s}_{\gamma+2s}}.
\end{equation*}
In ${\bf
R}^3\times{\bf
R}^3$, $\left\|w^\ell
f\right\|_{H^{s}_\gamma}=\left\|\left|w^\ell
f\right|_{H^{s}_{\gamma}}\right\|_{L^2_x}$ and $\left\|w^\ell
f\right\|_{N^{s}_\gamma}=\left\|\left|w^\ell
f\right|_{N^{s}_{\gamma}}\right\|_{L^2_v}$ are used. For the integer $K\geq 0$,  we also use the
anisotropic space $N_{\gamma,K}^{s}({\bf R}^3\times{\bf R}^3)$ containing the space-velocity derivatives, given
by
$$
\left\|w^\ell f\right\|^2_{N_{\gamma,K}^{s}}=\left\|w^\ell
f\right\|^2_{N_{\gamma,K}^{s}({\bf R}^3\times{\bf R}^3)}
=\sum\limits_{|\alpha|+|\beta|\leq
K}\left\|w^\ell\partial_\beta^\alpha
f\right\|^{2}_{N^{s}_{\gamma}({\bf R}^3\times{\bf R}^3)},
$$
where we write
$$
\left|w^\ell f\right|^2_{N_{\gamma,K}^{s}}=\left|w^\ell
f\right|^2_{N_{\gamma,K}^{s}({\bf R}^3)} =\sum\limits_{|\beta|\leq
K}\left|w^\ell\partial_\beta f\right|^{2}_{N^{s}_{\gamma}({\bf
R}^3)}
$$
whenever only the velocity derivatives are involved. For integer $K\geq0$, we define the Sobolev space
$$
\left|f\right|_{H^K}=\sum\limits_{|\beta|\leq K}\left|\partial_\beta
f\right|_{L^2({\bf R}^3)},\ \
\left\|f\right\|_{H^K}=\sum\limits_{|\alpha|+|\beta|\leq
K}\left\|\partial^\alpha_\beta f\right\|_{L^2({\bf R}^3\times{\bf
R}^3)}.
$$
For integer $K\geq0$ and $\ell\in{\bf R}$, we define the weighted
Sobolev space
$$
\left|w^\ell f\right|_{H^K}=\sum\limits_{|\beta|\leq
K}\left|w^\ell\partial_\beta f\right|_{L^2({\bf R}^3)},\ \
\left\|w^\ell f\right\|_{H^K}=\sum\limits_{|\alpha|+|\beta|\leq
K}\left\|w^\ell\partial^\alpha_\beta f\right\|_{L^2({\bf
R}^3\times{\bf R}^3)},
$$
and
$$
\left|w^\ell f\right|_{H^{K}_{\gamma}}=\sum\limits_{|\beta|\leq
K}\left|w^\ell\langle v\rangle^{\frac{\gamma}{2}}\partial_\beta
f\right|_{L^2({\bf R}^3)},\ \ \left\|w^\ell
f\right\|_{H^{K}_{\gamma}}=\sum\limits_{|\alpha|+|\beta|\leq
K}\left\|w^\ell\langle
v\rangle^{\frac{\gamma}{2}}\partial^\alpha_\beta f\right\|_{L^2({\bf
R}^3\times{\bf R}^3)}.
$$
Finally, we define $B_C\subset{\bf R}^3$ to be the ball with center origin and
radius $C$, and use $L^2(B_C)$ to denote the space
$L^2$ over  $B_C$ and likewise for other spaces.

\subsection{Main results}

To state the result of the paper, we introduce more notations. As in
{\cite{Guo5}}, for  $l\geq0$, $\al$ and $\be$, we define the velocity
weight function $w_l(\alpha,\beta)$ by
\begin{equation}
\label{def.wab}
w_l(\alpha,\beta)(v)={w}^{l+K-|\alpha|-|\beta|}(v)=\lag v\rag^{(-\ga)(l+K-|\al|-|\be|)},
\end{equation}
where $K$ is an integer. Corresponding to given $f=f(t,x,v)$, we define the instant energy
functional $\mathcal {E}_{l}(t)$ and the instant high-order energy
functional $\mathcal {E}^h_{l}(t)$, for $l\geq 0$, to be functionals satisfying the equivalent relations
\begin{eqnarray}
\mathcal {E}_{l}(t)&\sim& \sum\limits_{|\alpha|\leq
K}\|\partial^{\alpha}E(t)\|^2+\sum\limits_{|\alpha|+|\beta|\leq
K}\left\|w_l(\alpha,\beta)\partial_{\beta}^{\alpha}f(t)\right\|^{2},\label{def e}\\
\mathcal {E}^{h}_{l}(t)&\sim& \sum\limits_{|\alpha|\leq
K}\|\partial^{\alpha}E(t)\|^2+\sum\limits_{1\leq|\alpha|\leq
K}\|\partial^{\alpha}{\bf P}f\|^2+\sum\limits_{|\alpha|+|\beta|\leq
K}\left\|w_l(\alpha,\beta)\partial_{\beta}^{\alpha}({\bf I-P})f\right\|^{2},\label{def he}
\end{eqnarray}
and define the dissipation rate functional $\mathcal {D}_{l}(t)$ by
\begin{equation}\label{def d}
\mathcal {D}_{l}(t)=\sum\limits_{|\alpha|\leq
K-1}\|\partial^{\alpha}E(t)\|^2+\sum\limits_{1\leq|\alpha|\leq
K}\|\partial^{\alpha}{\bf P}f(t)\|^2+\sum\limits_{|\alpha|+|\beta|\leq
K}\left\|w_l(\alpha,\beta)\partial_{\beta}^{\alpha}({\bf
I-P})f(t)\right\|^{2}_{N^{s}_{\gamma}}.
\end{equation}
Here, $E=E(t,x)$ is understood to be determined by $f(t,x,v)$ in terms of
\begin{equation*}
E=-\na_x\phi, \quad \phi=\frac{1}{4\pi |x|} \ast_x\rho_f, \quad \rho_f=\int_{{\bf
R}^3}(f_+-f_-)\sqrt{\mu(v)}\,dv.
\end{equation*}
The main result of the paper is stated as follows.

\begin{Theorem}\label{thm.gl}
Let $-3<\ga<-2s$, $1/2\leq s<1$. Fix $l_0\geq 0$, $1/2<p<1$, the integer $K\geq 8$, and define
\begin{equation*}
l_1=\frac{5}{4(1-p)}\frac{\gamma+2s}{\gamma}.
\end{equation*}
Let $f_0(x,v)=[f_{0,+}(x,v),f_{0,-}(x,v)]$ satisfy
$F_{\pm}(0,x,v)=\mu(v)+\sqrt{\mu(v)}f_{0,{\pm}}(x,v)\geq 0$, and
\begin{equation}\label{thm.gl.a1}
\int_{{\bf R}^3}\rho_{f_0}\,dx=0,\ \ \int_{{\bf R}^3}(1+|x|)|\rho_{f_0}|\,dx<\infty.
\end{equation}
Then, there are functionals $\CE_{l}(t)$ and $\CE_l^{h}(t)$ in the sense of \eqref{def e} and \eqref{def he} such that the following thing holds true.  If
\begin{equation}\label{thm.gl.a2}
\epsilon_{0}=\sqrt{\mathcal{E}_{l_0+l_1}(0)}+\left\|w^{l_2}f_0\right\|_{Z_1}+\|(1+|x|)\rho_0\|_{L^1},
\end{equation}
is sufficiently small, where $l_2>\frac{5(\ga+2s)}{4\ga}$ is a constant, then there exists a unique global solution $f(t,x,v)$ to the Cauchy
problem (\ref{perturbed eqn})-(\ref{perturbed data}) of the Vlasov-Poisson-Boltzmann system such that
$F_{\pm}(t,x,v)=\mu(v)+\sqrt{\mu(v)}f_{\pm}(t,x,v)\geq 0$ and
\begin{eqnarray}
&&\CE_{l_0+l_1}(t)  \lesssim \eps_0^2,\label{thm.gl.c1}\\
&&\CE_{l_0}(t) \lesssim\eps_0^2(1+t)^{-\frac{3}{2}},\label{thm.gl.c2}\\
&&\CE_{l_0}^{h}(t) \lesssim \eps_0^2(1+t)^{-\frac{3}{2}-p},\label{thm.gl.c3}
\end{eqnarray}
for any $t\geq 0$.
\end{Theorem}

Some remarks are given as follows. Notice that $l_0$ can take zero while the choice of  other parameters $l_1$, $l_2$, $K$ and $p$ could not be optimal in order for initial data to have the weaker regularity and velocity moments. The restriction of those parameters is related to obtaining the following closed a priori estimate
$$
X(t)\lesssim \eps_0^2+ X^{\frac{3}{2}}(t) +X^2(t)
$$
with respect to the time-weighted energy norm $X(t)$ defined by
\begin{equation}\label{key.X}
\begin{split}
X(t)=\sup\limits_{0\leq \tau\leq t}\mathcal {E}_{l_0+l_1}(\tau)+\sup\limits_{0\leq\tau\leq t}(1+\tau)^{\frac{3}{2}}\mathcal {E}_{l_0}(\tau)+\sup\limits_{0\leq\tau\leq t}(1+\tau)^{\frac{3}{2}+p}\mathcal {E}^h_{l_0}(\tau).
\end{split}
\end{equation}
Here, the parameter $p$ is introduced to take care the time-decay of the high-order energy functional $\CE_{l_0}^{h}(t)$.  {The time decay rate $(1+t)^{-\frac{3}{2}-p}$ for $\CE_{l_0}^{h}(t)$ in \eqref{thm.gl.c3} is not optimal compared to the linearized system; see Theorem \ref{basic decay lemma}  given later on. It is then of interest to upgrade it to $(1+t)^{-5/2}$, that is to prove \eqref{thm.gl.c3} in the case of $p=1$, and we will make a further discussion of possibility at the end of the paper.
Moreover, similar results in the} hard potential case $\ga+2s\geq 0$ with $0<s<1$ could be considered in the simpler way, but for the soft potential case $-3<\ga<-2s$ in Theorem \ref{thm.gl}, the restriction $s\geq 1/2$ is necessary in our proof due to the technique of the approach; we will clarify this point later. The assumptions concerning \eqref{thm.gl.a1} and \eqref{thm.gl.a2} arise from the linearized analysis of the Vlasov-Poisson-Boltzmann system in order to assure the strong enough time-decay rate of the linearized solution operator; see \eqref{est.rho} in the proof of Theorem \ref{basic decay lemma}.

In what follows let us point out several key technical points in the proof of Theorem \ref{thm.gl}. First of all, we emphasize the role of the velocity weight $w_l(\al,\be)$ in  {\eqref{def.wab}.
Such weight was firstly introduced in \cite{SG1,SG2} to deal with the time decay of the Boltzmann equation for soft potentials on torus, and it was also used recently in \cite{Guo5} to investigate the Vlasov-Poisson-Landau system for Coulomb potentials.}
In fact, the linearized non cutoff Boltzmann operator  enjoys the anisotropic dissipation norm
\begin{equation}
\label{diss.norm}
\|f\|_{N^s_\ga}^2\gtrsim \|\lag v \rag^{\frac{\ga+2s}{2}} f\|^2+\|\lag v\rag^{\frac{\ga}{2}} (1-\De_v)^{\frac{s}{2}}f\|^2.
\end{equation}
Then, in terms of this dissipation norm, the choice of $w_l(\al,\be)$ should depend on the weighted estimates on the linear term $v\cdot \na_x f$ and two nonlinear terms $\na_x\phi\cdot \na_v f$ and $v\cdot \na_x\phi f$. For $v\cdot \na_x f$, one has to bound
\begin{equation}
\label{intro.t1}
\left(\pa_{\be-e_i}^{\al+e_i} f,  w_l^2(\al,\be) \pa_\be^\al f\right)
\end{equation}
with $|\al|+|\be|\leq K$ and $|\be|\geq 1$. For simplicity of presentation, let $f$ be purely microscopic, i.e.~$f=(\FI-\FP)f$. Since $|\be|\geq 1$, the trick is to do the splitting $\pa_{\be}^{\al}f=\pa_{e_i}\pa_{\be-e_i}^\al f$ in the inner product term above, write  in a rough way the first-order  velocity differentiation as
$\pa_{e_i}=\pa_{e_i}^{1/2}\pa_{e_i}^{1/2}$, and further use $w_l(\al,\be)=w_l(\al,\be-e_i)\lag v\rag^{\ga}$ to gain the degenerate velocity weight, so that \eqref{intro.t1} can be bounded by the dissipation norm
\begin{equation*}
\|\lag v\rag^{\frac{\ga}{2}} (1-\De_v)^{\frac{s}{2}}[w_l(\al,\be-e_i)\pa_{\be-e_i}^\al f]\|^2+\|\lag v\rag^{\frac{\ga}{2}} (1-\De_v)^{\frac{s}{2}}[w_l(\al+e_i,\be-e_i)\pa_{\be-e_i}^{{\al+e_i}}f]\|^2
\end{equation*}
up to some other controllable terms, where $s\geq 1/2$ was used. Notice that two terms in the sum above correspond to the second term on the right-hand side of \eqref{diss.norm}, and they also have velocity differentiation whose order is $|\be-e_i|$ less than $|\be|$.  For $\na_x\phi \cdot \na_v f$, one has to meet with the estimate on the trilinear inner product  term
\begin{equation*}
\left(\pa^{\al_1+e_i}\phi\pa_{\be+e_i}^{\al-\al_1} f,  w_l^2(\al,\be) \pa_\be^\al f\right)
\end{equation*}
with $|\al|+|\be|\leq K$ and $0<\al_1\leq \al$. For this term, we use the same trick to make estimates as for \eqref{intro.t1} by writing $\pa_{\be+e_i}^{\al-\al_1} f=\pa_{e_i}\pa_{\be}^{\al-\al_1} f$ and $\pa_{e_i}=\pa_{e_i}^{1/2}\pa_{e_i}^{1/2}$.  For $v\cdot \na_x \phi f$, one has to bound
\begin{equation*}
\left(v_i\pa^{\al_1+e_i}\phi \pa_{\be}^{\al-\al_1} f,  w_l^2(\al,\be) \pa_\be^\al f\right)
\end{equation*}
with $|\al|+|\be|\leq K$ and $0<\al_1\leq \al$. Since $\pa^{\al-\al_1} f$ losses at least one space differentiation compared to $\pa^\al f$, one can gain the velocity moment $\lag v\rag^{\ga}$ from $w_l(\al,\be)$ so that
\begin{equation*}
|v_i|w_l(\al,\be)\leq \lag v\rag^{\ga+2s}w_l(\al-\al_1,\be)
\end{equation*}
due to $s\geq 1/2$ once again. Therefore, one can use the first part on the right-hand side of \eqref{diss.norm} from the dissipation norm to bound
\begin{equation*}
\int_{{\bf R}^3}\int_{{\bf R}^3} |\pa^{\al_1+e_i}\phi|\cdot |\lag v\rag^{\frac{\ga+2s}{2}} w_l(\al-\al_1,\be) \pa_{\be}^{\al-\al_1} f|\cdot |\lag v\rag^{\frac{\ga+2s}{2}} w_l(\al,\be) \pa_{\be}^{\al} f|\,dxdv.
\end{equation*}

The second technical point concerns the $w_l(\al,\be)$-weighted estimate on the nonlinear term $\Ga(f,f)$.  The corresponding results
obtained in \cite{GR} can not be directly
applied here, since the velocity weight function depends on the total order of space-velocity differentiation. We will make some slight modifications involving
the distribution of weights. Essentially, whenever $|v'|^2\sim
|v|^2+|u|^2$, instead of using $
w^{2\ell}(v')\lesssim w^{2\ell}(u)w^{2\ell}(v),
$
we estimate $w^{2\ell}(v')$ as
$$
w^{2\ell}(v')\lesssim w^{2\ell}(u)+w^{2\ell}(v).
$$
The third point is to design the time-weighted norm $X(t)$ in \eqref{key.X} to capture the dispersive property of the Vlasov-Poisson-Boltzmann system in the case of the whole space ${\bf R}^3$. This trick has be been used in \cite{DuanUYZ} and the recent work \cite{DYZ-h,DYZ-s,DYZ-L}, where the key issue is to apply the time-decay property of solutions to close all the nonlinear energy estimates. Specifically, due to the technique of the approach, one has to deal with a term in the form of
\begin{equation*}
 \|\partial_t\phi\|_{L^\infty}\mathcal {E}_l(t)
\end{equation*}
which can not be absorbed by the energy dissipation norm even under the smallness assumption on the solution itself. Observe from the later proof that $\|\partial_t\phi\|_{L^\infty}$ is bounded by the high-order energy functional $\CE_l^{h}(t)$ and hence is time integrable as shown in \cite{Guo5}. Thus, it is natural to use a kind of time-weighted norm to close the a priori estimates. On the other hand, notice that the study of the time-decay property for the linearized system with or without the self-consistent forces is now well established; see \cite{DuanS1, DuanS2, YY-CMP} for hard potentials and \cite{Strain11, DYZ-s} for soft potentials; see also \cite{YZ2} in terms of the energy method only.

The rest of the paper is arranged as follows. In Section \ref{sec.2}, we list basic lemmas concerning the properties of $L$ and $\Ga$ in the functional framework of \cite{GR}. In Section \ref{sec.3}, we present series of the $w_l(\al,\be)$-weighted estimates on all the nonlinear terms. Section \ref{sec.4} and Section \ref{sec.5} are concerned with the linearized analysis for the time-decay property and the energy estimates to gain the macroscopic dissipation, respectively. In Section \ref{sec.6} we make series of the  a priori estimates through the energy method, and in Section \ref{sec.7} we complete the proof Theorem \ref{thm.gl}.

%\newpage
\section{Preliminaries}\label{sec.2}

In this section, we list two basic  lemmas which will be used in the
later proof. Recall $w=w(v)=\lag v \rag^{-\gamma}$ in \eqref{def.w}, and we always
suppose $-3<\ga<-2s$ and $1/2\leq s<1$. The first lemma concerns the
coercivity estimate on the linearized collision operators $L$; its
proof can be found in \cite[{pp.783, Lemma 2.6 and pp.829, Theorem
8.1}]{GR}.

\begin{Lemma}\label{estimate on L}
(i) It holds that
\begin{equation}\label{estimate on L2}
\left(Lg, g\right)\gtrsim
\left\|({\bf I-P})g\right\|^2_{N^{s}_{\gamma}}.
\end{equation}
(ii) Let $\ell\in {\bf R}$. There is $C>0$ such that
\begin{equation*}
\left(w^{2\ell}Lg,g\right)\gtrsim
\left\|w^{\ell}g\right\|^2_{N^{s}_{\gamma}}-C\left\|g\right\|^2_{L^2(B_C)}.
\end{equation*}
(iii) Let $\beta>0$, $\ell\in{\bf
R}$.  For any $\eta>0$, there are
$C_{\eta}>0$, $C>0$ such that
\begin{equation*}
\left(w^{2\ell}\partial_{\beta}Lg,\partial_{\beta}g\right)\gtrsim
\left\|w^{\ell}\partial_{\beta}g\right\|^2_{N^{s}_{\gamma}}
-\eta\sum\limits_{\beta_1\leq
\beta}\left\|w^{\ell}\partial_{\beta_1}g\right\|^2_{N^{s}_{\gamma}}
-C_{\eta}\left\|g\right\|^2_{L^2(B_C)}.
\end{equation*}
\end{Lemma}

The second lemma concerns the  estimates on the nonlinear collision
operator $\Gamma$. We point out that the corresponding results
obtained in \cite[{pp.817, Lemma 6.1}]{GR} can not be directly
applied here. One has to make some slight modifications involving
the distribution of weights. In fact, for the weighted estimate on
the triple inner product term  $\langle
w^{2\ell}\partial^{\alpha}_{\beta}\Ga(g_1,g_2),
\partial^{\alpha}_{\beta}g_{3}\rangle$,  the weight $w^{2\ell}$ was assigned to every one of the three functions $g_1,g_2$ and $g_3$ in \cite{GR}.  When the integration with respect to space variable is further taken,
Sobolev's inequality must be used to control the $L_x^\infty$-norm of either $g_1$ or $g_2$ so that the order of $x$-derivatives should be lifted. However, the lifting is dangerous in the case when  the weight $w_l(\alpha,\beta)$  is used for soft potentials. This can be seen by noticing that $\ell$ depends also on the order of $x$-derivatives and hence the higher order $x$-derivatives are associated with the weaker weight.

\begin{Lemma}\label{estimates on nonop}
Let $g_i=[g_{i,+}, g_{i,-}]\in C_{0}^{\infty}({\bf R}^3, {\bf R}^2),$ $1\leq i\leq3$, and let
$|\alpha|+|\beta|\leq K$ with $\al=\al_1+\al_2$ and $(\be_1,\be_2)\leq
\beta$. Then for any $\ell \geq0$ and $m\geq0$,

\medskip
\noindent(i) when $|\alpha_1|+|\beta_1|\leq K/2$,
\begin{equation}\label{nonop1}
\begin{split}
\Big|\big(w^{2\ell}&\partial^{\alpha}_{\beta}\Gamma_{\pm}(g_1,g_2),
\partial^{\alpha}_{\beta}g_{3,\pm}\big)\Big|\\
\lesssim& \sum\limits_{\alpha_1+\alpha_2= \alpha\atop{(\beta_{1},
\beta_{2}) \leq \beta}}\int_{{\bf R}^{3}}
\left|\partial^{\alpha_1}_{\beta_{1}}g_{1}\right|_{2}
\left|w^{\ell}\partial^{\alpha_2}_{\beta_{2}}g_{2}\right|_{N^{s}_{\gamma}}
\left|w^{\ell}\partial^{\alpha}_{\beta}g_{3}\right|_{N^{s}_{\gamma}}dx\\
&+\sum\limits_{\alpha_1+\alpha_2= \alpha\atop{(\beta_{1}, \beta_{2})
\leq \beta}}\int_{{\bf R}^{3}}
\left|w^{\ell}\partial^{\alpha_1}_{\beta_{1}}g_{1}\right|_{2}
\left|\partial^{\alpha_2}_{\beta_{2}}g_{2}\right|_{N^{s}_{\gamma}}
\left|w^{\ell}\partial^{\alpha}_{\beta}g_{3}\right|_{N^{s}_{\gamma}}dx
\\
&+\sum\limits_{\alpha_1+\alpha_2= \alpha\atop{(\beta_{1}, \beta_{2})
\leq \beta}}\int_{{\bf R}^{3}}
\left|w^{-m}\partial^{\alpha_1}_{\beta_{1}}g_{1}\right|_{H^{2}}
\left|w^{\ell}\partial^{\alpha_2}_{\beta_{2}}g_{2}\right|_{N^{s}_{\gamma}}
\left|w^{\ell}\partial^{\alpha}_{\beta}g_{3}\right|_{N^{s}_{\gamma}}dx,
\end{split}
\end{equation}

\medskip
\noindent(ii)
when $|\alpha_1|+|\beta_1|\geq K/2$,
\begin{equation}\label{nonop2}
\begin{split}
\Big|\big(w^{2\ell}&\partial^{\alpha}_{\beta}\Gamma_{\pm}(g_1,g_2),
\partial^{\alpha}_{\beta}g_{3,\pm}\big)\Big|\\
\lesssim& \sum\limits_{\alpha_1+\alpha_2= \alpha\atop{(\beta_{1},
\beta_{2}) \leq \beta}}\int_{{\bf R}^{3}}
\left|\partial^{\alpha_1}_{\beta_{1}}g_{1}\right|_{2}
\left|w^{\ell}\partial^{\alpha_2}_{\beta_{2}}g_{2}\right|_{N^{s}_{\gamma}}
\left|w^{\ell}\partial^{\alpha}_{\beta}g_{3}\right|_{N^{s}_{\gamma}}dx\\
&+\sum\limits_{\alpha_1+\alpha_2= \alpha\atop{(\beta_{1}, \beta_{2})
\leq \beta}}\int_{{\bf R}^{3}}
\left|w^{\ell}\partial^{\alpha_1}_{\beta_{1}}g_{1}\right|_{2}
\left|\partial^{\alpha_2}_{\beta_{2}}g_{2}\right|_{N^{s}_{\gamma}}
\left|w^{\ell}\partial^{\alpha}_{\beta}g_{3}\right|_{N^{s}_{\gamma}}dx
\\
&+\sum\limits_{\alpha_1+\alpha_2= \alpha\atop{(\beta_{1}, \beta_{2})
\leq \beta}}\int_{{\bf R}^{3}}
\left|w^{-m}\partial^{\alpha_1}_{\beta_{1}}g_{1}\right|_{2}
\left|w^{\ell}\partial^{\alpha_2}_{\beta_{2}}g_{2}\right|_{N_{\gamma,2}^{s}}
\left|w^{\ell}\partial^{\alpha}_{\beta}g_{3}\right|_{N^{s}_{\gamma}}dx.
\end{split}
\end{equation}
\end{Lemma}

To prove the lemma above, we need to make some preparations by
recalling some notations used in \cite[{pp.791--792}]{GR}.
Notice (\ref{simple nonop def}) for the definition of
$\mathscr{T}$. Consider the following inner product
\begin{equation*}
\left\langle w^{2\ell}\partial^{\alpha}_{\beta}\mathscr{T}(h_1,h_2),
\partial^{\alpha}_{\beta}h_3\right\rangle=\sum\limits_{\beta_1+\beta_2+\beta_{\mu}=\beta}
\sum\limits_{\alpha_1+\alpha_1=\alpha}C_{\alpha,\alpha_1}^{\beta,\beta_1,\beta_{\mu}}
\left\langle w^{2\ell}\mathscr{T}_{\beta_{\mu}}(\partial_{\beta_1}^{\alpha_1}h_1,\partial_{\beta_2}^{\alpha_2}h_2),
\partial^{\alpha}_{\beta}h_3\right\rangle,
\end{equation*}
where
$$
\mathscr{T}_{\beta_{\mu}}(h_1,h_2)=
\int_{{\bf R}^3}du\int_{{\bf S}^2}d\si\, B(v-u,\si)\left[\partial_{\beta_{\mu}}\mu^{1/2}(u)\right]\left[h_1(u')h_2(v')-h_1(u)h_2(v)\right].
$$
Let $\{\eta_\kappa\}^{\kappa=+\infty}_{\kappa=-\infty}$ be a partition on unity on $(0,\infty)$ such that
$|\eta_\kappa|_\infty\leq1$ and ${\rm supp}\,(\eta_\kappa)\subset \left[2^{-\kappa-1},2^{-\kappa}\right].$ For each $\kappa$
we use the notation
\begin{equation}
\label{def.B.ka}
B_\kappa=B(v-u,\si)\mathbbm{b}\left(\left\langle\frac{v-u}{|v-u|},\si\right\rangle\right)\eta_\kappa(|v-v'|).
\end{equation}
We now denote
\begin{eqnarray*}
\begin{split}
\mathscr{T}_+^{\kappa,\ell}(h_1,h_2,h_3)=&\int_{{\bf R}^3}du\int_{{\bf R}^3}dv\int_{{\bf S}^2}d\si\,B_\kappa(v-u,\si)h_1(u)h_2(v)\\
&\times\left[\partial_{\beta_\mu}\mu^{1/2}(u')\right]h_3(v')w^{2\ell}(v'),
\end{split}
\end{eqnarray*}
\begin{eqnarray*}
\begin{split}
\mathscr{T}_-^{\kappa,\ell}(h_1,h_2,h_3)=&\int_{{\bf R}^3}du\int_{{\bf R}^3}dv\int_{{\bf S}^2}d\si\, B_\kappa(v-u,\si)h_1(u)h_2(v)\\
&\times\left[\partial_{\beta_\mu}\mu^{1/2}(u)\right]h_3(v)w^{2\ell}(v).
\end{split}
\end{eqnarray*}
On the other hand, we express the collision operator (\ref{simple nonop def}) using its dual formulation as in \cite[A1]{GR}.
In fact, after a transformation, one can put cancellations on the function $h_2$ as follows
\begin{equation}\label{Carleman type}
\begin{split}
\big\langle& w^{2\ell}\mathscr{T}(h_1,h_2),
h_3\big\rangle\\ =&\int_{{\bf R}^3}du\int_{{\bf R}^3}dv'\int_{E_u^{v'}}d\pi_v\, \tilde{B}(v-u,\si)h_1(u)h_3(v')w^{2\ell}(v')\left[\mu^{1/2}(u')h_2(v)-\mu^{1/2}(u)h_2(v')\right]\\
&+\mathscr{T}_\ast^\ell(h_1,h_2,h_3),
\end{split}
\end{equation}
where $u'=u+v-v'$, the kernel $\tilde{B}$ is given by
\begin{equation*}
\tilde{B}=4\frac{B\left(v-u, {\frac{2v'-v-u}{|2v'-v-u|}}\right)}{|v'-u||v-u|},
\end{equation*}
while the corresponding dual operator $\mathscr{T}_\ast^\ell$ given by
\begin{equation*}
\begin{split}
\mathscr{T}_\ast^\ell(h_1,h_2,h_3)=&\int_{{\bf R}^3}dv'\,h_2(v')h_3(v')w^{2\ell}(v')
\int_{{\bf R}^3}du\,h_1(u)\left[\partial_{\beta_\mu}\mu^{1/2}(u)\right]\\
&\times\int_{E_u^{v'}}d\pi_v\, \tilde{B}\left(1-\frac{|v'-u|^{3+\gamma}}{|v-u|^{3+\gamma}}\right).
\end{split}
\end{equation*}
is not differential at longer. In those integrals, $d\pi_v$ means
the Lebesgue measure on the 2-dimensional hyperplane $E_u^{v'}$
defined by $E_u^{v'}=\left\{v\in{\bf R}^3: \langle
u-v',v-v'\rangle=0\right\}$, and $v$ is the variable of integration.
Note that in (\ref{Carleman type}) we use $\mathscr{T}_\ast^\ell$
with $\beta_\mu=0$.

With the observation above, one can use the following alternative representations for $\mathscr{T}_+^{\kappa,\ell}$
as well as $\mathscr{T}_\ast^{\kappa,\ell}$:
\begin{eqnarray*}
\begin{split}
\mathscr{T}_+^{\kappa,\ell}(h_1,h_2,h_3)=&\int_{{\bf R}^3}du\int_{{\bf R}^3}dv'\int_{E_u^{v'}}d\pi_v\, \tilde{B}_\kappa(v-u,\si)\\
&\times h_1(u)h_2(v)\left[\partial_{\beta_\mu}\mu^{1/2}(u')\right]h_3(v')w^{2\ell}(v'),
\end{split}
\end{eqnarray*}
\begin{eqnarray*}
\begin{split}
\mathscr{T}_\ast^{\kappa,\ell}(h_1,h_2,h_3)=&\int_{{\bf R}^3}du\int_{{\bf R}^3}dv'\int_{E_u^{v'}}d\pi_v\, \tilde{B}_\kappa(v-u,\si)\\
&\times h_1(u)h_2(v')\left[\partial_{\beta_\mu}\mu^{1/2}(u)\right]h_3(v')w^{2\ell}(v').
\end{split}
\end{eqnarray*}
Then for $h_1, h_2, h_3\in \mathcal {S}({\bf R}^3)$, the pre-post collisional change of variables, the dual representation, and the previous calculation guarantee that
\begin{eqnarray}\label{two key form}
\begin{split}
\left\langle w^{2\ell}\mathscr{T}_{\beta'}(h_1,h_2),
h_3\right\rangle=&\sum\limits_{\kappa=-\infty}^{+\infty}
\left\{\mathscr{T}_+^{\kappa,\ell}(h_1,h_2,h_3)-\mathscr{T}_-^{\kappa,\ell}(h_1,h_2,h_3)\right\}\\
=&\mathscr{T}_\ast^\ell(h_1,h_2,h_3)+\sum\limits_{\kappa=-\infty}^{+\infty}
\left\{\mathscr{T}_+^{\kappa,\ell}(h_1,h_2,h_3)-\mathscr{T}_\ast^{\kappa,\ell}(h_1,h_2,h_3)\right\}.
\end{split}
\end{eqnarray}
To the end $\zeta(v)$ denotes an arbitrary smooth function satisfying
for some positive constant $\la $ that
\begin{equation}\label{expo fun}
|\zeta(v)|\sim e^{-\la |v|^2}.
\end{equation}
Now we collect estimates for the operators $\mathscr{T}_+^{\kappa,\ell}$, $\mathscr{T}_-^{\kappa,\ell}$ and $\mathscr{T}_\ast^{\kappa,\ell}$ appearing in \eqref{two key form}, which can be used to prove  Lemma \ref{estimates on nonop}. Notice that only the soft potential case $-3<\gamma<-2s$ with $1/2\leq s<1$ is considered here.

\begin{Proposition}\label{prop.non}
Let  $\kappa$ be an integer, $m\geq0$, $\ell\in{\bf R}$ and
$\zeta$ defined by (\ref{expo fun}).

\medskip
\noindent (i)
\begin{equation*}
\left|\mathscr{T}_-^{\kappa,\ell}(h_1,h_2,h_3)\right|\lesssim 2^{2s\kappa}
\left|w^{-m}h_1\right|_{H^{2}}\left|w^{\ell}h_2\right|_{L^2_{\gamma+2s}}\left|w^{\ell}h_3\right|_{L^2_{\gamma+2s}},
\end{equation*}
\begin{equation*}
\left|\mathscr{T}_-^{\kappa,\ell}(h_1,h_2,h_3)\right|\lesssim
2^{2s\kappa}
\left|w^{-m}h_1\right|_{L^2}\left|w^{\ell}h_2\right|_{H^{2}_{\gamma+2s}}\left|w^{\ell}h_3\right|_{L^2_{\gamma+2s}},
\end{equation*}
and
\begin{equation*}
\left|\mathscr{T}_-^{\kappa,\ell}(h_1,h_2,\zeta)\right|+\left|\mathscr{T}_-^{\kappa,\ell}(h_1,\zeta,h_2)\right|\lesssim 2^{2s\kappa}
\left|w^{-m}h_1\right|_{L^2}\left|w^{-m}h_2\right|_{L^2}.
\end{equation*}

\medskip
\noindent (ii)
\begin{equation*}
\left|\mathscr{T}_\ast^{\kappa,\ell}(h_1,h_2,h_3)\right|\lesssim 2^{2s\kappa}
\left|w^{-m}h_1\right|_{H^{2}}\left|w^{\ell}h_2\right|_{L^2_{\gamma}}\left|w^{\ell}h_3\right|_{L^2_{\gamma}},
\end{equation*}
\begin{equation*}
\left|\mathscr{T}_\ast^{\kappa,\ell}(h_1,h_2,h_3)\right|\lesssim
2^{2s\kappa}
\left|w^{-m}h_1\right|_{L^2}\left|w^{\ell}h_2\right|_{H^{2}_{\gamma}}\left|w^{\ell}h_3\right|_{L^2_{\gamma}},
\end{equation*}
and
\begin{equation*}
\left|\mathscr{T}_\ast^{\kappa,\ell}(h_1,h_2,\zeta)\right|+\left|\mathscr{T}_\ast^{\kappa,\ell}(h_1,\zeta,h_2)\right|\lesssim 2^{2s\kappa}
\left|w^{-m}h_1\right|_{L^2}\left|w^{-m}h_2\right|_{L^2}.
\end{equation*}

\medskip
\noindent (iii)
\begin{equation}\label{nonopp1}
\begin{split}
\left|\mathscr{T}_+^{\kappa,\ell}(h_1,h_2,h_3)\right|\lesssim& 2^{2s\kappa}
\left|h_1\right|_{L^2}\left|w^{\ell}h_2\right|_{L^2_{\gamma+2s}}
\left|w^{\ell}h_3\right|_{L^2_{\gamma+2s}}\\
&+2^{2s\kappa}
\left|w^{\ell}h_1\right|_{L^2}\left|h_2\right|_{L^2_{\gamma+2s}}
\left|w^{\ell}h_3\right|_{L^2_{\gamma+2s}}\\
&+2^{2s\kappa}
\left|w^{-m}h_1\right|_{H^{2}}\left|w^{\ell}h_2\right|_{L^2_{\gamma+2s}}
\left|w^{\ell}h_3\right|_{L^2_{\gamma+2s}},
\end{split}
\end{equation}
\begin{equation}\label{nonopp2}
\begin{split}
\left|\mathscr{T}_+^{\kappa,\ell}(h_1,h_2,h_3)\right|\lesssim& 2^{2s\kappa}
\left|h_1\right|_{L^2}\left|w^{\ell}h_2\right|_{L^2_{\gamma+2s}}
\left|w^{\ell}h_3\right|_{L^2_{\gamma+2s}}\\
&+2^{2s\kappa}
\left|w^{\ell}h_1\right|_{L^2}\left|h_2\right|_{L^2_{\gamma+2s}}
\left|w^{\ell}h_3\right|_{L^2_{\gamma+2s}}\\
&+2^{2s\kappa}
\left|w^{-m}h_1\right|_{L^2}\left|w^{\ell}h_2\right|_{H^{2}_{\gamma+2s}}
\left|w^{\ell}h_3\right|_{L^2_{\gamma+2s}},
\end{split}
\end{equation}
and
\begin{equation}\label{nonoppexpo}
\left|\mathscr{T}_+^{\kappa,\ell}(h_1,h_2,\zeta)\right|\lesssim 2^{s\kappa}
\left|w^{-m}h_1\right|_{L^2}\left|w^{-m}h_2\right|_{L^2}.
\end{equation}

\end{Proposition}

\begin{proof}
First of all, notice that (i), (ii) and \eqref{nonoppexpo} in (iii)
are the same as in \cite[{pp.803--804, Proposition 4.1, 4.2, 4.4, and
pp.808, Proposition 4.8}]{GR}, and only \eqref{nonopp1} and
\eqref{nonopp2} are different. For brevity, we prove (\ref{nonopp1})
only. The key point is to assign the velocity weight to $h_1$ and
$h_2$ in a better way. The following inequalities will be frequently
used in the later proof:
\begin{equation}\label{kernel estimates}
\int_{{\bf S}^2}B_\kappa(v-u,\si)\,d\si\lesssim |v-u|^{\gamma}\int_{2^{-\kappa-1}|v-u|^{-1}}^{2^{-\kappa}|v-u|^{-1}}\theta^{-1-2s} d\theta\lesssim 2^{2s\kappa}|v-u|^{\gamma+2s},
\end{equation}
where  we recall \eqref{def.B.ka} and the geometric relation $|v'-v|=|v-u|\sin \frac{\theta}{2}$.

On the region $|v-u|\geq1$, the singularity of $|v-u|^{\gamma+2s}$ is absent.
Thus, by Cauchy-Schwarz's inequality and (\ref{kernel estimates}),
we have
\begin{equation}\label{compute nonopp1}
\begin{split}
\left|\mathscr{T}_+^{\kappa,\ell}(h_1,h_2,h_3)\right|\lesssim&
2^{2s\kappa}\left(\int_{{\bf R}^3}du\int_{{\bf R}^3}dv\,|v-u|^{2(\gamma+2s)}
\langle v'\rangle^{-\gamma-2s}\mu^{1/4}(u') h_3(v')w^{2\ell}(v')\right)^{\frac{1}{2}}\\
&\times\left(\int_{{\bf R}^3}du\int_{{\bf R}^3}dv\, h_1(u)h_2(v)\langle v'\rangle^{\gamma+2s}\mu^{1/4}(u')w^{2\ell}(v')\right)^{\frac{1}{2}}\\
\lesssim&
2^{2s\kappa}\left(\int_{{\bf R}^3}dv'\,
\langle v'\rangle^{\gamma+2s}|h_3(v')|^2w^{2\ell}(v')\right)^{\frac{1}{2}}\\
&\times\left(\int_{{\bf R}^3}du\int_{{\bf R}^3}dv\, |h_1(u)|^2|h_2(v)|^2\langle v'\rangle^{\gamma+2s}\mu^{1/4}(u')w^{2\ell}(v')\right)^{\frac{1}{2}}.
\end{split}
\end{equation}
The first factor on the right-hand side of \eqref{compute nonopp1} is bounded by $\left|w^\ell
h_3\right|_{L^2_{\gamma+2s}}$. For the second factor, in the case when
$|v'|^2\leq\frac{1}{2}(|v|^2+|u|^2)$, since the collisional
conversation laws imply $\mu^{\frac{1}{4}}(u')\leq
\mu^{\frac{1}{8}}(v)\mu^{\frac{1}{8}}(u)$, it follows that $\langle
v'\rangle^{\gamma+2s}\mu^{1/4}(u')w^{2\ell}(v')\lesssim \langle
v\rangle^{-m}\langle u\rangle^{-m}$ for any nonnegative constant $m$. In the case when
$|v'|^2\geq\frac{1}{2}(|v|^2+|u|^2)$ which implies $|v'|^2\sim
|v|^2+|u|^2$, we have for $\ell\geq0$,
$$
w^{2\ell}(v')\lesssim w^{2\ell}(u)+w^{2\ell}(v),
$$
and similarly $\langle v'\rangle^{\gamma+2s}\lesssim\langle v\rangle^{\gamma+2s}$, so that
$$
\langle v'\rangle^{\gamma+2s}\mu^{1/4}(u')w^{2\ell}(v')\lesssim \langle v\rangle^{\gamma+2s}\left(w^{2\ell}(u)+w^{2\ell}(v)\right).
$$
Therefore, in both cases, the second factor on the right-hand side of \eqref{compute nonopp1}  is bounded by
$$
\left|h_1\right|_{L^2}\left|w^{\ell}h_2\right|_{L^2_{\gamma+2s}}
+
\left|w^{\ell}h_1\right|_{L^2}\left|h_2\right|_{L^2_{\gamma+2s}}.
$$
On the remaining region $|v-u|\leq1$, where
$\mu^{1/4}(u')w^{2\ell}(v')\lesssim \mu^{\delta}(u)\mu^{\delta}(v)$
for some $0<\delta<1$, it follows that
\begin{equation*}
\begin{split}
\left|\mathscr{T}_+^{\kappa,\ell}(h_1,h_2,h_3)\right|\lesssim&
2^{2s\kappa}\int_{{\bf R}^3}du\int_{{\bf R}^3}dv\,|v-u|^{\gamma+2s}|h_1(u)h_2(v)||h_3(v')|\\
&\times\left(\mu(u')\mu(v')\mu(u)\mu(v)\right)^{\frac{\delta}{2}},
\end{split}
\end{equation*}
which from Cauchy-Schwarz and Sobolev inequalities, is further bounded by
\begin{equation*}
\begin{split}
&2^{2s\kappa}\left(\int_{{\bf R}^3}dv'\,
\langle v'\rangle^{\gamma+2s}|h_3(v')|^2\mu^{\delta}(v')\right)^{\frac{1}{2}}\\
&\qquad \times\left|h_1(u)\mu^{\frac{\delta}{2}}(u)\right|_{\infty}\left(\int_{{\bf R}^3}du\int_{{\bf R}^3}dv\, \langle v\rangle^{\gamma+2s}|h_2(v)|^2\mu^{\delta}(v)\right)^{\frac{1}{2}}\\
&\lesssim
2^{2s\kappa}
\left|w^{-m}h_1\right|_{H^{2}}\left|w^{\ell}h_2\right|_{L^2_{\gamma+2s}}
\left|w^{\ell}h_3\right|_{L^2_{\gamma+2s}}.
\end{split}
\end{equation*}
 This completes the proof of Proposition \ref{prop.non}.
\end{proof}

\begin{proof}[Proof of Lemma \ref{estimates on nonop}]
In terms of series of estimates obtained in Propositions \ref{prop.non}, by applying the cancellation inequalities constructed in
\cite[Propositions 4.5-4.7]{GR} and carrying out the similar procedure as that of \cite[Lemma 6.1]{GR}, one can prove \eqref{nonop1} and \eqref{nonop2} and the details are omitted for brevity. This completes the proof of  Lemma  \ref{estimates on nonop}.
\end{proof}

\section{Nonlinear estimates}\label{sec.3}

The goal of this section is to make the weighted energy estimates on   those nonlinear terms in \eqref{perturbed eqn}.
The following decomposition will be frequently used in the later proofs:
\begin{equation}\label{nonopsplit}
\begin{split}
\Gamma(f,f)
=&\Gamma({\bf P}f,{\bf P}f)+\Gamma({\bf P}f,\{{\bf I}-{\bf P}\}f)+\Gamma(\{{\bf I}-{\bf P}\}f,{\bf P}f)\\
&
+\Gamma(\{{\bf I}-{\bf P}\}f,\{{\bf I}-{\bf P}\}f).
\end{split}
\end{equation}
Recall \eqref{def.wab} for $w_l(\al,\be)$, and recall \eqref{def e} and \eqref{def d} for $\CE_l(t)$ and $\CD_l(t)$, respectively. To the end we always suppose $-3<\ga<-2s$, $1/2\leq s<1$, and $K\geq 8$.
The first three lemmas of this section concern the estimates on the nonlinear term $\Ga(f,f)$.

\begin{Lemma}\label{lem.non.ga}
Let $l\geq 0$, $|\alpha|+|\beta|\leq K$. It holds that
\begin{equation}\label{nonop3}
\begin{split}
\left|\left(\Gamma_{\pm}(f,f),
f_{\pm}\right)\right|\lesssim \mathcal {E}^{\frac{1}{2}}_{l}(t)\mathcal {D}_{l}(t),
\end{split}
\end{equation}
\begin{equation}\label{nonop5}
\begin{split}
\left|\left(w^2_l(\alpha,\beta)\partial^{\alpha}_{\beta}\Gamma_{\pm}(f,f),
\partial^{\alpha}_{\beta}({\bf I}_{\pm}-{\bf P}_{\pm})f\right)\right|\lesssim\mathcal {E}^{\frac{1}{2}}_{l}(t)\mathcal {D}_{l}(t),
\end{split}
\end{equation}
\begin{equation}\label{nonop6}
\begin{split}
\left|\left(w^2_l(\alpha,0)\partial^{\alpha}\Gamma_{\pm}(f,f),
\partial^{\alpha}f_{\pm}\right)\right|\lesssim\mathcal {E}^{\frac{1}{2}}_{l}(t)\mathcal {D}_{l}(t).
\end{split}
\end{equation}
Here and hereafter, we denote  $I_{\pm}f=f_{\pm}$.
\end{Lemma}

\begin{proof}
For brevity, we only prove \eqref{nonop5} in the case when $\al=\be=0$,  and the other two estimates \eqref{nonop3} and \eqref{nonop6} can be proved in the similar way.
For this, we set
$$
J_1=\left|\left(w_l^{2}(0,0)\Gamma_{\pm}(f,f),
({\bf I}_{\pm}-{\bf P}_{\pm})f\right)\right|,
$$
and denote $J_{1,1}, J_{1,2}, J_{1,3}, J_{1,4}$  to be the terms corresponding to the decomposition (\ref{nonopsplit}).
Now we turn to estimate these terms one by one. First, for $J_{1,1}$, recalling (\ref{macro def}) and applying Lemma \ref{estimates on nonop} with $\ell=l+K$,
one has
$$
J_{1,1}\lesssim
\|(a_{\pm},b,c)\|_{H^1}\|\nabla_x(a_{\pm},b,c)\left\|w_l(0,0)({\bf
I}_{\pm}-{\bf P}_{\pm})f\right\|_{N^{s}_{\gamma}},
$$
where we have used Sobolev's inequalities
$$
\|(a_{\pm},b,c)\|_{L^3}\lesssim \|(a_{\pm},b,c)\|_{H^1},\ \
\|(a_{\pm},b,c)\|_{L^6}\lesssim \|\nabla_x(a_{\pm},b,c)\|.
$$
For $J_{1,2}$, by using (\ref{nonop1}) in Lemma 2.2 with $\ell=l+K$, it follows that
$$
J_{1,2}\lesssim \|(a_{\pm},b,c)\|_{L^\infty}\left\|w_l(0,0)({\bf
I}_{\pm}-{\bf P}_{\pm})f\right\|^2_{N^{s}_{\gamma}}\lesssim
\|\nabla_x(a_{\pm},b,c)\|_{H^1}\left\|w_l(0,0)({\bf I}_{\pm}-{\bf
P}_{\pm})f\right\|^2_{N^{s}_{\gamma}}.
$$
In the same way, $J_{1,3}$ has the same bound as $J_{1,2}$. Finally, for $J_{1,4}$, due to Lemma \ref{estimates on nonop} and Sobolev's inequality, we obtain
\begin{equation*}
\begin{split}
J_{1,4}\lesssim& \left\|\left|({\bf I}_{\pm}-{\bf P}_{\pm})f\right|_{L^2}\right\|_{L^\infty_x}
\left\|w_l(0,0)({\bf I}_{\pm}-{\bf P}_{\pm})f\right\|^2_{N^{s}_{\gamma}}\\
&+\left\|w_l(0,0)({\bf I}_{\pm}-{\bf
P}_{\pm})f\right\|\left\|\left|({\bf I}_{\pm}-{\bf
P}_{\pm})f\right|_{N^{s}_{\gamma}}\right\|_{L^\infty_x}
\left\|w_l(0,0)({\bf I}_{\pm}-{\bf P}_{\pm})f\right\|_{N^{s}_{\gamma}}\\
&+{\left\|\left|w^{-m}({\bf I}_{\pm}-{\bf P}_{\pm})f\right|_{H^2}\right\|_{L^{\infty}_{x}}}
\left\|w_l(0,0)({\bf I}_{\pm}-{\bf P}_{\pm})f\right\|^2_{N^{s}_{\gamma}}\\
\lesssim& \sum\limits_{|\alpha|\leq2}
\left\|\partial^\alpha({\bf I}_{\pm}-{\bf P}_{\pm})f\right\|
\left\|w_l(0,0)({\bf I}_{\pm}-{\bf P}_{\pm})f\right\|^2_{N^{s}_{\gamma}}\\
&+\left\|w_l(0,0)({\bf I}_{\pm}-{\bf P}_{\pm})f\right\|\sum\limits_{|\alpha|\leq2}
\left\|\partial^\alpha({\bf I}_{\pm}-{\bf P}_{\pm})f\right\|_{N^{s}_{\gamma}}
\left\|w_l(0,0)({\bf I}_{\pm}-{\bf P}_{\pm})f\right\|_{N^{s}_{\gamma}}\\
&+\sum\limits_{|\alpha_1|\leq2, |\beta_1|\leq2}
\left\|w_l({\alpha_1,\beta_1})\partial^{\alpha_1}_{\beta_1}({\bf
I}_{\pm}-{\bf P}_{\pm})f\right\| \left\|w_l(0,0)({\bf I}_{\pm}-{\bf
P}_{\pm})f\right\|^2_{N^{s}_{\gamma}}.
\end{split}
\end{equation*}
Now the desired estimate \eqref{nonop5} in the case when $\al=\be=0$ holds by combing all the above estimates. This completes the proof of Lemma \ref{lem.non.ga}.
\end{proof}

%The following lemma will be used in the late proof.

\begin{Lemma}\label{estimates on nonop2}
Let $\zeta(v)$ be a smooth function satisfying (\ref{expo fun}), and let $|\alpha|\leq K$. Writing
$$
\partial^{\alpha}\Gamma(f,f)=\sum\limits_{\alpha_{1}+\alpha_{2}=\alpha}
\Gamma(\partial^{\alpha_{1}}f,\partial^{\alpha_{2}}f),
$$
one has
\begin{equation}\label{nonop estimates3}
\left\|\int\Gamma(\partial^{\alpha_{1}}f,\partial^{\alpha_{2}}f)\zeta(v)\,
dv\right\|\lesssim\mathcal {E}^{\frac{1}{2}}_{l}(t)\mathcal {D}^{\frac{1}{2}}_{l}(t).
\end{equation}
\end{Lemma}

\begin{proof}
With Lemma 2.2 in hand, (\ref{nonop estimates3}) can be verified directly by applying Sobolev's embedding inequalities, and details are omitted for brevity.
\end{proof}

%\begin{Lemma}[\cite{Strain11}]
%For any $\ell\geq0$, it holds that
%\begin{equation}\label{nonop square}
%\sum\limits_{|\alpha|\leq 1}\left|w^{\ell}\partial^{\alpha}\Gamma(g,g)\right|_{L^2}\lesssim \sum\limits_{|\alpha|\leq 1}
%\left|w^{\ell}\partial^{\alpha}g\right|_{L^2_{\gamma+2s}}\sum\limits_{|\alpha|\leq 1}\left|w^{\ell}\partial^{\alpha}g\right|_{H^{4}_{\gamma+2s}}.
%\end{equation}
%\end{Lemma}

\begin{Lemma}\label{lem.non.z1}
Let $\ell\geq0$. It holds that
\begin{equation}\label{nonop square}
\sum\limits_{|\alpha|\leq 1}\left|w^{\ell}\partial^{\alpha}\Gamma(f,f)\right|_{L^2}\lesssim \sum\limits_{|\alpha|\leq 1}
\left|w^{\ell}\partial^{\alpha}f\right|_{L^2_{\gamma+2s}}\sum\limits_{|\alpha|\leq 1}\left|w^{\ell}\partial^{\alpha}f\right|_{H^{4}_{\gamma+2s}}.
\end{equation}
Moreover,
\begin{equation}\label{L2+Z1}
\left\|w^{\ell}\Gamma(f,f)\right\|_{H_x^1}+\left\|w^{\ell}\Gamma(f,f)\right\|_{Z_1}
\lesssim \sum\limits_{|\alpha|+|\beta|\leq 5}\left\|w^{\ell-\frac{\gamma+2s}{2\gamma}}\partial^{\alpha}_{\beta}f\right\|^2.
\end{equation}
\end{Lemma}

\begin{proof}
First of all, \eqref{nonop square} has been proved in
\cite[{pp.21, Proposition 3.1}]{Strain11}. {Then applying \eqref{nonop square}} and Sobolev's
inequality, we have
\begin{equation*}
\begin{split}
\left\|w^{\ell}\Gamma(f,f)\right\|^2_{H_x^1}\lesssim& \int_{{\bf R}^3} \sum\limits_{|\alpha|\leq 1} \left|w^{\ell}\partial^{\alpha}f\right|^2_{L^2_{\gamma+2s}}\sum\limits_{|\alpha|\leq 1}\left|w^{\ell}\partial^{\alpha}f\right|^2_{H^{4}_{\gamma+2s}}dx\\
\lesssim&\sup\limits_{x}\sum\limits_{|\alpha|\leq 1} \left|w^{\ell}\partial^{\alpha}f\right|^2_{L^2_{\gamma+2s}}\int_{{\bf R}^3} \sum\limits_{|\alpha|\leq 1}\left|w^{\ell}\partial^{\alpha}f\right|^2_{H^{4}_{\gamma+2s}}dx
\\
\lesssim&\sum\limits_{|\alpha'|\leq 3} \left\|w^{\ell}\partial^{\alpha'}f\right\|^2_{L^2_{\gamma+2s}} \sum\limits_{|\alpha|\leq 1}\left\|w^{\ell}\partial^{\alpha}f\right\|^2_{H^{4}_{\gamma+2s}}.
\end{split}
\end{equation*}
Therefore $\left\|w^{\ell}\Gamma(f,f)\right\|_{H_x^1}$ is bounded by the right-hand term of (\ref{L2+Z1}).
As in \cite{Strain11}, $\left\|w^{\ell}\Gamma(f,f)\right\|_{Z_1}$ can be estimated in the completely same way as for $\left\|w^{\ell}\Gamma(f,f)\right\|_{H_x^1}$, and details are omitted for brevity. This completes the proof of Lemma \ref{lem.non.z1}.
\end{proof}

The following two lemmas concern the estimates on  $v\cdot\nabla_x\phi  f_{\pm}$. As in \cite{Guo5}, for simplicity, we use $e_i$ to denote the multi-index with the $i$th element unit and the rest ones zeros.

\begin{Lemma}\label{lem.vphi.x}
Let $1\leq|\alpha|\leq K$, and $l\geq0$. Suppose $\sqrt{\mathcal {E}_{l}(t)}< \de$
for a constant $\de>0$. Then, it holds that
\begin{equation}\label{vdx}
\sum\limits_{1\leq|\alpha_1|\leq|\alpha|}\left(v_i\partial^{\alpha_1+e_i}
\phi\partial^{\alpha-\alpha_1}f_{\pm},e^{\pm\phi}w_l^2(\alpha,0)\partial^{\alpha}f_{\pm}\right)\lesssim \mathcal {E}^{\frac{1}{2}}_{l}(t)\mathcal {D}_{l}(t).
\end{equation}
\end{Lemma}

\begin{proof}
In terms of $f_{\pm}=({\bf I}_{\pm}-{\bf P}_{\pm})f+{\bf P}_{\pm}f$, set
\begin{equation*}
\begin{split}
\sum\limits_{1\leq|\alpha_1|\leq|\alpha|}&\left(v_i\partial^{\alpha_1+e_i}
\phi\partial^{\alpha-\alpha_1}{\bf P}_{\pm}f,e^{\pm\phi}w_l^2(\alpha,0)\partial^{\alpha}f_{\pm}\right)\\
&+\sum\limits_{1\leq|\alpha_1|\leq|\alpha|}\left(v_i\partial^{\alpha_1+e_i}
\phi\partial^{\alpha-\alpha_1}({\bf I}_{\pm}-{\bf P}_{\pm})f,e^{\pm\phi}w_l^2(\alpha,0)\partial^{\alpha}{\bf P}_{\pm}f\right)\\
&+
\sum\limits_{1\leq|\alpha_1|\leq|\alpha|}\left(v_i\partial^{\alpha_1+e_i}
\phi\partial^{\alpha-\alpha_1}({\bf I}_{\pm}-{\bf P}_{\pm})f,e^{\pm\phi}w_l^2(\alpha,0)\partial^{\alpha}({\bf I}_{\pm}-{\bf P}_{\pm})f\right)\\ =&J_{2,1}+J_{2,2}+J_{2,3}.
\end{split}
\end{equation*}
Notice by Sobolev inequality that
\begin{equation}\label{force bdd}
|\phi|_{L^\infty}\lesssim \|\nabla_x\phi\|_{H^1}\lesssim \sqrt{\mathcal {E}_{l}(t)}< \de.
\end{equation}
For $J_{2,1}$,  one has
$$
|J_{2,1}|\lesssim \sum\limits_{1\leq|\alpha_1|\leq|\alpha|}\int_{{\bf R}^3}|\partial^{\alpha_1+e_i}
\phi||\partial^{\alpha-\alpha_1}(a_{\pm},b,c)|\partial^\alpha f_{\pm}|_{L^2_{\gamma}}dx,
$$
where for $|\alpha_1|\leq \left[K/2\right]$, the integrand is
bounded by
$$
\sup\limits_{x}|\partial^{\alpha_1+e_i}
\phi|\cdot \|\partial^{\alpha-\alpha_1}(a_{\pm},b,c)\|\cdot \|\partial^\alpha f_{\pm}\|_{L^2_{\gamma}},
$$
while for $|\alpha_1|\geq \left[K/2\right]+1$, it is bounded by
$$
\sup\limits_{x}|\partial^{\alpha-\alpha_1}(a_{\pm},b,c)|\cdot \|\partial^{\alpha_1+e_i}
\phi\|\cdot \|\partial^\alpha f_{\pm}\|_{L^2_{\gamma}}.
$$
Thus, by Sobolev inequality,
$$
|J_{2,1}|\lesssim  \|\nabla_x^2\phi\|_{H^{K-1}}\sum\limits_{1\leq|\alpha|\leq K}\left\{\|\partial^\alpha (a_{\pm},b,c)\|\|\partial^\alpha f_{\pm}\|_{L^2_{\gamma}}\right\}.
$$
Similarly, for $J_{2,2}$, one has
$$
\begin{array}{rlll}
\begin{split}
|J_{2,2}|\lesssim& \sum\limits_{1\leq|\alpha_1|\leq|\alpha|}\int_{{\bf R}^3}|\partial^{\alpha_1+e_i}
\phi|\left|\partial^{\alpha-\alpha_1}({\bf I}_{\pm}-{\bf P}_{\pm})f\right|_{L^2_{\gamma}}|\partial^\alpha (a_{\pm},b,c)|dx\\
\lesssim& \sum\limits_{|\alpha_1|\leq \left[K/2\right]}\sup\limits_{x}|\partial^{\alpha_1+e_i}
\phi|\cdot \|\partial^{\alpha-\alpha_1}({\bf I}_{\pm}-{\bf P}_{\pm})f\|_{L^2_{\gamma}}
\cdot \|\partial^\alpha (a_{\pm},b,c)\|\\
&+\sum\limits_{|\alpha_1|\geq \left[K/2\right]+1}\sup\limits_{x}|\partial^{\alpha-\alpha_1}({\bf I}_{\pm}-{\bf P}_{\pm})f|_{L^2_{\gamma}}\cdot \|\partial^{\alpha_1+e_i}
\phi\|
\cdot \|\partial^\alpha (a_{\pm},b,c)\|\\
\lesssim&
\|\nabla_x^2\phi\|_{H^{K-1}}\sum\limits_{1\leq|\alpha|\leq K}\|\partial^\alpha (a_{\pm},b,c)\|
\sum\limits_{|\alpha|\leq K}\|\partial^\alpha ({\bf I}_{\pm}-{\bf P}_{\pm})f\|_{L^2_{\gamma}}.
\end{split}
\end{array}
$$
Finally, for $J_{2,3}$, since $1/2\leq s <1$, it is straightforward to see
$$|w_l(\alpha,0)v_i|\lesssim w_l(\alpha-e_i,0)\langle v\rangle^{\gamma+2s}.$$
Therefore, one has
$$
\begin{array}{rlll}
\begin{split}
|J_{2,3}|\lesssim& \sum\limits_{1\leq|\alpha_1|\leq|\alpha|}\int_{{\bf R}^3}|\partial^{\alpha_1+e_i}
\phi|\cdot \left|w_l(\alpha-e_i,0)\partial^{\alpha-\alpha_1}({\bf I}_{\pm}-{\bf P}_{\pm})f\right|_{L^2_{\gamma+2s}}\\ &\times\left|w_l(\alpha,0)\partial^{\alpha}({\bf I}_{\pm}-{\bf P}_{\pm})f\right|_{L^2_{\gamma+2s}}dx\\
\lesssim& \sum\limits_{1\leq |\alpha_1|\leq \left[K/2\right]}\sup\limits_{x}|\partial^{\alpha_1+e_i}
\phi|\cdot \left\|w_l(\alpha-e_i,0)\partial^{\alpha-\alpha_1}({\bf I}_{\pm}-{\bf P}_{\pm})f\right\|_{L^2_{\gamma+2s}}\\
&\qquad\qquad\qquad \times\left\|w_l(\alpha,0)\partial^{\alpha}({\bf I}_{\pm}-{\bf P}_{\pm})f\right\|_{L^2_{\gamma+2s}}\\
&+\sum\limits_{|\al|\geq |\alpha_1|\geq \left[K/2\right]+1}\sup\limits_{x}\left|w_l(\alpha-e_i,0)\partial^{\alpha-\alpha_1}({\bf I}_{\pm}-{\bf P}_{\pm})f\right|_{L^2_{\gamma+2s}}\cdot \|\partial^{\alpha_1+e_i}
\phi\|
\\ &
\qquad\qquad\qquad \times\left\|w_l(\alpha,0)\partial^{\alpha}({\bf I}_{\pm}-{\bf P}_{\pm})f\right\|_{L^2_{\gamma+2s}}\\
\lesssim&
\|\nabla_x^2\phi\|_{H^{K-1}}
\sum\limits_{|\alpha|\leq K}\left\|w_l(\alpha,0)\partial^\alpha ({\bf I}_{\pm}-{\bf P}_{\pm})f_{\pm}\right\|^2_{L^2_{\gamma+2s}},
\end{split}
\end{array}
$$
where we have used $|\al-\al_1|\leq |\al-e_i|$ in the case when $1\leq |\alpha_1|\leq \left[K/2\right]$, and $|\alpha-\alpha_1|+2\leq |\alpha-e_i|$ in the case when  $|\al|\geq |\alpha_1|\geq \left[K/2\right]+1$. Collecting all the estimates, \eqref{vdx} follows.  This then completes the proof of Lemma \ref{lem.vphi.x}.
\end{proof}

\begin{Lemma}\label{lem.vphi.xv}
Let $1\leq|\alpha|+|\beta|\leq K$ with $|\beta|\geq1$, and $l\geq0$. Suppose $\sqrt{\mathcal {E}_{l}(t)}<\de$
for a constant $\de>0$. Then, it holds that
\begin{equation}\label{vdxv}
\sum\limits_{|\alpha_1|+|\beta_1|\geq1\atop{|\alpha_1|\leq|\alpha|,|\beta_1|\leq1}}
\left(\partial_{\beta_1}v_i\partial^{\alpha_1+e_i}
\phi\partial_{\beta-\beta_1}^{\alpha-\alpha_1}({\bf I}_{\pm}-{\bf P}_{\pm})f,e^{\pm\phi}w_l^2(\alpha,\beta)\partial_\beta^{\alpha}({\bf I}_{\pm}-{\bf P}_{\pm})f\right)\lesssim \mathcal {E}^{\frac{1}{2}}_{l}(t)\mathcal {D}_{l}(t).
\end{equation}
\end{Lemma}

\begin{proof}
For brevity, we denote the left-hand term of (\ref{vdxv}) as $J_3$. As in Lemma \ref{lem.vphi.x}, we
prove (\ref{vdxv}) by considering the following two cases. For the case $|\alpha_1|\leq\left[K/2\right]$, from (\ref{force bdd}),
one has
$$
\begin{array}{rlll}
\begin{split}
|J_3|\lesssim& \sup\limits_{x}|\partial^{\alpha_1+e_i}\phi|\cdot \left
\|w_l(\alpha-\alpha_1,\beta-\beta_1)\partial_{\beta-\beta_1}^{\alpha-\alpha_1}({\bf I}_{\pm}-{\bf P}_{\pm})f\right\|_{L^2_{\gamma+2s}}\\ &\times\left\|w_l(\alpha,\beta)\partial_{\beta}^{\alpha}({\bf I}_{\pm}-{\bf P}_{\pm})f\right\|_{L^2_{\gamma+2s}},
\end{split}
\end{array}
$$
which is further bounded by the right-hand term of (\ref{vdxv}). For the case $|\alpha_1|\geq\left[K/2\right]+1$,
it is similar to verify
$$
\begin{array}{rlll}
\begin{split}
|J_3|\lesssim& \sup\limits_{x}\left|w_l(\alpha-\alpha_1,\beta-\beta_1)\partial_{\beta-\beta_1}^{\alpha-\alpha_1}({\bf I}_{\pm}-{\bf P}_{\pm})f\right|_{L^2_{\gamma+2s}}\|\partial^{\alpha_1+e_i}\phi\|
\\ &\times\left\|w_l(\alpha,\beta)\partial_{\beta}^{\alpha}({\bf I}_{\pm}-{\bf P}_{\pm})f\right\|_{L^2_{\gamma+2s}},
\end{split}
\end{array}
$$
which is also bounded by the right-hand term of (\ref{vdxv}). This completes the proof of Lemma \ref{lem.vphi.xv}.
\end{proof}

The following two lemmas concern the estimates on  $\nabla_x \phi\cdot\nabla_v f_{\pm}$.

\begin{Lemma}\label{lem.phi.x}
Let $1\leq|\alpha|\leq K$, and $l\geq0$. Suppose  $\sqrt{\mathcal
{E}_{l}(t)}< \de$ for a constant $\de>0$. Then, it holds that
\begin{equation}\label{vdx2}
\sum\limits_{|\alpha_1|\leq|\alpha|}\left(\partial^{\alpha_1+e_i}
\phi\partial_{e_i}^{\alpha-\alpha_1}f_{\pm},e^{\pm\phi}w_l^2(\alpha,0)\partial^{\alpha}f_{\pm}\right)\lesssim \mathcal {E}^{\frac{1}{2}}_{l}(t)\mathcal {D}_{l}(t).
\end{equation}
\end{Lemma}

\begin{proof}
Denote the left-hand term of (\ref{vdx2}) by $J_4$. When $\alpha_1=0$, by taking integration by parts with respect to $v_i$, one has
$$
\begin{array}{rlll}
\begin{split}
J_4\lesssim& \sup\limits_{x}|\partial^{e_i}
\phi|\int_{{\bf R}^3}\int_{{\bf R}^3}|\partial^{\alpha}f_{\pm}|^2\left|\partial_{e_i}w_l^2(\alpha,0)\right|dxdv\\[2mm]
\lesssim& {\left\|\nabla_x\phi\right\|_{H^{2}}\left\|w_l(\alpha,0)\partial^{\alpha}f_{\pm}\right\|
\left\|w_l(\alpha,0)\partial^{\alpha}f_{\pm}\right\|_{L^{2}_{\gamma+2s}}\lesssim\mathcal {E}^{\frac{1}{2}}_{l}(t)\mathcal {D}_{l}(t),}
\end{split}
\end{array}
$$  
where the inequality $|\partial_{e_i}w_l^2(\alpha,0)|\lesssim \langle v\rangle^{\frac{\gamma+2s}{2}}w_l^2(\alpha,0)$ has been used due to $-3<\ga<-2s$ and $1/2\leq s<1$.

Whenever $\alpha_1>0$, one can write
$$
J_4=J_{4,1}+J_{4,2},
$$
with
$$
J_{4,1}=\sum\limits_{1\leq|\alpha_1|\leq|\alpha|}\left(\partial^{\alpha_1+e_i}
\phi\partial_{e_i}^{\alpha-\alpha_1}{\bf P}_{\pm}f,e^{\pm\phi}w_l^2(\alpha,0)\partial^{\alpha}f_{\pm}\right),
$$
$$
J_{4,2}=\sum\limits_{1\leq|\alpha_1|\leq|\alpha|}\left(\partial^{\alpha_1+e_i}
\phi\partial_{e_i}^{\alpha-\alpha_1}({\bf I}_{\pm}-{\bf P}_{\pm})f,e^{\pm\phi}w_l^2(\alpha,0)\partial^{\alpha}f_{\pm}\right).
$$
We now estimate $J_{4,1}$ and $J_{4,2}$ as follows. By using the similar argument as for estimating $J_{2,1}$, we deduce that
$$
J_{4,1}\lesssim
\|\nabla_x^2\phi\|_{H^{K-1}}\sum\limits_{1\leq|\alpha|\leq K}\left\{\|\partial^\alpha (a_{\pm},b,c)\|\cdot \|\partial^\alpha f_{\pm}\|_{L^2_{\gamma}}\right\}.
$$
To estimate $J_{4,2}$, we use the trick as in \cite{Guo5}. First of all, notice that $J_{4,2}$ can be written as
$$
\begin{array}{rlll}
\begin{split}
J_{4,2}=&\sum\limits_{1\leq|\alpha_1|\leq|\alpha|}\bigg\{\left(\partial^{\alpha_1+e_i}
\phi\partial_{e_i}\left[w_l(\alpha-\alpha_1,0)w^{-\frac{|\alpha_1|}{2}}\partial^{\alpha-\alpha_1}({\bf I}_{\pm}-{\bf P}_{\pm})f\right],e^{\pm\phi}w_l(\alpha,0)w^{-\frac{|\alpha_1|}{2}}\partial^{\alpha}f_{\pm}\right)\\
&-\left(\partial^{\alpha_1+e_i}
\phi\partial_{e_i}\left[w_l(\alpha-\alpha_1,0)w^{-\frac{|\alpha_1|}{2}}\right]
w^{-\frac{|\alpha_1|}{2}}\partial^{\alpha-\alpha_1}({\bf I}_{\pm}-{\bf P}_{\pm})f,e^{\pm\phi}w_l(\alpha,0)\partial^{\alpha}f_{\pm}\right)\bigg\}
\\ =&J^{(1)}_{4,2}+J^{(2)}_{4,2}.
\end{split}
\end{array}
$$
For the term $J^{(2)}_{4,2}$, it is straightforward to estimate it by
$$
\begin{array}{rlll}
\begin{split}
C\int_{{\bf R}^3}&\left|\partial^{\alpha_1+e_i}\phi\right|\left|w_l(\alpha-\alpha_1,0)\partial^{\alpha-\alpha_1}({\bf I}_{\pm}-{\bf P}_{\pm})f\right|_{L^2_{\gamma+2s}}\left|w_l(\alpha,0)\partial^{\alpha}f_{\pm}\right|_{L^2_{\gamma+2s}}dx
\\
\lesssim& \|\nabla_x^2\phi\|_{H^{K-1}}\sum\limits_{|\alpha|\leq K}\left\|w_l(\alpha,0)\partial^\alpha ({\bf I}_{\pm}-{\bf P}_{\pm})f\right\|^2_{L^2_{\gamma+2s}}\sum\limits_{1\leq|\alpha|\leq K}\left\|w_l(\alpha,0)\partial^\alpha f_{\pm}\right\|^2_{L^2_{\gamma+2s}}.
\end{split}
\end{array}
$$
For $J^{(1)}_{4,2}$, noticing $1/2\leq s<1$, by the Parseval
 identity, one has
$$
\begin{array}{rlll}
\begin{split}
J^{(1)}_{4,2}\lesssim&{\int_{{\bf R}^3}\left|\int_{{\bf R}^3}
\mathbbm{i}\xi_i
\CF_v\left[w_l(\alpha-\alpha_1,0)w^{-\frac{|\alpha_1|}{2}}
\partial^{\alpha-\alpha_1}({\bf I}_{\pm}-{\bf P}_{\pm})f\right]
\overline{\CF_v\left[w_l(\alpha,0)w^{-\frac{|\alpha_1|}{2}}\partial^{\alpha}f_{\pm}\right]}d\xi\right|}
\\
&\qquad \times\left|\partial^{\alpha_1+e_i}\phi\right|dx
\\
\lesssim&{\int_{{\bf
R}^3}\left|\langle\xi\rangle^{\frac{1}{2}}
\CF_v\left[w_l(\alpha-\alpha_1,0)w^{-\frac{|\alpha_1|}{2}}
\partial^{\alpha-\alpha_1}({\bf I}_{\pm}-{\bf
P}_{\pm})f\right]\right|_{L^2_\xi}
\left|\langle\xi\rangle^{\frac{1}{2}}
\CF_v\left[w_l(\alpha,0)w^{-\frac{|\alpha_1|}{2}}\partial^{\alpha}f_{\pm}\right]\right|_{L^2_\xi}}
\\
&\qquad \times\left|\partial^{\alpha_1+e_i}\phi\right|dx
\\
\lesssim&\int_{{\bf R}^3}\left|w_l(\alpha-\alpha_1,0)w^{-\frac{|\alpha_1|}{2}}\partial^{\alpha-\alpha_1}({\bf I}_{\pm}-{\bf P}_{\pm})f\right|_{H^s_v}\left|w_l(\alpha,0)w^{-\frac{|\alpha_1|}{2}}\partial^{\alpha}f_{\pm}\right|_{H^s_v}
\\
&\qquad \times\left|\partial^{\alpha_1+e_i}\phi\right|dx
\\
\lesssim& \|\nabla_x^2\phi\|_{H^{K-1}}\sum\limits_{|\alpha|\leq
K}\left\|w_l(\alpha,0)\partial^\alpha ({\bf I}_{\pm}-{\bf
P}_{\pm})f\right\|^2_{H^{s}_{\gamma}}\sum\limits_{1\leq|\alpha|\leq
K}\left\|w_l(\alpha,0)\partial^\alpha
f_{\pm}\right\|^2_{H^{s}_{\gamma}},
\end{split}
\end{array}
$$
where $\CF_v$ is the Fourier transform with respect to $v$-variable, $\xi$ denotes the corresponding frequency variable, $\bar{\cdot}$ denotes the complex conjugate, and $\mathbbm{i}=\sqrt{-1}\in\mathbb{C}$ is the pure imaginary unit.
Collecting the estimates above, this completes the proof of Lemma \ref{lem.phi.x}.
\end{proof}

%For the case of higher order mixed derivative with respect to the velocity and spatial %variables, we have

\begin{Lemma}\label{lem.phi.xv}
Let $1\leq|\alpha|+|\beta|\leq K$ with $|\beta|\geq1$, and $l\geq0$. Suppose  $\sqrt{\mathcal {E}_{l}(t)}<\de$
for a constant $\de>0$. Then, it holds that
\begin{equation}\label{vdxv2}
\sum\limits_{\alpha_1\leq\alpha}\left(
\partial^{\alpha_1+e_i}\phi\partial_{\beta+e_i}^{\alpha-\alpha_1}({\bf I}_{\pm}-{\bf P}_{\pm})f,e^{\pm\phi}w_l^2(\alpha,\beta)\partial_\beta^{\alpha}({\bf I}_{\pm}-{\bf P}_{\pm})f\right)\lesssim \mathcal {E}^{\frac{1}{2}}_{l}(t)\mathcal {D}_{l}(t).
\end{equation}
\end{Lemma}

\begin{proof}
(\ref{vdxv2})  can be proved by using the same argument as for Lemma  \ref{lem.phi.x}, and the details are omitted for brevity.
\end{proof}

\section{The linearized system and time-decay}\label{sec.4}

In this section, we consider the Cauchy
problem on the linearized system with a nonhomogeneous source
$h=[h_+, h_-]$:
\begin{eqnarray}\label{LB eqn}
\left\{\begin{array}{rlll}
\begin{split}
&\partial_{t}f_{\pm}+v\cdot\nabla_{x}f_{\pm}\pm\sqrt{\mu}v\cdot \nabla_x\phi+Lf_{\pm}=h_{\pm},\\
&-\Delta_x\phi=\rho_f=\int_{{\bf R}^3}(f_+-f_-)\sqrt{\mu}\,dv,\ \ \phi\rightarrow0\ \ \text{as} \ |x|\rightarrow\infty, \\
&f_{\pm}|_{t=0}=f_{0,\pm},
\end{split}
\end{array}\right.
\end{eqnarray}
where $h_{\pm}(t,x,v)$, $f_{0,\pm}=f_{0,\pm}(x,v)$
are given. Notice that for the nonlinear Vlasov-Poisson-Boltzmann system
(\ref{perturbed eqn}) and (\ref{perturbed E}), the nonhomogeneous
source takes the form of
\begin{equation}\label{h VPB}
h_{\pm}=\pm \nabla_x\phi\cdot\nabla_v f_{\pm}\mp \frac{1}{2}\nabla_x\phi\cdot
vf_{\pm}+\Gamma_{\pm}(f,f),
\end{equation}
which satisfies the mass conservation laws
\begin{equation}\label{nonh cons}
\langle h_{\pm},\sqrt{\mu}\rangle=0.
\end{equation}

Whenever the linearized system is homogeneous, i.e.~$h=0$, the formal solution to the Cauchy problem (\ref{LB eqn}) can be
written as the mild form
\begin{equation}\label{formal sol}
f=S(t)f_0,
\end{equation}
where $S(t)$ denotes the solution operator for the Cauchy problem on
the linearized system without any source. For an integer $k\geq 0$, set the index $\sigma_k$ of the time-decay rate by
$$
\sigma_k=\frac{3}{4}+\frac{k}{2}.
$$
The time-decay property of $S(t)$ is stated in the following

\begin{Theorem}\label{basic decay lemma}
Recall $w=w(v)=\lag v\rag^{-\ga}$. Let $-3<\gamma<-2s$, $1/2\leq s<1$, $l\geq 0$ and $k\geq 0$ be an integer, and let $l_\ast>\sigma_{k}(\gamma+2s)/\gamma$. Assume that
\begin{equation*}
\int_{{\bf R}^3}\rho_{f_0}\,dx=0,\ \ \int_{{\bf R}^3}(1+|x|)|\rho_{f_0}|\,dx<\infty,
\end{equation*}
where
$$
\rho_{f_0}=\int_{{\bf R}^3}(f_{0,+}-f_{0,-})\sqrt{\mu}\,dv.
$$
Then, for $f(t)=S(t)f_0$, it holds that
\begin{equation}\label{basic decay}
\begin{split}
\left\|w^{l}\nabla_{x}^{k}f(t)\right\|+\|\nabla_{x}^{k}E(t)\| \lesssim
(1+t)^{-\sigma_{k}}\left(\left\|w^{l+l_\ast}f_{0}\right\|_{Z_1}+\left\|(1+|x|)\rho_{f_0}\right\|_{L^1}
+\left\|w^{l+l_\ast}\nabla_{x}^{k}f_0\right\|\right),
\end{split}
\end{equation}
for any $t\geq 0$.
\end{Theorem}

Before proving Theorem \ref{basic decay lemma}, we make some preparations as follows. Firstly, for the later use in the general situation,  let us consider the linearized system \eqref{LB eqn} with the nonhomogeneous source
$h$ satisfying \eqref{nonh cons}.
As in \cite{Guo2},  one can derive  the corresponding local conservation laws. In fact, from multiplying the first equation of \eqref{LB eqn} by
$\sqrt{\mu}$, $v_j\sqrt{\mu}$ $(j=1,2,3)$ and
$\frac{1}{6}\left(|v|^2-3\right)\sqrt{\mu}$ and then
integrating them over $v\in{\bf R}^3$, one has
\begin{equation}\label{conservation a}
\partial_ta_{\pm}+\nabla_{x}\cdot
b=-\nabla_x\cdot\langle v\sqrt{\mu}, ({\bf I}_{\pm}-{\bf
P}_{\pm})f\rangle,
\end{equation}
\begin{equation}\label{conservation b}
\begin{split}
\partial_t\left\{b_{j}+\langle v_j\sqrt{\mu}, ({\bf I}_{\pm}-{\bf
P}_{\pm})f\rangle\right\}&+\partial_{j}(a_{\pm}+2c)\mp
E_j\\
=&-\left\langle v\cdot\nabla_{x} ({\bf I}_{\pm}-{\bf
P}_{\pm})f,v_j\sqrt{\mu}\right\rangle +\left\langle
h_{\pm}-L_{\pm},v_j\sqrt{\mu}\right\rangle,
\end{split}
\end{equation}
\begin{equation}\label{conservation c}
\begin{split}
\partial_t\left\{c+\frac{1}{6}\left\langle \left(|v|^2-3\right)\sqrt{\mu}, ({\bf I}_{\pm}-{\bf
P}_{\pm})f\right\rangle\right\}+\frac{1}{3}\nabla_{x}\cdot b=&
-\frac{1}{6}\left\langle v\cdot\nabla_{x}({\bf I}_{\pm}-{\bf
P}_{\pm})f,(|v|^2-3)\sqrt{\mu}\right\rangle\\ &+\frac{1}{6}\left\langle
h_{\pm}-L_{\pm},(|v|^2-3)\sqrt{\mu}\right\rangle.
\end{split}
\end{equation}
Moreover, we need the equations of high-order moments. For that, as in  \cite{Duan, DuanS1}, we define the high-order moment functions
$\Xi(f_{\pm})=\left(\Xi_{jk}(f_{\pm})\right)_{3\times3}$ and
$\Pi(f_{\pm})=\left(\Pi_{1}(f_{\pm}), \Pi_{2}(f_{\pm}),
\Pi_{3}(f_{\pm})\right)$ by
$$
\Xi_{jk}(f_{\pm})=\langle(v_jv_k-1)\sqrt{\mu},f_{\pm}\rangle, \ \ \
\Pi_j(f_{\pm})=\frac{1}{10}\langle(|v|^2-5)v_j\sqrt{\mu},f_{\pm}\rangle.
$$
Then, by taking the velocity integrations of the first equation of \eqref{LB eqn} with respect to
the above high-order moments, one has
\begin{equation}\label{h momemt1}
\partial_t\{\Xi_{jj}(\{{\bf I}_{\pm}-{\bf
P}_{\pm}\}f)+2c\}+2\partial_jb_j=\Xi_{jj}(r_{\pm}+h_{\pm}),
\end{equation}
\begin{equation}\label{h momemt2}
\begin{split}
\partial_t\Xi_{jk}(\{{\bf I}_{\pm}-{\bf
P}_{\pm}\}f)+\partial_jb_k+&\partial_kb_j+\nabla_x\cdot\langle
v\sqrt{\mu}, ({\bf I}_{\pm}-{\bf
P}_{\pm})f\rangle\\ &=\Xi_{jk}(r_{\pm}+h_{\pm})+\left\langle
h_{\pm},\sqrt{\mu}\right\rangle,\ \ j\neq k,
\end{split}
\end{equation}
\begin{equation}\label{h momemt3}
\partial_t\Pi_{j}(\{{\bf I}_{\pm}-{\bf
P}_{\pm}\}f)+\partial_jc=\Pi_{j}(r_{\pm}+h_{\pm}),
\end{equation}
with
$
r_{\pm}=-v\cdot\nabla_x\{{\bf I}_{\pm}-{\bf P}_{\pm}\}f+L_{\pm}f.
$
Here, notice that we have used (\ref{conservation a}) to derive (\ref{h momemt2}).

Consequently, as in \cite{DuanS2}, by taking the mean value of every two equations with
$\pm$ sign for (\ref{conservation a})-(\ref{conservation c}) and noticing $\langle h_{\pm},\sqrt{\mu}\rangle=0$, one has
\begin{eqnarray}\label{conservation abc}
\left\{
\begin{array}{rlll}
\begin{split}
&\partial_t\left(\frac{a_++a_-}{2}\right)+\nabla_{x}\cdot
b=0,\\
&\partial_tb_{j}+\partial_{j}\left(\frac{a_++a_-}{2}+2c\right)+\frac{1}{2}
\partial_j\Xi_{jk}(\{{\bf I}-{\bf P}\}f\cdot[1,1])=\frac{1}{2}\langle h_++h_-, v_j\sqrt{\mu}\rangle,\\
&\partial_tc+\frac{1}{3}\nabla_{x}\cdot
b+\frac{5}{6}\sum\limits_{j=1}^3\partial_j\Pi_{j}(\{{\bf I}-{\bf
P}\}f\cdot[1,1])=\frac{1}{12}\left\langle
h_{+}+h_{-},(|v|^2-3)\sqrt{\mu}\right\rangle,
\end{split}
\end{array} \right.
\end{eqnarray}
for $1\leq j\leq3$, and similarly, it follows from (\ref{h
momemt1})-(\ref{h momemt3}) that
\begin{eqnarray*}
\left\{
\begin{array}{rlll}
\begin{split}
&\partial_t\left\{\frac{1}{2}\Xi_{jk}(\{{\bf I}-{\bf
P}\}f\cdot[1,1])+2c\delta_{jk}\right\}+\partial_jb_k+\partial_kb_j=\frac{1}{2}\Xi_{jk}(r_++r_{-}+h_++h_-),\\
&\frac{1}{2}\partial_t\Pi_{j}(\{{\bf I}-{\bf
P}\}f\cdot[1,1])+\partial_jc=\frac{1}{2}\Pi_{j}(r_{+}+r_{-}+h_{+}+h_{-}),
\end{split}
\end{array} \right.
\end{eqnarray*}
for $1\leq j,k\leq3$, and $\delta_{jk}$ denoted the Kronecker delta.
Moreover, in order to obtain the dissipation rate of the electric field $E$, by taking the difference of two
equations with $\pm$ sign for (\ref{conservation a}) and
(\ref{conservation b}), one has
\begin{equation}\label{a-a}
\partial_t(a_+-a_-)+\nabla_x\cdot G=0,
\end{equation}
\begin{equation}\label{b-b}
\partial_tG+\nabla(a_+-a_-)-2E+\nabla_x\cdot \Xi(({\bf I-P})f\cdot
q_1)=\langle h-Lf,[v\sqrt{\mu},-v\sqrt{\mu}]\rangle,
\end{equation}
where
\begin{equation}\label{G}
G=\langle v\sqrt{\mu}, ({\bf I-P})f\cdot q_1\rangle.
\end{equation}
Notice that the second equation of \eqref{LB eqn} gives
\begin{equation}\label{ea}
\nabla_x\cdot E=a_+-a_-.
\end{equation}

%The strategy for the proof of Lemma 3.1 is to construct a
%time-frequency Lyapunov functional $\mathscr{E}_l(t,\xi)$
%corresponding to the Fourier transform of the system (\ref{LB eqn})
%such that the functional is not only equivalent with energy-type
%norm
%$$
%\left|w^l\widehat{f}\right|^2+|\widehat{E}|^2,
%$$
%but also its dissipation rate can be characterized by the functional
%itself; see Lemma 3.5. To construct the time-frequency Lyapunov
%functional as motioned above, we employ Fourier analysis to derive
%careful estimates of the macro dissipation, micro dissipation and
%the weak dissipation of the electric potential force. And the main
%results are given in the Lemmas 3.2, 3.3 and 3.4 below.

Now, we will prove Theorem \ref{basic decay lemma} by using as in
\cite{DuanS2} the following lemma  whose proof is omitted for
brevity. To state it,  in what follows we let $h$ be identical to
zero and let $f(t)=S(t)f_0$ be the solution to the Cauchy problem
\eqref{LB eqn} with $h\equiv 0$. Some more notations are given as
follows. For two complex vectors $z_1,z_2\in {\bf C}^{3}$,
$(z_1|z_2)=z_1\cdot \overline{z_2}$ denotes the dot product in the
complex field ${\bf C}$, where $\overline{z_2}$ is the complex
conjugate of $z_2$. $\hat{g}(\xi)$ denotes the Fourier
transform $\CF_x g$ with respect to the variable $x$, where for notational simplicity we still use $\xi$ to denote the corresponding frequency variable. In fact,
one has

\begin{Lemma}\label{lem.lide.f}
(i) For any $t\geq0$ and $\xi\in {\bf R}^3$, it holds that
\begin{equation*}
\partial_t\left\{|\hat{f}|^2+|\hat{E}|^2\right\}+\lambda |({\bf
I-P})\hat{f}|_{N^{s}_{\gamma}}^2\leq0.
\end{equation*}

\noindent (ii) There is a time-frequency interactive functional $\mathcal
{E}^{(0)}(t,\xi)$ defined by
\begin{equation*}
\begin{split}
\mathcal
{E}_{int}^{(0)}(t,\xi)=&\sum\limits_{j=1}^3\frac{1}{2}\left(\mathbbm{i}\xi_j\hat{c}|\Pi_{j}(\{{\bf
I}-{\bf
P}\}\hat{f}\cdot[1,1])\right)\\ &+\kappa_1\sum\limits_{j,k=1}^3
\left(\mathbbm{i}\xi_j\hat{b}_k+\mathbbm{i}\xi_k\hat{b}_j|\frac{1}{2}\Xi_{jk}(\{{\bf
I}-{\bf P}\}\hat{f}\cdot[1,1])+2\hat{c}\delta_{jk}\right)\\
&+\kappa_2\sum\limits_{j=1}^3
\left(\mathbbm{i}\xi_j\frac{\hat{a}_++\hat{a}_-}{2}|\hat{b}_j\right),
\end{split}
\end{equation*}
with two properly chosen constants $0<\kappa_2\ll\kappa_1$, such that for any $t\geq0$ and $\xi\in {\bf R}^3$,
\begin{equation*}
\begin{split}
\partial_t \Re\, \frac{\mathcal {E}_{int}^{(0)}(t,\xi)}{1+|\xi|^2}+\frac{\lambda|\xi|^2}{1+|\xi|^2}
\left(|\hat{a}_++\hat{a}_-|^2+|\hat{b}|^2+|\hat{c}|^2\right)\lesssim|\{{\bf
I}-{\bf P}\}\hat{f}|_{N^{s}_{\gamma}}^2.
\end{split}
\end{equation*}

\noindent (iii) There is $0<\kappa_3\ll1$ such that for any $t\geq0$ and $\xi\in {\bf R}^3$,
\begin{equation*}
\begin{split}
\frac{\partial_t
\Re\,(\hat{G}|\mathbbm{i}\xi(\hat{a}_+-\hat{a}_-))}{1+|\xi|^2}&-\kappa_3\partial_t
\Re\,(\hat{G}|\hat{E})+\lambda|\hat{a}_+-\hat{a}_-|^2+\lambda|\hat{E}|^2\\ \lesssim&|\{{\bf I}-{\bf P}\}\hat{f}|_{N^{s}_{\gamma}}^2.
\end{split}
\end{equation*}

%It is straightforward to obtain the following estimates on the micro
%dissipation.

%Now, we summarize the above results in the following lemma.

\noindent (iv)
Let $[f,E]$ be the solution to the Cauchy problem \eqref{LB eqn} with $h=0$. Then there is a time-frequency interactive functional $\mathcal
{E}^{(1)}_{int}(t,\xi)$ such that
\begin{equation*}
\mathcal {E}^{(1)}_{int}(t,\xi)\sim \left|\widehat{f}\right|^2+|\widehat{E}|^2,
\end{equation*}
and for any $t\geq0$ and $\xi\in {\bf R}^3$,
\begin{equation}\label{zero instant energy h0}
\partial_t\mathcal {E}^{(1)}_{int}(t,\xi)+\frac{\lambda|\xi|^2}{1+|\xi|^2}
\left\{\left|\widehat{f}\right|_{L^2_{\gamma+2s}}^2+|\widehat{E}|^2\right\}
\leq0.
\end{equation}
\end{Lemma}

%\begin{proof}
%Let
%\begin{equation}\label{int fun2}
%\begin{split}
%\mathcal
%{E}_{int}^{(1)}(t,\xi)=|\widehat{f}|^2+|\widehat{E}|^2+\kappa_4\left\{\frac{Re\mathcal
%{E}^{(0)}_{int}(t,\xi)}{1+|\xi|^2}+Re\frac{(\hat{G}|\mathbbm{i}\xi(\hat{a}_+-\hat{a}_-))}{1+|\xi|^2}
%-\kappa_3Re(\hat{G}|\hat{E})\right\},
%\end{split}
%\end{equation}
%in fact, one can fix $\kappa_4$ small enough such that (\ref{instant
%energy 1}) holds. Then a linear combination of (\ref{mic-mac
%es1}), (\ref{a-a energy}) and
%(\ref{spatial energy}) yields
%\begin{equation}\label{int fun3}
%\begin{split}
%\frac{d}{dt}\mathcal
%{E}_{int}^{(1)}(t,\xi)&+\lambda\frac{|\xi|^2}{1+|\xi|^2}
%\left(|\hat{a}_{\pm}|^2+|\hat{b}|^2+|\hat{c}|^2\right)\\ &+\lambda
%\left|({\bf I-P})\widehat{f}\right|_{L^2_{\gamma+2s}}^2+\lambda|\hat{a}_+-\hat{a}_-|^2+%\lambda|\hat{E}|^2\leq0.
%\end{split}
%\end{equation}
%Therefore (\ref{zero instant energy h0}) is true.
%The proof of Lemma 3.5 is completed.
%\end{proof}

\begin{proof}[The proof of Theorem 4.1]
%To prove Lemma 3.1, we will perform the energy analysis to $(\ref{LB eqn})_1$ with
%$h=0$ as \cite{Strain11}.
It follows from the first equation of (\ref{LB eqn}) that
\begin{equation*}
\begin{split}
\partial_t({\bf
I-P})\widehat{f}&+\mathbbm{i}\xi\cdot v({\bf I-P})\widehat{f}+ L({\bf
I-P})\widehat{f}\\ =&-({\bf
I-P})(\widehat{E}\cdot v)\sqrt{\mu}q_1-({\bf
I-P})\left[\mathbbm{i}\xi\cdot v{\bf P}\widehat{f}\right]+{\bf P}\left[\mathbbm{i}\xi\cdot v({\bf
I-P})\widehat{f}\right],
\end{split}
\end{equation*}
which yields that for any $l\geq 0$,
\begin{equation}\label{ft microzero energy}
\begin{split}
\frac{1}{2}\partial_t&\left|w^l({\bf
I-P})\widehat{f}\right|_2^2+\Re\,\left\langle L({\bf
I-P})\widehat{f},w^{2l}({\bf
I-P})\widehat{f}\right\rangle\\ &=-\Re\,\left\langle ({\bf
I-P})(\widehat{E}\cdot v)\sqrt{\mu}q_1+({\bf
I-P})\left[\mathbbm{i}\xi\cdot v{\bf P}\widehat{f}\right]-{\bf P}\left[\mathbbm{i}\xi\cdot v({\bf
I-P})\widehat{f}\right]
,w^{2l}({\bf
I-P})\widehat{f}\right\rangle.
\end{split}
\end{equation}
Notice that from Lemma \ref{estimate on L},
$$\mathcal{R}\left\langle L({\bf
I-P})\widehat{f},w^{2l}({\bf I-P})\widehat{f}\right\rangle\geq
\lambda\left|w^l({\bf
I-P})\widehat{f}\right|_{L^2_{\gamma+2s}}^2-C\left|({\bf
I-P})\widehat{f}\right|_{L^2_{B_C}}^2,
$$
and also
the right-hand term of (\ref{ft microzero energy}) can be bounded by
$$C|\widehat{E}|^2+C|\xi|^2\left|\widehat{f}\right|_{L^2_{\gamma+2s}}^2+\eta\left|w^l({\bf
I-P})\widehat{f}\right|_{L^2_{\gamma+2s}}^2,$$ for small $\eta>0.$

Plugging the above estimates into (\ref{ft microzero energy}), one
has
\begin{equation}\label{ft microzero energy2}
\begin{split}\partial_t\left|w^l({\bf
I-P})\widehat{f}\right|_{L^2}&+\lambda\left|w^l({\bf
I-P})\widehat{f}\right|_{L^2_{\gamma+2s}}^2\lesssim |\widehat{E}|^2+|\xi|^2\left|\widehat{f}\right|^2_{L^2_{\gamma+2s}}
+\left|({\bf
I-P})\widehat{f}\right|_{L^{2}_{B_C}}^2.
\end{split}
\end{equation}
In a similar way, starting with the first equation of (\ref{LB eqn}), the direct weighted estimates  also give
\begin{equation}\label{ft spatialzero energy}
\frac{1}{2}\partial_t\left|w^l\widehat{f}\right|_{L^{2}}^2+\lambda\left|w^l\widehat{f}\right|_{L^2_{\gamma+2s}}^2\lesssim
|\widehat{E}|^2+
\left|\widehat{f}\right|_{L^{2}_{B_C}}^2.
\end{equation}

In the case when $|\xi|\geq 1$ which implies $|\xi|^2/(1+|\xi|^2)\geq 1/2$,  one can combine (\ref{zero instant energy h0})
and  (\ref{ft spatialzero energy}) to obtain
\begin{equation}\label{ft spatialzero energy split1}
\partial_t\mathscr{E}_l^{1}(t,\xi)+\lambda\left\{\left|w^l\widehat{f}\right|_{L^2_{\gamma+2s}}^2+|\widehat{E}|^2\right\}\chi_{|\xi|\geq 1}\leq
0,
\end{equation}
where $\mathscr{E}_l^{1}(t,\xi)$ is defined by
\begin{equation*}
\mathscr{E}_l^1(t,\xi)=\left\{\mathcal
{E}_{int}^{(1)}(t,\xi)+\kappa_4\left|w^l\widehat{f}\right|_2^2\right\}\chi_{|\xi|\geq
1}
\end{equation*}
for a constant $\kappa_4>0$ small enough. In the case when
$|\xi|\leq 1$ which implies $|\xi|^2/(1+|\xi|^2)\geq |\xi|^2/2$, one
can combine (\ref{zero instant energy h0}) and (\ref{ft microzero
energy2}) on $|\xi|\leq 1$ to obtain
\begin{equation}\label{ft spatialzero energy split2}
\partial_t\mathscr{E}^0_l(t,\xi)+\lambda|\xi|^2\left\{\left|w^l\widehat{f}\right|_{L^2_{\gamma+2s}}^2+|\widehat{E}|^2\right\}
\chi_{|\xi|\leq 1}\leq
0,
\end{equation}
where $\mathscr{E}^0_l(t,\xi)$ is defined by
\begin{equation*}
\mathscr{E}^0_l(t,\xi)=\left\{\mathcal
{E}_{int}^{(1)}(t,\xi)+\kappa_5\left|w^l({\bf
I-P})\widehat{f}\right|_2^2\right\}\chi_{|\xi|\leq 1}
\end{equation*}
for a constant $\kappa_5>0$ small enough. Therefore, for $l\geq 0$,
by introducing
$$
\mathscr{E}_l(t,\xi)=\mathscr{E}^0_l(t,\xi)+\mathscr{E}^1_l(t,\xi)\sim
\left|w^l\hat{f}(t,\xi,v)\right|^2_2+|\widehat{E}|^2,
$$
it follows from (\ref{ft spatialzero energy split1}) and (\ref{ft
spatialzero energy split2}) that
$$
\partial_t\mathscr{E}_l(t,\xi)+\lambda\frac{|\xi|^2}{1+|\xi|^2}
\left\{\left|w^{l}\widehat{f}\right|_{L^2_{\gamma+2s}}^2+|\widehat{E}|^2\right\}\leq
0,
$$
that is,
\begin{equation}\label{El energy}
\partial_t\mathscr{E}_l(t,\xi)+\lambda\frac{|\xi|^2}{1+|\xi|^2}
\left\{\left|w^{l-\frac{\gamma+2s}{2\gamma}}\widehat{f}\right|_{L^2}^2+|\widehat{E}|^2\right\}\leq
0.
\end{equation}

Now, basing on the above estimate \eqref{El energy} for any $l\geq 0$, one can use the trick in either \cite{Strain11} or  \cite{DYZ-s} to deduce the desired estimate {\eqref{basic decay}}. We here use the trick in \cite{Strain11} to deal with the velocity degeneration. In fact, for $j>0$, it follows from the H\"{o}lder inequality that
\begin{equation*}
\begin{split}
\mathscr{E}_l(t,\xi)\lesssim&\mathscr{E}_{l-\frac{\gamma+2s}{2\gamma}}^{j/(j+1)}(t,\xi)
\mathscr{E}_{l+j\frac{\gamma+2s}{2\gamma}}^{1/(j+1)}(t,\xi)\lesssim
\left\{\left|w^{l-\frac{\gamma+2s}{2\gamma}}\widehat{f}\right|_{L^2}^2+|\widehat{E}|^2\right\}^{j/(j+1)}
\mathscr{E}_{l+j\frac{\gamma+2s}{2\gamma}}^{1/(j+1)}(t,\xi),
\end{split}
\end{equation*}
which  implies
\begin{equation*}
\begin{split}
\mathscr{E}_l^{(j+1)/j}(t,\xi)\lesssim&\left\{\left|w^{l-\frac{\gamma+2s}{2\gamma}}\widehat{f}\right|_{L^2}^2
+|\widehat{E}|^2\right\}
\mathscr{E}_{l+j\frac{\gamma+2s}{2\gamma}}^{1/j}(t,\xi)
\lesssim\left\{\left|w^{l-\frac{\gamma+2s}{2\gamma}}\widehat{f}\right|_{L^2}^2
+|\widehat{E}|^2\right\}
\mathscr{E}_{l+j\frac{\gamma+2s}{2\gamma}}^{1/j}(0,\xi).
\end{split}
\end{equation*}
Then, (\ref{El energy}) together with the above estimate give
\begin{equation*}
\partial_t\mathscr{E}_l(t,\xi)+\lambda\frac{|\xi|^2}{1+|\xi|^2}\mathscr{E}_l^{(j+1)/j}(t,\xi)
\mathscr{E}_{l+j\frac{\gamma+2s}{2\gamma}}^{-1/j}(0,\xi)\leq 0.
\end{equation*}
By solving the inequality, one has
$$j\mathscr{E}_l^{-1/j}(0,\xi)-j\mathscr{E}_l^{-1/j}(t,\xi)\leq -\lambda
t\frac{|\xi|^2}{1+|\xi|^2}\mathscr{E}_{l+j\frac{\gamma+2s}{2\gamma}}^{-1/j}(0,\xi).$$ Therefore, for any $l\geq 0$
and $j>0$, it holds that
\begin{equation}
\label{lide.p1}
\mathscr{E}_l(t,\xi)\lesssim\mathscr{E}_{l+j\frac{\gamma+2s}{2\gamma}}(0,\xi)\left(1+\frac{\lambda
t}{j}\cdot\frac{|\xi|^2}{1+|\xi|^2}\right)^{-j}.
\end{equation}
Here we have used
the fact that $\mathscr{E}_l(0,\xi)\lesssim \mathscr{E}_{l+j\frac{\gamma+2s}{2\gamma}}(0,\xi)$. By integrating \eqref{lide.p1} over $|\xi|\geq1$,
\begin{equation*}
\begin{split}
\int_{|\xi|\geq1}|\xi|^{2k}\mathscr{E}_l(t,\xi)d\xi\lesssim&\left(1+\frac{\lambda
t}{2j}\right)^{-j}\int_{|\xi|\geq1}|\xi|^{2k}\mathscr{E}_{l+j\frac{\gamma+2s}{2\gamma}}(0,\xi)d\xi\\
\lesssim&
(1+t)^{-j}\left\|w^{l+j\frac{\gamma+2s}{2\gamma}}\nabla_x^kf_0\right\|^2,
\end{split}
\end{equation*}
where we have used
$$
\int_{|\xi|\geq1}|\xi|^{2k}|\hat{E}_0|^2d\xi=\int_{|\xi|\geq1}|\xi|^{2k-2}|\cdot |\widehat{\rho_{f_0}}|^2d\xi\lesssim\int_{|\xi|\geq1}|\xi|^{2k}|\cdot |\widehat{\rho_{f_0}}|^2d\xi\lesssim
\left\|\nabla_x^kf_0\right\|^2.
$$
On the other hand, over $|\xi|\leq1$, one can let
$j>2\sigma_{k}$ so that
\begin{equation*}
\begin{split}
\int_{|\xi|<1}|\xi|^{2k}\mathscr{E}_l(t,\xi)d\xi\leq&
\int_{|\xi|<1}\left(1+\frac{\lambda
t}{j}\cdot\frac{|\xi|^2}{1+|\xi|^2}\right)^{-j}|\xi|^{2k}\mathscr{E}_{l+j\frac{\gamma+2s}{2\gamma}}(0,\xi)d\xi
\\
\lesssim&(1+t)^{-2\sigma_{k}}\left\|\mathscr{E}_{l+j\frac{\gamma+2s}{2\gamma}}(0,\xi)\right\|_{L^\infty_\xi}
\\
\lesssim&
(1+t)^{-2\sigma_{k}}\left\{\left\|w^{l+j\frac{\gamma+2s}{2\gamma}}f_0\right\|_{Z^1}^2
+\left\|(1+|x|)\rho_0\right\|^2_{L^1}\right\},
\end{split}
\end{equation*}
where the fact that
\begin{equation}
\label{est.rho}
\|\hat{E}_0\|_{L^\infty_\xi}=\left\|{|\xi|^{-1}}\widehat{\rho_{f_0}}\right\|_{L^\infty_\xi}\lesssim \left\|(1+|x|)\rho_{f_0}\right\|_{L^1}
\end{equation}
has been used. Then, the desired estimate \eqref{basic decay}
follows by taking $\ell_\ast=j\frac{\ga+2s}{2\ga}$ with $j>2\si_k$.
This completes the proof of Theorem \ref{basic decay lemma}.
\end{proof}

\section{Macroscopic dissipation}\label{sec.5}

This section is concerned with the analysis of the macroscopic dissipation basing on the linearized system \eqref{LB eqn} with the nonhomogeneous source
$h$ satisfying \eqref{nonh cons}, which will be applied in the next section to the energy estimates on the nonlinear Vlasov-Poisson-Boltzmann system
(\ref{perturbed eqn}) and (\ref{perturbed E}).

\begin{Lemma}\label{lem.sys.h}
Let $[f,E]$ be the solution to the Cauchy problem \eqref{LB eqn} with the
nonhomogenous term $h$ satisfying $\langle h,\sqrt{\mu}\rangle=0$. Let $K\geq 8$ be an integer. Then, there are two time-frequency interaction functionals $\mathscr{E}_{int}^{K}(t)$
and $\mathscr{E}_{int}^{K,h}(t)$
satisfying
\begin{eqnarray}\label{instant energy 2}
\begin{split}
|\mathscr{E}_{int}^{K}(t)|\lesssim & \sum\limits_{|\alpha|\leq K}\|\partial^{\alpha}
f\|^2+\sum\limits_{|\alpha|\leq K-1}\|\partial^{\alpha}
E\|^2,\\
|\mathscr{E}_{int}^{K,h}(t)|\lesssim & \sum\limits_{1\leq|\alpha|\leq K}\|\partial^{\alpha}{\bf P}
f\|^2+\sum\limits_{|\alpha|\leq K}\|\partial^{\alpha}({\bf I}-{\bf P})
f\|^2+\sum\limits_{|\alpha|\leq K-1}\|\partial^{\alpha}
E\|^2,
\end{split}
\end{eqnarray}
such that for any $t\geq0$,
\begin{equation}\label{mic-mac2}
\begin{split}
\partial_t\mathscr{E}_{int}^{K}(t)&+\lambda\sum\limits_{|\alpha|\leq K-1}\|\partial^{\alpha}\nabla_x
(a_{\pm},b,c)\|^2+\lambda\|(a_+-a_-)\|^2+\lambda\|E\|^2
\\ \lesssim& \sum\limits_{|\alpha|\leq K}
\left\|({\bf I-P})\partial^{\alpha}f\right\|^2_{N^{s}_{\gamma}}
+\sum\limits_{|\alpha|\leq
K-1}{\left\|\zeta(v)\partial^{\alpha}h\right\|^2},
\end{split}
\end{equation}
and
\begin{equation}\label{mic-mac3}
\begin{split}
\partial_t\mathscr{E}_{int}^{K,h}(t)&+\lambda\sum\limits_{1\leq|\alpha|\leq K-1}\|\partial^{\alpha}\nabla_x
(a_{\pm},b,c)\|^2+\lambda\|\nabla_x(a_+-a_-)\|^2+\lambda\sum\limits_{|\alpha|\leq K-1}\|\partial^\alpha E\|^2
\\ \lesssim& \sum\limits_{|\alpha|\leq K}
\left\|({\bf I-P})\partial^{\alpha}f\right\|^2_{N^{s}_{\gamma}}
+\sum\limits_{|\alpha|\leq
K-1}{\left\|\zeta(v)\partial^{\alpha}h\right\|^2},
\end{split}
\end{equation}
where $\zeta(v)$ is defined by \eqref{expo fun}.
\end{Lemma}

\begin{proof}
Take $\al$ with $|\al|\leq K$. By denoting
\begin{equation}\label{int fun3}
\begin{split}
\mathcal {E}^{(\alpha)}_{int}(t)=& \frac{1}{2}\left(\nabla_x\partial^{\alpha}c,\Pi(\partial^{\alpha}\{{\bf
I}-{\bf
P}\}f\cdot[1,1])\right)\\ &+\kappa_1\sum\limits_{j,k=1}^3
\left(\partial_j\partial^{\alpha}b_k+\partial_k\partial^{\alpha}b_j,\frac{1}{2}\Xi_{jk}(\{{\bf
I}-{\bf P}\}\partial^{\alpha}f\cdot[1,1])+2\partial^{\alpha}c\delta_{jk}\right)\\
&+\kappa_2
\left(\nabla_x\frac{\partial^{\alpha}a_++\partial^{\alpha}a_-}{2},\partial^{\alpha}b\right),
\end{split}
\end{equation}
one can use the same argument as in \cite[Lemma 4.1]{Duan} to obtain
\begin{equation}\label{int fun3 energy}
\frac{d}{dt}\mathcal
{E}^{(\alpha)}_{int}(t)+\lambda\|\partial^{\alpha}\nabla
((a_{+}+a_{-}),b,c)\|^2\lesssim
\sum\limits_{|\alpha'|\leq|\alpha|+1}\|\partial^{\alpha'}({\bf
I-P})f\|_{N^{s}_{\gamma}}^2+\left\|\zeta(v)\partial^{\alpha}h\right\|^2,
\end{equation}
and similarly, as for obtaining (iii) in Lemma \ref{lem.lide.f}, one has
\begin{eqnarray}
\frac{d}{dt}(\partial^{\alpha}G,\nabla_x\partial^{\alpha}(a_+-a_-))&+&\lambda\sum\limits_{|\alpha|\leq|\alpha'|\leq|\alpha|+1}\|\partial^{\alpha'}(a_+-a_-)\|^2\notag\\
&\lesssim&
\sum\limits_{|\alpha'|\leq|\alpha|+1}\|\partial^{\alpha'}({\bf
I-P})f\|_{N^{s}_{\gamma}}^2+\left\|\zeta(v)\partial^{\alpha}h\right\|^2.\label{int fun4 energy}
\end{eqnarray}
To include the zero-order term $\|E\|$  in the dissipation, from applying
$\partial^{\alpha}$ to (\ref{b-b}) and then taking the inner product
with $\partial^{\alpha}E$ over ${\bf R}^3\times{\bf R}^3$, we have
\begin{equation}\label{e energy}
\begin{split}
2\|\partial^{\alpha}E\|^2\leq&
(\partial_t\partial^{\alpha}G,\partial^{\alpha}E)
+|(\partial^{\alpha}\nabla_x\cdot \Xi(({\bf I-P})f\cdot
q_1),\partial^{\alpha}E)|+\left(\nabla_x\partial^{\alpha}(a_+-a_-),\partial^{\alpha}E\right)\\
& +|(\pa^\al \lag h-Lf, [v\sqrt{\mu},-v\sqrt{\mu}]\rag, \pa^\al E)|\\
\leq&\frac{d}{dt}(\partial^{\alpha}G,\partial^{\alpha}E)-(\partial^{\alpha}G,\partial_t\partial^{\alpha}E)+\epsilon\|\partial^{\alpha}E\|^2
\\ &+
\frac{C}{\epsilon}\left\{\|\partial^{\alpha}\nabla_x({\bf
I-P})f\|^2+\left\|\partial^{\alpha}\nabla_x(a_+-a_-)\right\|^2+\|\pa^\al
\lag h-Lf, [v\sqrt{\mu},-v\sqrt{\mu}]\rag\|^2\right\},
\end{split}
\end{equation}
where in the last inequality we have used integrations by part in
the $t$-variable and the Cauchy-Schwarz inequality with $\epsilon$.
Notice that from (\ref{a-a}) and (\ref{ea}),
\begin{equation}\label{e energy1}
-(\partial^{\alpha}G,\partial_t\partial^{\alpha}E)=
(\partial^{\alpha}G,\partial^{\alpha}\nabla_x\De_x^{-1}\nabla_x\cdot
G)\lesssim\|\partial^{\alpha}G\|^2\lesssim \|\partial^{\alpha}({\bf
I-P})f\|_{N^{s}_{\gamma}}^2.
\end{equation}
Therefore, (\ref{e energy}) together with (\ref{e energy1}) imply
\begin{equation}\label{e energy2}
\begin{split}
-\frac{d}{dt}(\partial^{\alpha}G,\partial^{\alpha}E)+\|\partial^{\alpha}E\|^2\lesssim
\sum\limits_{|\alpha'|\leq|\alpha|+1}\|\partial^{\alpha'}({\bf
I-P})f\|_{N^{s}_{\gamma}}^2+\left\|\partial^{\alpha}\nabla_x(a_+-a_-)\right\|^2+\left\|\zeta(v)\partial^{\alpha}h\right\|^2.
\end{split}
\end{equation}
Now, we define
\begin{equation*}
\begin{split}
\mathscr{E}_{int}^{K,h}(t)=& \frac{1}{2}\sum\limits_{1\leq|\alpha|\leq K-1}\left(\nabla_x\partial^{\alpha}c,\Pi(\partial^{\alpha}\{{\bf
I}-{\bf
P}\}f\cdot[1,1])\right)\\ &+\kappa_1\sum\limits_{1\leq|\alpha|\leq K-1}\sum\limits_{j,k=1}^3
\left(\partial_j\partial^{\alpha}b_k+\partial_k\partial^{\alpha}b_j,\frac{1}{2}\Xi_{jk}(\{{\bf
I}-{\bf P}\}\partial^{\alpha}f\cdot[1,1])+2\partial^{\alpha}c\delta_{jk}\right)\\
&+\kappa_2\sum\limits_{1\leq|\alpha|\leq K-1}
\left(\nabla_x\frac{\partial^{\alpha}a_++\partial^{\alpha}a_-}{2},\partial^{\alpha}b\right)
+{\sum\limits_{1\leq|\alpha|\leq K-1}}
(\partial^{\alpha}G,\nabla_x\partial^{\alpha}(a_+-a_-))\\
&-\kappa_6\sum\limits_{|\alpha|\leq
K-1}(\partial^{\alpha}G,\partial^{\alpha}E),
\end{split}
\end{equation*}
for a suitably chosen constant $0<\kappa_6\ll1$. It is
straightforward to see that the {second} equation of (\ref{instant
energy 2}) holds true and also  (\ref{mic-mac3}) follows by taking
the sum of $(\ref{int fun3 energy})$, $(\ref{int fun4 energy})$ and
$(\ref{e energy2})\times\kappa_6$. In a similar way, basing on the
obtained estimates, $\mathscr{E}_{int}^{K}(t)$ can be constructed to
satisfy both the {first} equation of (\ref{instant energy 2}) and
\eqref{mic-mac2}. This completes the proof of Lemma \ref{lem.sys.h}.
\end{proof}

\section{The a priori estimates}\label{sec.6}

In this section, we are going to deduce the uniform-in-time a priori estimates on the solution to the Cauchy problem \eqref{perturbed eqn}-\eqref{perturbed data} of the Vlasov-Poisson-Boltzmann
system. For this purpose, we define the following time-weighted energy norm $X(t)$ by
\begin{equation}\label{def.X}
\begin{split}
X(t)=\sup\limits_{0\leq \tau\leq t}\mathcal {E}_{l_0+l_1}(\tau)+\sup\limits_{0\leq\tau\leq t}(1+\tau)^{\frac{3}{2}}\mathcal {E}_{l_0}(\tau)+\sup\limits_{0\leq\tau\leq t}(1+\tau)^{\frac{3}{2}+p}\mathcal {E}^h_{l_0}(\tau),
\end{split}
\end{equation}
with all involved parameters fixed to satisfy
\begin{eqnarray}
&\dis -3<\ga<-2s,\ \frac{1}{2}\leq s<1,\ l_0 \geq 0,\ K\geq 8,\ \frac{1}{2}<p<1,\label{para1}
\end{eqnarray}
and
\begin{equation}\label{para2}
l_1=\frac{5}{4(1-p)}\frac{\gamma+2s}{\gamma},
\end{equation}
where the construction of the temporal energy functionals $\mathcal {E}_{l}(t)$ and $\mathcal {E}^h_{l}(t)$ will be given in the following two lemmas. Suppose that the Cauchy problem  \eqref{perturbed eqn}-\eqref{perturbed data}  admits a smooth solution $f(t,x,v)$
over $0\leq t\leq T$ for $0< T\leq\infty$, and also the solution $f(t,x,v)$ satisfies
\begin{equation}\label{a priori estimates}
\begin{split}
\sup\limits_{0\leq t\leq T}X(t)\leq  \de_0,
\end{split}
\end{equation}
where $\de_0>0$ is a suitably small constant.

%We are ready to prove the energy inequalities (\ref{energy estimates}) and (\ref{high %energy estimates}). Now we consider the proof of (\ref{energy estimates}) in the following.

\begin{Lemma}\label{lem.en.l}
%Assume $-3<\gamma<-2s$. Let $K\geq8$, $l\geq0$. Suppose that a priori estimate (\ref{a priori estimates})
%holds for $M_0$ small enough.
For any $l$ with $0\leq l\leq l_0+l_1$, there is $\mathcal {E}_l(t)$ satisfying (\ref{def e}) such that
\begin{equation}\label{energy liyine}
\frac{d}{dt}\mathcal {E}_l(t)+\lambda \mathcal {D}_l(t)\lesssim \|\partial_t\phi\|_{L^\infty}\mathcal {E}_l(t)
\end{equation}
holds for any $0\leq t\leq T$, where $\mathcal {D}_l(t)$ is defined by (\ref{def d}).
\end{Lemma}

\begin{proof}
It is divided by three steps. Notice that
(\ref{perturbed eqn}) can be rewritten as
\begin{equation}\label{two VPBeqn}
\begin{split}
\big[\partial_{t}&+v_i\partial^{e_i}\mp\partial^{e_i}\phi\partial_{e_i}\big]f_{\pm}\pm\frac{1}{2}(\partial^{e_i}\phi v_i)f_{\pm}\pm(\partial^{e_i}\phi v_i)
\sqrt{\mu}
+L_{\pm}f=\Gamma_{\pm}(f,f).
\end{split}
\end{equation}
%where $|e_i|=1.$

\noindent
{\bf Step 1.}  Energy estimates without any weight: Applying $\partial^\alpha$ with $|\alpha|\leq K$ to (\ref{two VPBeqn}) and  taking the inner product with $e^{\pm\phi}\partial^\alpha f_{\pm}$ over ${\bf R}^3\times{\bf R}^3$, one has
\begin{equation}\label{spatial inner withoutw}
\begin{split}
\left(\partial_{t}\partial^\alpha f_{\pm}, e^{\pm\phi}\partial^\alpha f_{\pm}\right)&+\left(v_i\partial^{e_i+\alpha}f_{\pm},e^{\pm\phi}\partial^\alpha f_{\pm}\right)
\pm\frac{1}{2}\left((\partial^{e_i}\phi v_i)\partial^{\alpha}f_{\pm},e^{\pm\phi}\partial^\alpha f_{\pm}\right)\\
&\pm\left((\partial^{e_i+\alpha}\phi v_i)
\sqrt{\mu},\partial^\alpha f_{\pm}\right)
+\left(L_{\pm}\partial^\alpha f,\partial^\alpha f_{\pm}\right)\\ =&
\mp\chi_{|\alpha|}\sum\limits_{1\leq|\alpha_1|\leq|\alpha|}C_\alpha^{\alpha_1}
\frac{1}{2}\left((\partial^{e_i+\alpha_1}\phi v_i)\partial^{\alpha-\alpha_1}f_{\pm},e^{\pm\phi}\partial^\alpha f_{\pm}\right)\\
&\pm\chi_{|\alpha|}\sum\limits_{1\leq|\alpha_1|\leq|\alpha|}C_\alpha^{\alpha_1}\left((\partial^{e_i+\alpha_1}\phi \partial_{e_i}^{\alpha-\alpha_1}f_{\pm},e^{\pm\phi}\partial^\alpha f_{\pm}\right)\\
&\mp\left((\partial^{e_i+\alpha}\phi v_i)
\sqrt{\mu},\left(1-e^{\pm\phi}\right)\partial^\alpha f_{\pm}\right)
\\
&+\left(L_{\pm}\partial^\alpha f,\left(1-e^{\pm\phi}\right)\partial^\alpha f_{\pm}\right)
+\left(\partial^\alpha\Gamma_{\pm}(f,f),e^{\pm\phi}\partial^\alpha f_{\pm}\right),
\end{split}
\end{equation}
where $\chi_{|\al|}=1$ if $|\al|>0$, and $\chi_{|\al|}=0$ if $|\al|=0$. Now  we sum the above equations with $\pm$ and estimate the resulting equations
term by term.

The first term on the left hand side of $(\ref{spatial inner withoutw})$ is just
$$
\frac{1}{2}\frac{d}{dt}\sum\limits_{\pm}\left\|e^{\frac{\pm\phi}{2}}\partial^\alpha f_{\pm}\right\|\mp\sum\limits_{\pm}\frac{1}{2}(\partial_t\phi\partial^\alpha f_{\pm},\partial^\alpha f_{\pm}).
$$
As in \cite{Guo5}, the extra factor $e^{\pm\phi}$ is designed to treat the second term. In fact, from integration by parts,
\begin{equation*}
\begin{split}
\left(v_i\partial^{e_i+\alpha}f_{\pm},e^{\pm\phi}\partial^\alpha f_{\pm}\right)
\pm\frac{1}{2}\left((\partial^{e_i}\phi v_i)\partial^{\alpha}f_{\pm},e^{\pm\phi}\partial^\alpha f_{\pm}\right)=0.
\end{split}
\end{equation*}
Recalling (\ref{poisson eqn}), (\ref{a-a}) and (\ref{ea}), one has
$$
\sum\limits_{\pm}\pm\left((\partial^{e_i+\alpha}\phi v_i)
\sqrt{\mu},\partial^\alpha f_{\pm}\right)=\frac{1}{2}\frac{d}{dt}\|\partial^\alpha\nabla_x\phi\|^2.
$$
Notice that by \eqref{estimate on L2},
$$
\sum\limits_{\pm}\left(L_{\pm}\partial^\alpha f,\partial^\alpha
f_{\pm}\right)\geq \la\|({\bf I}-{\bf P})\partial^\alpha
f\|^2_{N^{s}_{\gamma}},
$$
and by Lemma \ref{lem.vphi.x} and Lemma \ref{lem.phi.x},
the first and second terms on the right hand side of (\ref{spatial inner withoutw}) are bounded by
$C\sqrt{\mathcal {E}_l(t)}\mathcal {D}_l(t)$.

We denote the third term on the right hand side of (\ref{spatial inner withoutw}) as $J_5$.
Notice
\begin{equation}\label{epowerforce}
\sup_{x\in {\bf
R}^3}|1-e^{\pm\phi}|\lesssim\|\phi\|_{L^\infty}\lesssim
\|\nabla_x\phi\|_{H^1}\lesssim \de_0.
\end{equation}
For $\alpha=0$, by Cauchy-Schwarz and Sobolev inequalities,
$$
J_5\lesssim\|\nabla_x \phi\|^2\|\mu^{\frac{1}{4}}f\|_{H^2}\lesssim \sqrt{\mathcal {E}_l(t)}\mathcal {D}_l(t),
$$
while for $|\alpha|\geq1$,
$$
J_5\lesssim \|\nabla_x \phi\|_{H^2}\sum\limits_{1\leq|\alpha|\leq K}\left\{\|\partial^\alpha \nabla_x \phi\|^2\left\|\mu^{\frac{1}{4}}\partial^\alpha f\right\|\right\}\lesssim \sqrt{\mathcal {E}_l(t)}\mathcal {D}_l(t).
$$
For the fourth term on the right hand side of (\ref{spatial inner withoutw}), recalling (\ref{macro def}) and (\ref{Lop expr2}), and then applying Lemma \ref{lem.non.ga}, one has
\begin{equation}\label{Lup1}
\begin{split}
\left|\left(L_{\pm}\partial^\alpha f,\left(1-e^{\pm\phi}\right)\partial^\alpha f_{\pm}\right)\right|
=&\left|\left(L_{\pm}\partial^\alpha ({\bf I}_{\pm}-{\bf P}_{\pm})f,\left(1-e^{\pm\phi}\right)
\partial^\alpha f_{\pm}\right)\right|\\
\lesssim& \int_{{\bf R}^3}|\partial^\alpha({\bf I}-{\bf P})f|_{N^{s}_{\gamma}}|\partial^\alpha f|_{N^{s}_{\gamma}}|\nabla_x\phi|dx\\ \lesssim&
\|\nabla_x\phi\|_{H^2}\|\partial^\alpha({\bf I}-{\bf
P})f\|_{N^{s}_{\gamma}}\|\partial^\alpha f\|_{N^{s}_{\gamma}}\lesssim \sqrt{\mathcal {E}_l(t)}\mathcal {D}_l(t).
\end{split}
\end{equation}
For the fifth term on the right hand side of (\ref{spatial inner withoutw}), by Lemma \ref{lem.non.ga},
$$
\sum\limits_{\pm}\left(\partial^\alpha\Gamma_{\pm}(f,f),e^{\pm\phi}\partial^\alpha f_{\pm}\right)
\lesssim \sqrt{\mathcal {E}_l(t)}\mathcal {D}_l(t).
$$
Furthermore, one can apply the estimate \eqref{mic-mac2} in Lemma \ref{lem.sys.h}, where for the term involving
$h$ defined by (\ref{h VPB}),
through splitting $f$ into ${\bf P}f+({\bf I}-{\bf P})f$ and using Lemma \ref{estimates on nonop2}, it can be bounded as
\begin{equation}\label{hL2}
\sum\limits_{|\alpha|\leq
K-1}\left\|\zeta(v)\partial^{\alpha}h\right\|^2\lesssim \mathcal
{E}_l(t)\mathcal {D}_l(t).
\end{equation}
Therefore, by combing  \eqref{mic-mac2} and all the estimates above, we conclude in this step that
\begin{equation}\label{total spatial energy}
\begin{split}
\frac{d}{dt}&\left\{
\sum\limits_{|\alpha|\leq K}\sum\limits_{\pm}\left\|e^{\frac{\pm\phi}{2}}\partial^\alpha f_{\pm}\right\|^2
+\sum\limits_{|\alpha|\leq K}\|\partial^\alpha\nabla_x\phi\|^2+\kappa\mathscr{E}_{int}^{K}(t)
\right\}\\ &+\lambda\sum\limits_{|\alpha|\leq K-1}\|\partial^{\alpha}\nabla
(a_{\pm},b,c)\|^2+\lambda\|(a_+-a_-)\|^2+\lambda\|\nabla_x\phi\|^2+\lambda\sum\limits_{|\alpha|\leq K}\left\|({\bf I-P})\partial^{\alpha}f\right\|^2_{N^{s}_{\gamma}}\\
\lesssim& \|\partial_t\phi\|_{L^\infty}\sum\limits_{|\alpha|\leq K}\sum\limits_{\pm}\left\|\partial^\alpha f_{\pm}\right\|^2+\left(\sqrt{\mathcal {E}_l(t)}+\mathcal {E}_l(t)\right)\mathcal {D}_l(t),
\end{split}
\end{equation}
where the constant $\kappa>0$ is small enough.

\medskip
\noindent{\bf Step 2.} Energy estimates with weight function $w_l(\alpha,\beta)$:

\medskip
\noindent{\it Step 2.1} One can rewrite (\ref{two VPBeqn}) as
\begin{equation}\label{micro two VPBeqn}
\begin{split}
\big[\partial_{t}&+v_i\partial^{e_i}\mp\partial^{e_i}\phi\partial_{e_i}\big]({\bf I}_{\pm}-{\bf P}_{\pm})f
\pm\frac{1}{2}(\partial^{e_i}\phi v_i)({\bf I}_{\pm}-{\bf P}_{\pm})f+L_{\pm}({\bf I}_{\pm}-{\bf P}_{\pm})f\\ =&-\big[\partial_{t}+v_i\partial^{e_i}\mp\partial^{e_i}\phi\partial_{e_i}\big]{\bf P}_{\pm}f\mp\frac{1}{2}(\partial^{e_i}\phi v_i){\bf P}_{\pm}f\mp(\partial^{e_i}\phi v_i)
\sqrt{\mu}+\Gamma_{\pm}(f,f).
\end{split}
\end{equation}
By multiplying (\ref{micro two VPBeqn}) by $e^{\pm\phi}w^2_l(0,0)({\bf I}_{\pm}-{\bf P}_{\pm})f$ and taking the integration over
${\bf R}^3\times{\bf R}^3$, one has
\begin{equation}\label{micro two VPBeqn inner}
\begin{split}
\frac{1}{2}\sum\limits_{\pm}&\frac{d}{dt}\left\|e^{\frac{\pm\phi}{2}}w_l(0,0)({\bf I}_{\pm}-{\bf P}_{\pm})f\right\|^2
+\sum\limits_{\pm}\left(L_{\pm}({\bf I}_{\pm}-{\bf P}_{\pm})f,w^2_l(0,0)({\bf I}_{\pm}-{\bf P}_{\pm})f\right)
\\ &-\frac{1}{2}\sum\limits_{\pm}\left(\pm\partial_t\phi ({\bf I}_{\pm}-{\bf P}_{\pm})f,e^{\pm\phi}w^2_l(0,0)({\bf I}_{\pm}-{\bf P}_{\pm})f\right)\\
=&\sum\limits_{\pm}\left(L_{\pm}({\bf I}_{\pm}-{\bf P}_{\pm})f,\left(1-e^{\pm\phi}\right)w^2_l(0,0)({\bf I}_{\pm}-{\bf P}_{\pm})f\right)\\
&+\sum\limits_{\pm}\left(\pm\partial^{e_i}\phi \partial_{e_i}({\bf I}_{\pm}-{\bf P}_{\pm})f,e^{\pm\phi}w^2_l(0,0)({\bf I}_{\pm}-{\bf P}_{\pm})f\right)\\
&\mp\left(\frac{1}{2}(\partial^{e_i}\phi v_i){\bf P}_{\pm}f,e^{\pm\phi}w^2_l(0,0)({\bf I}_{\pm}-{\bf P}_{\pm})f\right)\\
&\mp\sum\limits_{\pm}\left((\partial^{e_i}\phi v_i)
\sqrt{\mu},e^{\pm\phi}w^2_l(0,0)({\bf I}_{\pm}-{\bf P}_{\pm})f\right)
\\
&-\sum\limits_{\pm}\left(\left[\partial_{t}+v_i\partial^{e_i}\mp\partial^{e_i}\phi\partial_{e_i}\right]{\bf P}_{\pm}f,e^{\pm\phi}w^2_l(0,0)({\bf I}_{\pm}-{\bf P}_{\pm})f\right)
\\ &+\sum\limits_{\pm}\left(\Gamma_{\pm}(f,f),e^{\pm\phi}w^2_l(0,0)({\bf I}_{\pm}-{\bf P}_{\pm})f\right),
\end{split}
\end{equation}
where we have used the identity
$$
\left(v_i\partial^{e_i}({\bf I}_{\pm}-{\bf P}_{\pm})f
\pm\frac{1}{2}(\partial^{e_i}\phi v_i)({\bf I}_{\pm}-{\bf P}_{\pm})f,e^{\pm\phi}w^2_l(0,0)({\bf I}_{\pm}-{\bf P}_{\pm})f\right)=0.
$$
Now we further estimate (\ref{micro two VPBeqn inner}) term by term. The third term on the left hand side of
(\ref{micro two VPBeqn inner}) is bounded by
$$
\|\partial_t\phi\|_{L^\infty}\left\|e^{\frac{\pm\phi}{2}}w_l(0,0)({\bf I}_{\pm}-{\bf P}_{\pm})f\right\|^2.
$$
For the left-hand second term, it follows from Lemma \ref{estimate on L} that
$$
\sum\limits_{\pm}\left(L_{\pm}({\bf I}_{\pm}-{\bf
P}_{\pm})f,w^2_l(0,0)({\bf I}_{\pm}-{\bf P}_{\pm})f\right)\geq
\la\left\|w_l(0,0)({\bf I}-{\bf
P})f\right\|^2_{N^{s}_{\gamma}}-C\|({\bf I}-{\bf P})f\|_{L^2(B_C)}.
$$
Similar to estimating (\ref{Lup1}), the first term on the right hand side of
(\ref{micro two VPBeqn inner}) is bounded by
$$
\|\nabla_x\phi\|_{H^2}\left\|w_l(0,0)({\bf I}-{\bf
P})f\right\|^2_{N^{s}_{\gamma}}.
$$
The right-hand second and third terms of (\ref{micro two VPBeqn
inner}) are bounded by $C\sqrt{\mathcal {E}_l(t)}\mathcal {D}_l(t)$,
where we have used the velocity integration by parts for the second
term and (\ref{macro def}) for the third term, as well as the a
priori assumption (\ref{a priori estimates}). From Cauchy-Schwarz
inequality with $\eta$ and the local conservation laws
{(\ref{conservation abc}) and (\ref{a-a})}, the right-hand
fourth and fifth terms are bounded by
\begin{equation*}
\begin{split}
\eta&\left\|w_l(0,0)({\bf I}-{\bf P})f\right\|^2_{N^{s}_{\gamma}}\\ &+C\left\{\|\nabla_x\phi\|^2+\|\nabla_x(a_{\pm},b,c)\|^2
+\sum\limits_{|\alpha|\leq1}\|({\bf I}-{\bf
P})f\|^2_{N^{s}_{\gamma}}+\mathcal {E}_l(t)\mathcal {D}_l(t)\right\}.
\end{split}
\end{equation*}
For the nonlinear term $\Ga_\pm$, Lemma \ref{lem.non.ga} implies
$$
\sum\limits_{\pm}\left(\Gamma_{\pm}(f,f),e^{\pm\phi}w^2_l(0,0)({\bf I}_{\pm}-{\bf P}_{\pm})f\right)\lesssim
\sqrt{\mathcal {E}_l(t)}\mathcal {D}_l(t).
$$
Plugging all the above estimates into (\ref{micro two VPBeqn inner}) and fixing a properly small
constant $\eta>0$ yield
\begin{equation}\label{micro zeao energy}
\begin{split}
\sum\limits_{\pm}&\frac{d}{dt}\left\|e^{\frac{\pm\phi}{2}}w_l(0,0)({\bf I}_{\pm}-{\bf P}_{\pm})f\right\|^2
+\lambda\left\|w_l(0,0)({\bf I}-{\bf P})f\right\|^2_{N^{s}_{\gamma}}\\
\lesssim& \|\partial_t\phi\|_{L^\infty}\sum\limits_{\pm}\left\|e^{\frac{\pm\phi}{2}}w_l(0,0)({\bf I}_{\pm}-{\bf P}_{\pm})f\right\|^2\\
&+\|\nabla_x\phi\|^2+\|\nabla_x(a_{\pm},b,c)\|^2
+\sum\limits_{|\alpha|\leq1}\|\pa^\al ({\bf I}-{\bf
P})f\|^2_{N^{s}_{\gamma}}+\left\{\sqrt{\mathcal
{E}_l(t)}+\mathcal {E}_l(t)\right\}\mathcal {D}_l(t).
\end{split}
\end{equation}

%\medskip
\noindent{\it Step 2.2.} For the weighted estimate on the spatial derivatives, we start with
(\ref{two VPBeqn}). In fact, take $1\leq|\alpha|\leq N$. By applying $\partial^\alpha$ to (\ref{two VPBeqn})
and taking the inner product with $e^{\pm\phi}w_l^2(\alpha,0)\partial^\alpha f_{\pm}$ over ${\bf R}^3\times{\bf R}^3$, one has
\begin{equation}\label{spatial inner withw}
\begin{split}
\frac{1}{2}\frac{d}{dt}&\sum\limits_{\pm}\left\|e^{\frac{\pm\phi}{2}}w_l(\alpha,0)\partial^\alpha
f_{\pm}\right\| +\la\left\|w_l(\alpha,0)\partial^\alpha
f\right\|_{N^{s}_{\gamma}}^2
-C\|\partial^\alpha f\|_{L^2_{B_C}}^2\\ \lesssim&
\|\partial_t\phi\|_{L^\infty}\sum\limits_{\pm}\left\|e^{\frac{\pm\phi}{2}}w_l(\alpha,0)\partial^\alpha f_{\pm}\right\|^2
+\sum\limits_{\pm}\sum\limits_{1\leq|\alpha_1|\leq|\alpha|}\frac{1}{2}
C_\alpha^{\alpha_1}\left(\mp(\partial^{e_i+\alpha_1}\phi v_i)\partial^{\alpha-\alpha_1}f_{\pm},e^{\pm\phi}w^2_l(\alpha,0)\partial^\alpha f_{\pm}\right)\\
&+\sum\limits_{\pm}\sum\limits_{|\alpha_1|\leq|\alpha|}C_\alpha^{\alpha_1}\left(\pm(\partial^{e_i+\alpha_1}\phi
\partial_{e_i}^{\alpha-\alpha_1}f_{\pm},e^{\pm\phi}w^2_l(\alpha,0)\partial^\alpha
f_{\pm}\right)+\sum\limits_{\pm}\left(\mp(\partial^{e_i+\alpha}\phi
v_i) \sqrt{\mu},e^{\pm\phi}w^2_l(\alpha,0)\partial^\alpha
f_{\pm}\right)
\\
&+\sum\limits_{\pm}\left(L_{\pm}\partial^\alpha f,\left(1-e^{\pm\phi}\right)w^2_l(\alpha,0)\partial^\alpha f_{\pm}\right)
+\left(\partial^\alpha\Gamma_{\pm}(f,f),e^{\pm\phi}w^2_l(\alpha,0)\partial^\alpha f_{\pm}\right),
\end{split}
\end{equation}
where we have used Lemma \ref{lem.non.ga} and the fact that
$$
\left(v_i\partial^{e_i+\alpha}f_{\pm},e^{\pm\phi}w^2_l(\alpha,0)\partial^\alpha f_{\pm}\right)
\pm\frac{1}{2}\left((\partial^{e_i}\phi v_i)\partial^{\alpha}f_{\pm},e^{\pm\phi}w^2_l(\alpha,0)\partial^\alpha f_{\pm}\right)=0.
$$
From Lemma \ref{lem.vphi.x} and Lemma \ref{lem.phi.x}, the right-hand second and third terms of (\ref{spatial inner withw}) are bounded by $\sqrt{\mathcal {E}_l(t)}\mathcal {D}_l(t)$. For the right-hand fourth term of (\ref{spatial inner withw}), by Cauchy-Schwarz inequality with $\eta$, one has
$$
\sum\limits_{\pm}\left((\partial^{e_i+\alpha}\phi v_i)
\sqrt{\mu},e^{\pm\phi}w^2_l(\alpha,0)\partial^\alpha
f_{\pm}\right)\lesssim \eta\left\|w_l(\alpha,0)\partial^\alpha
f\right\|_{N^{s}_{\gamma}}^2+C\left\|\partial^{e_i+\alpha}\phi\right\|^2.
$$
From Lemma \ref{lem.non.ga}, the right-hand fifth and sixth terms of (\ref{spatial inner withw}) are also
bounded by $\sqrt{\mathcal {E}_l(t)}\mathcal {D}_l(t)$.
Putting those estimates into (\ref{spatial inner withw}), taking summation over $1\leq|\alpha|\leq K$ and fixing
a small constant $\eta>0$ give
\begin{equation}\label{spatial inner withw total}
\begin{split}
\frac{d}{dt}\sum\limits_{\pm}&\sum\limits_{1\leq|\alpha|\leq K}\left\|e^{\frac{\pm\phi}{2}}w_l(\alpha,0)\partial^\alpha f_{\pm}\right\|^2
+\lambda\sum\limits_{1\leq|\alpha|\leq K}\left\|w_l(\alpha,0)\partial^\alpha f\right\|^2_{N^{s}_{\gamma}}\\
\lesssim& \|\partial_t\phi\|_{L^\infty}\sum\limits_{1\leq|\alpha|\leq K}\sum\limits_{\pm}\left\|e^{\frac{\pm\phi}{2}}w_l(\alpha,0)\partial^\alpha f_{\pm}\right\|^2+\sum\limits_{1\leq|\alpha|\leq K}\|\nabla_x\partial^\alpha\phi\|^2\\
&+\sum\limits_{1\leq|\alpha|\leq K}\|\partial^\alpha f\|^2_{N^{s}_{\gamma}}
+\sqrt{\mathcal {E}_l(t)}\mathcal {D}_l(t).
\end{split}
\end{equation}

\medskip
\noindent{\it Step 2.3.} For the weighted estimate on the mixed derivatives, we start with (\ref{micro two VPBeqn}).
Let $1\leq m\leq K$. By applying $\partial_\beta^\alpha$ with $|\beta|=m$ and
$|\alpha|+|\beta|\leq K$ to (\ref{micro two VPBeqn}), taking the inner product with $e^{\pm\phi}w_l^2(\alpha,\beta)\partial_\beta^\alpha({\bf I}_{\pm}-{\bf P}_{\pm})f$ over ${\bf R}^3\times{\bf R}^3$,
one obtains
\begin{equation}\label{micro two VPBeqn xvder inner}
\begin{split}
\frac{1}{2}\sum\limits_{\pm}&\frac{d}{dt}\left\|e^{\frac{\pm\phi}{2}}w_l(\alpha,\beta)\partial_\beta^\alpha({\bf I}_{\pm}-{\bf P}_{\pm})f\right\|^2
+\underbrace{\sum\limits_{\pm}\left(\partial_\beta^\alpha L_{\pm}({\bf I}_{\pm}-{\bf P}_{\pm})f,w^2_l(\alpha,\beta)\partial_\beta^\alpha({\bf I}_{\pm}-{\bf P}_{\pm})f\right)}_{J_6}
\\
&\underbrace{-\frac{1}{2}\sum\limits_{\pm}\left(\pm\partial_t\phi \partial_\beta^\alpha({\bf I}_{\pm}-{\bf P}_{\pm})f,e^{\pm\phi}w^2_l(\alpha,\beta)\partial_\beta^\alpha({\bf I}_{\pm}-{\bf P}_{\pm})f\right)}_{J_7}\\
\pagebreak=&\underbrace{\sum\limits_{\pm}\left(\partial_\beta^\alpha L_{\pm}({\bf I}_{\pm}-{\bf P}_{\pm})f,\left(1-e^{\pm\phi}\right)w^2_l(\alpha,\beta)\partial_\beta^\alpha({\bf I}_{\pm}-{\bf P}_{\pm})f\right)}_{J_8}+\sum_{m=9}^{18} J_m,
\end{split}
\end{equation}
with
\begin{eqnarray*}
J_9 & = &
\sum\limits_{\pm}\sum\limits_{|\alpha_1|\leq|\alpha|}C_\alpha^{\alpha_1}
\left(\pm\partial^{\alpha_1+e_i}\phi
\partial_{\beta+e_i}^{\alpha-\alpha_1} ({\bf I}_{\pm}-{\bf
P}_{\pm})f,e^{\pm\phi}w^2_l(\alpha,\beta)
\partial_\beta^\alpha({\bf I}_{\pm}-{\bf P}_{\pm})f\right), \\
J_{10} & = & \sum\limits_{\pm}\sum\limits_{|\alpha_1|\leq|\alpha|}
C_\alpha^{\alpha_1}\left(\pm\partial^{\alpha_1+e_i}\phi
\partial_{\beta+e_i}^{\alpha-\alpha_1}{\bf P}_{\pm}f,e^{\pm\phi}w^2_l(\alpha,\beta)\partial_\beta^\alpha({\bf I}_{\pm}-{\bf P}_{\pm})f\right), \\
J_{11} &=& -\sum\limits_{\pm}\left(\pm\frac{1}{2}(\partial^{e_i}\phi v_i)\partial_\beta^\alpha{\bf P}_{\pm}f,e^{\pm\phi}w^2_l(\alpha,\beta)
\partial_\beta^\alpha({\bf I}_{\pm}-{\bf P}_{\pm})f\right),\\
J_{12} &=&
-\sum\limits_{\pm}\sum\limits_{|\alpha_1|+|\beta_1|\geq1\atop{|\beta_1|\leq1}}
C_{\alpha,\beta}^{\alpha_1,\beta_1}\left(\pm\frac{1}{2}(\partial^{\alpha_1+e_i}\phi
\partial_{\beta_1}v_i)\partial_{\beta-\beta_1}^{\alpha-\alpha_1}({\bf
I}_{\pm}-{\bf P}_{\pm})f,e^{\pm\phi}w^2_l(\alpha,\beta)
\partial_\beta^\alpha({\bf I}_{\pm}-{\bf P}_{\pm})f\right),
\end{eqnarray*}
and
\begin{eqnarray*}
J_{13}& = &
-\sum\limits_{\pm}\sum\limits_{|\alpha_1|+|\beta_1|\geq1\atop{|\beta_1|\leq1}}
C_{\alpha,\beta}^{\alpha_1,\beta_1}\left(\pm\frac{1}{2}(\partial^{\alpha_1+e_i}\phi
\partial_{\beta_1}v_i)\partial_{\beta-\beta_1}^{\alpha-\alpha_1}{\bf
P}_{\pm}f,e^{\pm\phi}w^2_l(\alpha,\beta)
\partial_\beta^\alpha({\bf I}_{\pm}-{\bf P}_{\pm})f\right), \\
J_{14} & = &
-\sum\limits_{\pm}\sum\limits_{|\alpha_1|\leq|\alpha|}C_\alpha^{\alpha_1}\left(\pm(\partial^{\alpha_1+e_i}\phi
v_i)
\sqrt{\mu},e^{\pm\phi}w^2_l(\alpha,\beta)\partial_\beta^\alpha({\bf I}_{\pm}-{\bf P}_{\pm})f\right), \\
J_{15} &=&-\sum\limits_{\pm}\left(\left[\partial_{t}+v_i\partial^{e_i}\right]\partial_\beta^\alpha{\bf P}_{\pm}f,e^{\pm\phi}w^2_l(\alpha,\beta)\partial_\beta^\alpha({\bf I}_{\pm}-{\bf P}_{\pm})f\right),
\end{eqnarray*}
and
\begin{eqnarray*}
J_{16} & = &  -\sum\limits_{\pm}{C_{\beta}^{e_i}}\left(\partial_{\beta-e_i}^{\alpha+e_i} ({\bf I}_{\pm}-{\bf P}_{\pm})f,e^{\pm\phi}w^2_l(\alpha,\beta)\partial_\beta^\alpha({\bf I}_{\pm}-{\bf P}_{\pm})f\right), \\
J_{17} & = & -\sum\limits_{\pm}{C_{\beta}^{e_i}}\left(\partial_{\beta-e_i}^{\alpha+e_i} {\bf P}_{\pm}f,e^{\pm\phi}w^2_l(\alpha,\beta)\partial_\beta^\alpha({\bf I}_{\pm}-{\bf P}_{\pm})f\right), \\
J_{18} &=&
\sum\limits_{\pm}\left(\partial_\beta^\alpha\Gamma_{\pm}(f,f),e^{\pm\phi}
w^2_l(\alpha,\beta)\partial_\beta^\alpha({\bf I}_{\pm}-{\bf
P}_{\pm})f\right),
\end{eqnarray*}
where as before we have used
$$
\left(v_i\partial_\beta^{\alpha+e_i}({\bf I}_{\pm}-{\bf P}_{\pm})f\pm\frac{1}{2}(\partial^{e_i}\phi v_i)\partial_\beta^\alpha({\bf I}_{\pm}-{\bf P}_{\pm})f,e^{\pm\phi}w^2_l(\alpha,\beta)
\partial_\beta^\alpha({\bf I}_{\pm}-{\bf P}_{\pm})f\right)=0.
$$
Now we turn to estimate $J_{i}$ $(6\leq i\leq 18)$ term by term.
Lemma \ref{estimate on L} yields
$$
J_6\gtrsim \left\|w_l(\alpha,\beta)\partial_\beta^\alpha({\bf
I}-{\bf P})f\right\|^2_{N^{s}_{\gamma}}
-\eta\sum\limits_{\beta_1\leq
\beta}\left\|w^{l}\partial^{\alpha}_{\beta_1}({\bf I}-{\bf
P})f\right\|^2_{N^{s}_{\gamma}}
-C_{\eta}\left\|\partial^{\alpha}({\bf I}-{\bf
P})f\right\|^2_{L^2(B_C)}.
$$
It is straightforward to see that $J_7$ is bounded by
$$
C\|\partial_t\phi\|_{L^\infty}\left\|w_l(\alpha,\beta)\partial_\beta^\alpha({\bf I}-{\bf P})f\right\|^2.
$$
From Lemma \ref{lem.non.ga}, $J_8$ and $J_{18}$ have the same upper bound $\sqrt{\mathcal {E}_l(t)}\mathcal {D}_l(t)$, and so do $J_9$ and $J_{12}$ by Lemma \ref{lem.phi.xv} and Lemma \ref{lem.vphi.xv}. Moreover,
since there is the exponential decay in $v$ in the terms $J_{10}$, $J_{11}$ and $J_{13}$, one can verify that they all have the upper bound $ \sqrt{\mathcal {E}_l(t)}\mathcal {D}_l(t)$.

Next, by Cauchy-Schwartz's inequality with $\eta$ and local
conservation laws {\eqref{conservation abc} and \eqref{a-a}}, we
obtain
$$
J_{14}\lesssim C\sum\limits_{|\alpha|\leq
K-1}\|\partial^\alpha\nabla_x\phi\|^2+\eta\left\|w_l(\alpha,\beta)\partial_\beta^\alpha({\bf
I}-{\bf P})f\right\|^2_{N^{s}_{\gamma}},
$$
\begin{equation*}
\begin{split}
J_{15}\lesssim&\eta\left\|w_l(\alpha,\beta)\partial_\beta^\alpha({\bf I}-{\bf P})f\right\|^2_{N^{s}_{\gamma}}\\ &+C\left\{\sum\limits_{|\alpha|\leq K-1}\|\partial^\alpha\nabla_x\phi\|^2+\sum\limits_{|\alpha|\leq K-1}\|\partial^\alpha\nabla_x(a_{\pm},b,c)\|^2
+\sum\limits_{|\alpha|\leq K}\|\partial^\alpha({\bf I}-{\bf
P})f\|^2_{N^{s}_{\gamma}}+\mathcal {E}_l(t)\mathcal {D}_l(t)\right\},
\end{split}
\end{equation*}
$$
J_{17}\lesssim C\sum\limits_{|\alpha|\leq
K-1}\|\partial^\alpha\nabla_x(a_{\pm},b,c)\|^2+\eta\left\|w_l(\alpha,\beta)\partial_\beta^\alpha({\bf
I}-{\bf P})f\right\|^2_{N^{s}_{\gamma}}.
$$
For the remaining term $J_{16}$, since $|\beta|\geq1$, by writing
\begin{equation*}
\begin{split}
&\left(\partial_{\beta-e_i}^{\alpha+e_i} ({\bf I}_{\pm}-{\bf P}_{\pm})f,e^{\pm\phi}w^2_l(\alpha,\beta)\partial_\beta^\alpha({\bf I}_{\pm}-{\bf P}_{\pm})f\right)\\
=&\left(\partial_{\beta-e_i}^{\alpha+e_i} ({\bf I}_{\pm}-{\bf P}_{\pm})f,e^{\pm\phi}w^2_l(\alpha,\beta)\partial_{e_i}\partial_{\beta-e_i}^\alpha({\bf I}_{\pm}-{\bf P}_{\pm})f\right),
\end{split}
\end{equation*}
and then doing the similar calculations as $J_{4,2}$, one obtains
$$
J_{16}\lesssim {\sum\limits_{|\alpha'|+|\beta'|\leq
K\atop{|\beta'|=|\beta|-1}}
\left\|w_l(\alpha',\beta')\partial_{\beta'}^{\alpha'}({\bf I}-{\bf
P})f\right\|^2_{N^{s}_{\gamma}}}.
$$
Therefore, by plugging all the estimates above into (\ref{micro two VPBeqn xvder inner}), taking the summation
over $\{|\beta|=m,|\alpha|+|\beta|\leq K\}$ for each given $1\leq m\leq K$, and then taking the proper linear combination
of those $K$ estimates with properly chosen constants $C_m>0$ $(1\leq m\leq K)$ and $\eta$ small enough, one has
\begin{equation}\label{micro two VPBeqn xvder inner total}
\begin{split}
\frac{d}{dt}\sum\limits_{m=1}^K&C_m\sum\limits_{|\beta|=m,|\alpha|+|\beta|\leq K}
\sum\limits_{\pm}\left\|e^{\frac{\pm\phi}{2}}w_l(\alpha,\beta)\partial_\beta^\alpha({\bf I}_{\pm}-{\bf P}_{\pm})f\right\|^2\\
&+\lambda\sum\limits_{|\beta|=m,|\alpha|+|\beta|\leq K}\left\|w_l(\alpha,\beta)\partial_\beta^\alpha({\bf I}-{\bf P})f\right\|^2_{N^{s}_{\gamma}}\\
\lesssim&|\partial_t\phi|\sum\limits_{|\beta|=m,|\alpha|+|\beta|\leq K}\left\|w_l(\alpha,\beta)\partial_\beta^\alpha({\bf I}-{\bf P})f\right\|^2+C\sum\limits_{|\alpha|\leq K}\left\|w_l(\alpha,0)\partial^{\alpha}({\bf I}-{\bf P})f\right\|^2_{N^{s}_{\gamma}}\\
&+C\sum\limits_{|\alpha|\leq K-1}\|\partial^\alpha\nabla_x(a_{\pm},b,c)\|^2
+C\sum\limits_{|\alpha|\leq K-1}\|\partial^\alpha\nabla_x\phi\|^2
+C\left\{\sqrt{\mathcal {E}_l(t)}+\mathcal {E}_l(t)\right\}\mathcal {D}_l(t).
\end{split}
\end{equation}

\medskip
\noindent{\bf Step 3.}  We are in a position to prove (\ref{energy liyine}) by taking the proper linear combination of those estimates obtained in the previous two steps as follows. The combination
$C_2\times [C_1\times(\ref{total spatial energy})+(\ref{micro zeao energy})+(\ref{spatial inner withw total})]+
(\ref{micro two VPBeqn xvder inner total})$ for $C_1>0$ and $C_2>0$ large enough gives
\begin{equation}\label{lem.energy.pf}
\frac{d}{dt}\mathcal {E}_l(t)+\lambda \mathcal {D}_l(t)\lesssim \|\partial_t\phi\|_{L^\infty}\mathcal {E}_l(t)+\left\{\sqrt{\mathcal {E}_l(t)}+\mathcal {E}_l(t)\right\}\mathcal {D}_l(t),
\end{equation}
where $\mathcal {E}_l(t)$ is given by
\begin{equation*}
\begin{split}
\mathcal {E}_l(t)=& C_2\Bigg[C_1\Bigg\{\sum\limits_{|\alpha|\leq K}\sum\limits_{\pm}\left\|e^{\frac{\pm\phi}{2}}\partial^\alpha f_{\pm}\right\|^2
+\sum\limits_{|\alpha|\leq K}\|\partial^\alpha\nabla_x\phi\|^2+\kappa \mathscr{E}_{int}^{K}(t)\Bigg\}\\ &+
\left\|e^{\frac{\pm\phi}{2}}w_l(0,0)({\bf I}_{\pm}-{\bf P}_{\pm})f\right\|^2+\sum\limits_{1\leq|\alpha|\leq K}\left\|e^{\frac{\pm\phi}{2}}w_l(\alpha,0)\partial^\alpha f_{\pm}\right\|^2\\ &+
\sum\limits_{m=1}^KC_m\sum\limits_{|\beta|=m,|\alpha|+|\beta|\leq K}
\sum\limits_{\pm}\left\|e^{\frac{\pm\phi}{2}}w_l(\alpha,\beta)\partial_\beta^\alpha({\bf I}_{\pm}-{\bf P}_{\pm})f\right\|^2\Bigg].
\end{split}
\end{equation*}
Noticing (\ref{int fun3}), it is easy to see
$$
\mathcal {E}_l(t)\sim \sum\limits_{|\alpha|\leq
K}\|\partial^{\alpha}\nabla_x\phi\|^2+\sum\limits_{|\alpha|\leq
K}\|\partial^{\alpha}{\bf P}f\|^2+\sum\limits_{|\alpha|+|\beta|\leq
K}\left\|w_l(\alpha,\beta)\partial_{\beta}^{\alpha}({\bf I-P})f\right\|^{2}.
$$
Recalling the a priori assumption (\ref{a priori estimates}), the desired estimate \eqref{energy liyine} follows directly by \eqref{lem.energy.pf}.
This completes the proof of Lemma \ref{lem.en.l}.
\end{proof}

%Moreover, for the high-order instant energy functional (\ref{def he}), we have

\begin{Lemma}\label{lem.ener.hf}
%Assume $-3<\gamma<-2s$. Let $K\geq8$, $l\geq0$. Suppose that a priori estimate (\ref{a priori estimates})
%holds for $M_0$ small enough.
For any $l$ with $0\leq l\leq l_0$, there is $\mathcal {E}^h_l(t)$ satisfying (\ref{def he}) such that
\begin{equation}\label{high energy liyine}
\frac{d}{dt}\mathcal {E}^h_l(t)+\lambda \mathcal {D}_l(t)\lesssim \|\nabla_x(a_{\pm},b,c)(t)\|^2+\|\partial_t\phi\|_{L^\infty}\mathcal {E}^h_l(t)
\end{equation}
holds for any $0\leq t\leq T$, where $\mathcal {D}_l(t)$ is defined by (\ref{def d}).
\end{Lemma}

\begin{proof}
By letting $|\alpha|\geq1$ in (\ref{spatial inner withoutw}),
repeating those computations in (\ref{spatial inner
withoutw})-(\ref{hL2}) and combing (\ref{mic-mac3}), one can instead
obtain
\begin{equation}\label{total high spatial energy}
\begin{split}
\frac{d}{dt}&\left\{\sum\limits_{1\leq|\alpha|\leq K}\sum\limits_{\pm}\left\|e^{\frac{\pm\phi}{2}}\partial^\alpha f_{\pm}\right\|^2
+\sum\limits_{1\leq|\alpha|\leq K}\|\partial^\alpha\nabla_x\phi\|^2+\kappa \mathscr{E}_{int}^{K,h}(t)
\right\}\\ &+\lambda\sum\limits_{1\leq|\alpha|\leq K-1}\|\partial^{\alpha}\nabla
(a_{\pm},b,c)\|^2+\lambda\|\nabla_x(a_+-a_-)\|^2\\ &+\lambda\|\nabla^2_x\phi\|^2+\lambda\|\nabla_x\phi\|^2
+\lambda\sum\limits_{1\leq|\alpha|\leq K}\left\|({\bf I-P})\partial^{\alpha}f\right\|^2_{N^{s}_{\gamma}}\\
\lesssim& \sum\limits_{1\leq|\alpha|\leq K}\sum\limits_{\pm}
\|\partial_t\phi\|_{L^\infty}\left\|\partial^\alpha
f_{\pm}\right\|^2 +{\left\|({\bf
I-P})f\right\|^2_{N^{s}_{\gamma}}}+\left(\sqrt{\mathcal
{E}_l(t)}+\mathcal {E}_l(t)\right)\mathcal {D}_l(t).
\end{split}
\end{equation}
Moreover, taking the inner product of (\ref{micro two VPBeqn}) with $e^{\pm\phi}({\bf I}_{\pm}-{\bf P}_{\pm})f$
over ${\bf R}^3\times {\bf R}^3$ gives
\begin{equation}\label{zero micro diss}
\begin{split}
\frac{1}{2}\sum\limits_{\pm}&\frac{d}{dt}\left\|e^{\frac{\pm\phi}{2}}({\bf I}_{\pm}-{\bf P}_{\pm})f\right\|^2
+\sum\limits_{\pm}\left(L_{\pm}({\bf I}_{\pm}-{\bf P}_{\pm})f,({\bf I}_{\pm}-{\bf P}_{\pm})f\right)
\\ &+\sum\limits_{\pm}\left(\pm(\partial^{e_i}\phi v_i)
\sqrt{\mu},({\bf I}_{\pm}-{\bf P}_{\pm})f\right)
-\frac{1}{2}\sum\limits_{\pm}\left(\pm\partial_t\phi ({\bf I}_{\pm}-{\bf P}_{\pm})f,e^{\pm\phi}
({\bf I}_{\pm}-{\bf P}_{\pm})f\right)\\
=&\sum\limits_{\pm}\left(L_{\pm}({\bf I}_{\pm}-{\bf P}_{\pm})f,
\left(1-e^{\pm\phi}\right)({\bf I}_{\pm}-{\bf P}_{\pm})f\right)
-\sum\limits_{\pm}\left(\frac{1}{2}\pm(\partial^{e_i}\phi v_i){\bf
P}_{\pm}f,e^{\pm\phi}({\bf I}_{\pm}-{\bf
P}_{\pm})f\right)\\
&-\sum\limits_{\pm}\left(\pm(\partial^{e_i}\phi v_i)
\sqrt{\mu},\left(e^{\pm\phi}-1\right)({\bf I}_{\pm}-{\bf
P}_{\pm})f\right) -\sum\limits_{\pm}\left(\partial_{t}{\bf
P}_{\pm}f,e^{\pm\phi}({\bf I}_{\pm}-{\bf P}_{\pm})f\right)
\\
&-\sum\limits_{\pm}\left(v_i\partial^{e_i}{\bf
P}_{\pm}f,e^{\pm\phi}({\bf I}_{\pm}-{\bf P}_{\pm})f\right)
-\sum\limits_{\pm}\left(\mp\partial^{e_i}\phi\partial_{e_i}{\bf
P}_{\pm}f,e^{\pm\phi}({\bf I}_{\pm}-{\bf P}_{\pm})f\right)
\\
&+\sum\limits_{\pm}\left(\Gamma_{\pm}(f,f),e^{\pm\phi}({\bf
I}_{\pm}-{\bf P}_{\pm})f\right).
\end{split}
\end{equation}
Now we turn to further estimate (\ref{zero micro diss})
term by term. In light of (\ref{a-a}), (\ref{G}) and (\ref{ea}),
$$
\sum\limits_{\pm}\left(\pm(\partial^{e_i}\phi v_i)
\sqrt{\mu},({\bf I}_{\pm}-{\bf P}_{\pm})f\right)=-(\phi,\nabla_xG)=-(\phi,\partial_t\Delta\phi)=\frac{d}{dt}\|\nabla_x\phi\|^2.
$$
It is also straightforward to get from Lemma \ref{estimate on L} that
$$
\sum\limits_{\pm}\left(L_{\pm}({\bf I}_{\pm}-{\bf P}_{\pm})f,({\bf
I}_{\pm}-{\bf P}_{\pm})f\right)\geq \lambda\left\|({\bf I}-{\bf
P})f\right\|^2_{N^{s}_{\gamma}}.
$$
The first term on the right-hand side of (\ref{zero micro
diss}) can be bounded by $\sqrt{\mathcal {E}_l(t)}\mathcal {D}_l(t)$
according to the estimate (\ref{Lup1}). The second, third, sixth and
seventh terms on the right-hand side of (\ref{zero micro diss}) are
also dominated by $\sqrt{\mathcal {E}_l(t)}\mathcal {D}_l(t)$, where
we have used Sobolev's inequality for the second, third and sixth
terms and Lemma \ref{estimates on nonop} for the seventh term. For
the fourth term, by the definitions of ${\bf P}$ and ${\bf I-P}$, 
$$
\sum\limits_{\pm}\left(\partial_{t}{\bf P}_{\pm}f,e^{\pm\phi}({\bf
I}_{\pm}-{\bf
P}_{\pm})f\right)=\sum\limits_{\pm}\left(\partial_{t}{\bf
P}_{\pm}f,\left[e^{\pm\phi}-1\right]({\bf I}_{\pm}-{\bf
P}_{\pm})f\right).
$$
Recalling (\ref{epowerforce}) and the local conservation laws
(\ref{conservation abc}), (\ref{a-a}), and applying Sobolev's
inequality as well as Cauchy-Schwarz's inequality with $\lambda$, it follows that
$$
\begin{array}{rlll}
\begin{split}
\left|\sum\limits_{\pm}\left(\partial_{t}{\bf
P}_{\pm}f,e^{\pm\phi}({\bf I}_{\pm}-{\bf
P}_{\pm})f\right)\right|\lesssim&\frac{\lambda}{4}\left\|({\bf
I}-{\bf P})f\right\|^2_{N^{s}_{\gamma}}+
\left[\|\nabla_x(a_{\pm},b,c)\|^2+\|\nabla_x\phi\|^2\right]
\|\nabla_x\phi\|^2_{H^1}\\[2mm]&+\|\nabla_x\phi\|^2_{H^1}\left\|\nabla_x({\bf
I}-{\bf P})f\right\|^2_{N^{s}_{\gamma}}+\mathcal {E}_l(t)\mathcal
{D}_l(t)\\[2mm]
\lesssim&\frac{\lambda}{4}\left\|({\bf
I}-{\bf P})f\right\|^2_{N^{s}_{\gamma}}+\mathcal {E}_l(t)\mathcal
{D}_l(t).
\end{split}
\end{array}
$$
For the remaining fifth term, by Cauchy-Schwarz's inequality
with $\lambda$, 
$$
\left|\sum\limits_{\pm}\left(v_i\partial^{e_i}{\bf
P}_{\pm}f,e^{\pm\phi}({\bf I}_{\pm}-{\bf
P}_{\pm})f\right)\right|\lesssim \frac{\lambda}{4}\left\|({\bf
I}-{\bf P})f\right\|^2_{N^{s}_{\gamma}}+\|\nabla_x(a_{\pm},b,c)\|^2.
$$
Therefore, combing all the estimates above, one has
\begin{equation}\label{micro zeao energy withoutw}
\begin{split}
\frac{d}{dt}&\left\{\sum\limits_{\pm}\left\|e^{\frac{\pm\phi}{2}}({\bf I}_{\pm}-{\bf P}_{\pm})f\right\|^2+\left\|\nabla_x\phi\right\|^2\right\}
+\lambda\left\|({\bf I}-{\bf P})f\right\|^2_{N^{s}_{\gamma}}\\
\lesssim& \|\partial_t\phi\|_{L^\infty}\sum\limits_{\pm}
\left\|e^{\frac{\pm\phi}{2}}({\bf I}_{\pm}-{\bf P}_{\pm})f\right\|^2
{+\|\nabla_x(a_{\pm},b,c)\|^2}+\left\{\sqrt{\mathcal
{E}_l(t)}+\mathcal {E}_l(t)\right\}\mathcal {D}_l(t).
\end{split}
\end{equation}
As in step 3 in the proof of Lemma \ref{lem.en.l}, a suitably linear combination of (\ref{total high spatial energy}),
(\ref{micro zeao energy withoutw}), (\ref{micro zeao energy}), (\ref{spatial inner withw total}) and
(\ref{micro two VPBeqn xvder inner total}) yields
\begin{equation*}
\frac{d}{dt}\mathcal {E}^h_l(t)+\lambda \mathcal {D}_l(t)\lesssim \|\nabla_x(a_{\pm},b,c)\|^2+\|\partial_t\phi\|_{L^\infty}\mathcal {E}^h_l(t)+\left\{\sqrt{\mathcal {E}_l(t)}+\mathcal {E}_l(t)\right\}\mathcal {D}_l(t),
\end{equation*}
which hence implies the desired estimate  (\ref{high energy liyine})  under the a priori assumption (\ref{a priori estimates}). This completes the proof of Lemma \ref{lem.ener.hf}.
\end{proof}

\section{Proof of global existence}\label{sec.7}

In this section we prove Theorem \ref{thm.gl} through obtaining the closed uniform-in-time estimates on $X(t)$ in the following lemma. Recall \eqref{def.X} for the definition of $X(t)$ with \eqref{para1}-\eqref{para2} for the choice of the corresponding parameters.

%In this section, we turn to close the energy estimates under the a priori estimates (\ref{a %priori estimates}), for this, we have to obtain the time-decay of $\mathcal {E}_l(t)$ and $\|
%\partial_t\phi\|_{L^\infty}$. The key is to prove the following

\begin{Lemma}\label{lem.X}
Consider the Cauchy problem \eqref{perturbed eqn}-\eqref{perturbed data}.  Assume that
\begin{equation*}
\int_{{\bf R}^3}\rho_{f_0}\,dx=0,\ \ \int_{{\bf R}^3}(1+|x|)|\rho_{f_0}|\,dx<\infty,
\end{equation*}
where
$\rho_{f_0}=\int_{{\bf R}^3}(f_{0,+}-f_{0,-})\sqrt{\mu}\,dv$.
Define $\eps_0$ by
\begin{equation*}
\epsilon_{0}=\sqrt{\mathcal{E}_{l_0+l_1}(0)}+\left\|w^{l_2}f_0\right\|_{Z_1}+\|(1+|x|)\rho_{f_0}\|_{L^1},
\end{equation*}
where $l_2>\frac{5(\ga+2s)}{4\ga}$ is a constant. Then, under the a priori assumption (\ref{a priori estimates}) for $\de_0>0$ suitably small,  one has
\begin{equation}\label{close energy}
X(t)\lesssim \epsilon^2_0+X^{\frac{3}{2}}(t)+X^2(t),
\end{equation}
for any $0\leq t\leq T$.
\end{Lemma}

\begin{proof}
We divide it by three steps.

\medskip
\noindent{\bf Step 1.} It follows from (\ref{a-a}) and (\ref{ea}) that
\begin{equation*}
\pa_t \phi=-\De_x^{-1} \pa_t (a_+-a_-)=\De_x^{-1}\na_x\cdot G,
\end{equation*}
where $G$ is given by \eqref{G}. Then, by Sobolev and Riesz inequalities,
\begin{equation}
\label{lem.X.p1}
\|\partial_t\phi\|^2_{L^\infty}\lesssim\|\nabla_x\partial_t\phi\|\cdot \|\nabla^2_x\partial_t\phi\|
\lesssim \|G\|\cdot \|\nabla_xG\|\lesssim \CE_{l_0}^{h}(t)\lesssim (1+t)^{-\frac{3}{2}-p}X(t).
\end{equation}
Here notice $3/2+p>2$ since $p>1/2$. By applying Lemma \ref{lem.en.l} in the case when $l=l_0+l_1$, one has from \eqref{energy liyine} that
\begin{equation*}
\frac{d}{dt}\mathcal {E}_{l_0+l_1}(t)+\lambda \mathcal {D}_{l_0+l_1}(t)\lesssim \|\partial_t\phi\|_{L^\infty}\mathcal {E}_{l_0+l_1}(t),
\end{equation*}
which together with \eqref{lem.X.p1} and Gronwall's inequality as well as \eqref{a priori estimates}, imply
\begin{equation}
\label{lem.X.p2}
\dis\mathcal {E}_{l_0+l_1}(t)\lesssim \mathcal {E}_{l_0+l_1}(0) e^{C\int_0^t \|\partial_t\phi(\tau)\|_{L^\infty}\,d\tau}\lesssim \mathcal {E}_{l_0+l_1}(0) \lesssim \eps_0^2
\end{equation}
for any $0\leq t\leq T$.

\medskip
\noindent{\bf Step 2.} We begin with \eqref{high energy liyine} with $l=l_0$ in Lemma \ref{lem.ener.hf}, i.e.
\begin{equation}
\label{lem.X.p3}
\frac{d}{dt}\mathcal {E}^h_{l_0}(t)+\lambda \mathcal {D}_{l_0}(t)\lesssim \|\nabla_x(a_{\pm},b,c)(t)\|^2+\|\partial_t\phi\|_{L^\infty}\mathcal {E}^h_{l_0}(t).
\end{equation}
As in \cite{SG1, SG2},  at time $t$, we split the velocity space ${\bf R}_v^3$ into
\begin{equation}\label{Esplit}
 {\mathbf{E}(t)}=\left\{\langle v\rangle^{-\gamma-2s}\leq t^{1-p}\right\},\ \  {\mathbf{E}^c(t)}=\left\{\langle v\rangle^{-\gamma-2s}> t^{1-p}\right\},
\end{equation}
where recall $\ga+2s<0$ and $1/2<p<1$. Corresponding to this splitting, we define
$\mathcal {E}_{l_0}^{h,low}(t)$ to be the restriction of $\mathcal {E}^h_{l_0}(t)$ to  {$\mathbf{E}(t)$} and
similarly $\mathcal {E}_{l_0}^{h,high}(t)$ to be the restriction of $\mathcal {E}^h_{l_0}(t)$ to
 {$\mathbf{E}^c(t)$}. Then, due to (\ref{Esplit}) and (\ref{def he}),
%one has
%\begin{equation}\label{HlowD}
%\frac{\mathcal {E}_l^{h,low}(t)}{t^{p'}}\lesssim\widetilde{\mathcal{D}}_l(t).
%\end{equation} Now
it follows from (\ref{lem.X.p3})  that
\begin{equation*}
\frac{d}{dt}\mathcal {E}^h_{l_0}(t)+\lambda t^{p-1}\mathcal {E}^h_{l_0}(t)\lesssim\|\nabla_x{\bf P} f\|^2+\|\partial_t\phi\|_{L^\infty}\mathcal {E}^h_{l_0}(t)
+t^{p-1}\mathcal {E}_{l_0}^{h,high}(t),
\end{equation*}
which by solving the ODE inequality, gives
\begin{equation}\label{Ehhlinte}
\mathcal {E}^h_{l_0}(t)\lesssim e^{-\lambda t^p}\mathcal {E}^h_{l_0}(0)+\int_0^td\tau\, e^{-\lambda(t^p-\tau^p)}\left(\|\nabla_x{\bf P} f(\tau)\|^2
+\|\partial_t\phi\|_{L^\infty}\mathcal {E}^h_{l_0}(\tau)+\tau^{p-1}\mathcal {E}_{l_0}^{h,high}(\tau)\right).
\end{equation}
In what follows we estimate those three terms in the time integral on the right-hand side of \eqref{Ehhlinte}. First of all, by \eqref{lem.X.p1},
\begin{equation}
\label{lem.X.p4}
\|\partial_t\phi\|_{L^\infty}\mathcal {E}^h_{l_0}(t)\lesssim [\mathcal {E}^h_{l_0}(t)]^{\frac{3}{2}}\lesssim (1+t)^{-(\frac{3}{2}+p)\frac{3}{2}}X^{\frac{3}{2}} (t).
\end{equation}
Next,
one can prove
\begin{equation}\label{Eh decay1}
\mathcal {E}_{l_0}^{h,high}(t)\lesssim\epsilon^2_{0}(1+t)^{-5/2}.
\end{equation}
In fact, it is straightforward to see from \eqref{lem.X.p2} that for any $0\leq t\leq T$,
\begin{equation*}
\mathcal {E}_{l_0}^{h,high}(t)\lesssim\mathcal {E}_{l_0}^{h}(t)\lesssim\mathcal {E}_{l_0}(t))\lesssim\mathcal {E}_{l_0+l_1}(t)\lesssim \eps_0^2.
\end{equation*}
On the other hand, noticing that $t^{5/2}\leq w^{\frac{5}{2(1-p)}\frac{\gamma+2s}{\gamma}}$ on the set $E^c(t)$ and  $l_1$ is given by
$$
l_1=\frac{5}{4(1-p)}\frac{\gamma+2s}{\gamma},
$$
one has from \eqref{lem.X.p2} that
\begin{equation*}
\mathcal {E}_{l_0}^{h,high}(t)\lesssim t^{-5/2}\mathcal{E}_{l_0+l_1}(t)\lesssim\epsilon^2_{0}t^{-5/2}.
\end{equation*}
Therefore, \eqref{Eh decay1} holds true. Finally,
one also can prove
\begin{equation}
\label{lem.X.p5}
\begin{split}
\|\nabla_x{\bf P}f(t)\|+\|\nabla_x^2\phi\|\lesssim (1+t)^{-5/4}\left\{\epsilon_{0}+X(t)\right\}.
\end{split}
\end{equation}
In fact, recalling (\ref{formal sol}),  the solution $f(t,x,v)$ to
\eqref{perturbed eqn} can be written as
\begin{equation*}
f(t)=S(t)f_0+\int_{0}^{t}S(t-\tau)h(\tau)\,d\tau,
\end{equation*}
with
\begin{equation}\label{h expr}
h=q\nabla_x\phi\cdot\nabla_v f-\frac{q}{2}v\cdot \nabla_x\phi f+\Gamma(f,f).
\end{equation}
Notice $\lag h_\pm,\sqrt{\mu}\rag =0$.
Applying Theorem \ref{basic decay lemma} in the case when $k=0$ and $\si_1=5/4$, it follows that
$$
\begin{array}{rlll}
\begin{split}
\|\nabla_x{\bf P}f(t)\|&+\|\nabla_x^2\phi\|\lesssim\|\nabla_xf(t)\|+\|\nabla_x^2\phi\|\\ \lesssim&
\epsilon_{0}(1+t)^{-5/4}+\int_0^t(1+t-\tau)^{-5/4}
\left(\left\|w^{l_\ast}h(\tau)\right\|_{Z_1}+\left\|w^{l_\ast}\nabla_xh(\tau)\right\|\right)d\tau,
\end{split}
\end{array}
$$
where in the homogeneous part, we choose $l_\ast=l_2$ with $l_2>\frac{\si_1(\ga+2s)}{\ga}=\frac{5(\ga+2s)}{4\ga}$,  while in the nonhomogeneous part, for later use, we choose $l_\ast$ such that
\begin{equation}
\label{def.l.ast}
 \frac{5(\ga+2s)}{4\ga}<l_\ast \leq K-5+\frac{\gamma+2s}{2\gamma}.
\end{equation}
In fact, $K\geq 8$, $-3<\ga<-2s$ and $1/2\leq s<1$ imply that
\begin{equation*}
\frac{\ga+2s}{\ga}=\left|\frac{\ga+2s}{\ga}\right|\leq \frac{3-2s}{2s}\in [1/2, 2],
\end{equation*}
and hence
\begin{equation*}
K>5+\frac{3(\ga+2s)}{4\ga},\ \text{i.e.}\  \frac{5(\ga+2s)}{4\ga}<K-5+\frac{\gamma+2s}{2\gamma},
\end{equation*}
so that it is valid to choose $l_\ast$ as in \eqref{def.l.ast}.
By using the direct calculations on the collision operator, cf.~Lemma \ref{lem.non.z1},
%$$\|\Gamma(f,g)\|_{Z_1}\leq C\|f\|\|g\| +C\|g\|\|f\|, $$
one can deduce that
\begin{eqnarray*}
&\dis \left\|w^{l_\ast}\Gamma(f,f)\right\|_{Z_1}
\lesssim\mathcal{E}_{l_0}(t),\\
&\dis \left\|w^{l_\ast }q\nabla_x\phi\cdot\nabla_v f\right\|_{Z_1}+\left\|w^{l_\ast}\frac{q}{2}v\cdot \nabla_x\phi f\right\|_{Z_1}\lesssim\mathcal{E}_{l_0}(t).
\end{eqnarray*}
where we have used
$$
\sum\limits_{|\alpha|+|\beta|\leq 5}\left\|w^{l_*-\frac{\gamma+2s}{2\gamma}}\partial_\beta^\alpha f\right\|\lesssim
\sum\limits_{|\alpha|+|\beta|\leq 5}\left\|w^{l_0+K-5}\partial_\beta^\alpha f\right\|,
$$
which can be guaranteed by $l_0+K-5\geq l_*-\frac{\gamma+2s}{2\gamma}$ due to $l_0\geq0$ and the choice of $l_\ast$ as before.
Then,
$$
\left\|w^{l_\ast}h(s)\right\|_{Z_1}\lesssim\mathcal{E}_{l_0}(t).
$$
Furthermore, in a similar way, one has
$$\left\|w^{l_\ast}q\nabla_x(\nabla_x\phi\cdot\nabla_vf)\right\|
+\left\|w^{l_\ast}\frac{q}{2}v\cdot \nabla_x\left(\nabla_x\phi f\right)\right\|\lesssim\mathcal{E}_{l_0}(t),$$
$$\left\|w^{l_\ast}\nabla_x\Gamma(f,f)\right\| \lesssim\mathcal{E}_{l_0}(t),
$$
which imply
 $$
 \left\|w^{l_\ast}h(t)\right\|_{Z_1}+\left\|w^{l_\ast}\nabla_xh(t)\right\|\lesssim\mathcal{E}_{l_0}(t).
 $$
With the above estimates and then using $\mathcal{E}_{l_0}(t)\leq (1+t)^{-3/2}X(t)$, one has
\begin{equation}\label{high decay}
\begin{split}
\|\nabla_x{\bf P}f(t)\|+\|\nabla_x^2\phi\|\lesssim&
\epsilon_{0}(1+t)^{-5/4}+X(t)\int_0^t(1+t-\tau)^{-5/4}(1+\tau)^{-3/2}d\tau
\\ \lesssim& (1+t)^{-5/4}\left(\epsilon_{0}+X(t)\right),
\end{split}
\end{equation}
which is the desired estimate \eqref{lem.X.p5}.

Notice that one has the following inequalities
\begin{equation*}
\int_0^t e^{-\lambda(t^p-\tau^p)}(1+\tau)^{-(\frac{3}{2}+p)\frac{3}{2}}\,d\tau\lesssim (1+t)^{-(\frac{5}{4}+\frac{5p}{2})},
\end{equation*}
\begin{equation*}
\int_0^t e^{-\lambda(t^p-\tau^p)}\tau^{p-1}(1+\tau)^{-\frac{5}{2}}\,d\tau\lesssim (1+t)^{-\frac{5}{2}},
\end{equation*}
and
 {\begin{equation}
\label{add-e1}
\int_0^te^{-\lambda(t^p-\tau^p)}(1+\tau)^{-5/2}\,d\tau\lesssim (1+t)^{-\frac{3}{2}-p}.
\end{equation}}
Then, plugging \eqref{lem.X.p4}, \eqref{Eh decay1} and \eqref{lem.X.p5} into \eqref{Ehhlinte} and using $1/2\leq p<1$, one has
\begin{eqnarray}\label{high total deccay}
\mathcal {E}^h_{l_0}(t)\lesssim \left\{\epsilon^2_{0}+X^{\frac{3}{2}} (t)+X^2(t)\right\}(1+t)^{-\frac{3}{2}-p},
\end{eqnarray}
for any $0\leq t\leq T$.

\medskip
\noindent{\bf Step 3.} In the same way to prove \eqref{high decay} basing on \eqref{lem.X.p5} and \eqref{h expr}, one can show
\begin{equation}\label{abc decay}
\|f\|+\|\nabla_x\phi\|\lesssim (1+t)^{-\frac{3}{4}}\left\{\epsilon_{0}+X(t)\right\}.
\end{equation}
Noticing $\mathcal {E}_{l_0}(t)\sim \|{\bf P}f\|^2+\|\nabla_x\phi\|^2+\mathcal {E}^h_{l_0}(t)$, \eqref{abc decay} together with (\ref{high total deccay}) give
\begin{equation}
\label{lem.X.p6}
\mathcal {E}_{l_0}(t)\lesssim (1+t)^{-\frac{3}{2}}\left\{\epsilon^2_{0}+X^{\frac{3}{2}} (t)+X^2(t)\right\}.
\end{equation}
Now, recall \eqref{def.X}. The desired estimate \eqref{close energy} follows from \eqref{lem.X.p2}, \eqref{high total deccay} and \eqref{lem.X.p6}.
This completes the proof of Lemma \ref{lem.X}.
\end{proof}

\begin{proof}[The proof of Theorem \ref{thm.gl}] It is immediate to follow from the a priori estimate \eqref{close energy} that
$
X(t)\lesssim \eps_0^2
$
holds true for any $0\leq t\leq T$, as long as $\eps_0$ is sufficiently small. The rest is to prove the local existence and uniqueness of solutions in  terms of the energy norm $\CE_{l_0+l_1}(t)$ and the non negativity of $F_\pm=\mu+\sqrt{\mu}f$, and the details of the proof are omitted for brevity; see also \cite{Guo5,Strain-Zhu} and \cite{GR}. Therefore, the global existence of solutions follows with the help of the continuity argument, and the estimates \eqref{thm.gl.c1}, \eqref{thm.gl.c2}  and \eqref{thm.gl.c3} hold by the definition of $X(t)$  \eqref{def.X}. This completes the proof of Theorem \ref{thm.gl}.
\end{proof}

 {Finally we give a remark on the possibility of upgrading the rate $(1+t)^{-\frac{3}{2}-p}$ in \eqref{thm.gl.c3} to $(1+t)^{-5/2}$ that corresponds to the case $p=1$. In fact, due to the technique of the paper it seems difficult to achieve such optimal rate for $\CE_{l_0}^{h}(t)$. The main reason is that the energy dissipation rate $\CD_{l_0}(t)$ in \eqref{lem.X.p3} is degenerate at large velocity for soft potentials, and thus, although the first term on the right-hand side of \eqref{lem.X.p3} decays with the optimal rate $(1+t)^{-5/2}$, it is impossible to deduce from  \eqref{lem.X.p3} the optimal time decay of $\CE_{l_0}^{h}(t)$ because of the inequality \eqref{add-e1}. On the other hand it is still possible to obtain the optimal time decay of only those high-order space differentiations in $\CE_{l_0}^{h}(t)$ by applying the linearized estimate \eqref{basic decay} to the nonlinear system. Indeed, define
$$
\widetilde{\CE}_0^{h}(t)=\sum_{1\leq k\leq K_0} \left(\|\na_x^k f(t)\|^2+\|\na_x^k E(t)\|^2\right).
$$
Then, with the help of Duhamel's principle as well as Theorem \ref{basic decay lemma}, similar for obtaining \eqref{high decay}, it is straightforward to show that 
$$
\widetilde{\CE}_0^{h}(t)\lesssim \eps_0^2 (1+t)^{-\frac{5}{2}}
$$
for an appropriate choice of $K_0$ in terms of $K$ large enough.}

\medskip
\noindent {\bf Acknowledgements:} RJD was supported by the General
Research Fund (Project No.~400511) from RGC of Hong Kong, and SQL was
supported by the grant from the National Natural Science Foundation
of China under contract 11101188.  {RJD would thank Professor Seiji Ukai for his continuous encouragement and strong support since they met each other in City University of Hong Kong in 2005, and this article is dedicated to the memory of S. Ukai. SQL would like to express his gratitude for the
hospitality of the Department of Mathematics at The Chinese University of Hong Kong during his visit.
The authors would thank the anonymous referee for the valuable comments on the manuscript. }

\end{document}